\documentclass{amsart}
\usepackage{amssymb}
\usepackage{colortbl}

\newtheorem{thm}{Theorem}[section]
\newtheorem{lem}[thm]{{Lemma}}
\newtheorem{cor}[thm]{{Corollary}}
\newtheorem{rem}[thm]{{Remark}}
\newtheorem{example}[thm]{{Example}}

\newtheorem{prop}[thm]{{Proposition}}

\theoremstyle{definition}
\newtheorem{defn}[thm]{{Definition}}
\newtheorem{ex}[thm]{{\bf Example}}

\numberwithin{equation}{section}

\newcommand{\sbm}[1]{\left[\begin{smallmatrix} #1
                \end{smallmatrix}\right]}

\newcommand{\C}{{\mathbb C}}
\newcommand{\D}{{\mathbb D}}

\newcommand{\Z}{{\mathbb Z}}

\newcommand{\bzeta}{{\overline{\zeta}}}

\newcommand{\lam}{\lambda}
\newcommand{\bev}{\mathbf{ev}}

\newcommand{\mat}[1]{\begin{bmatrix} #1 \end{bmatrix}}

\newcommand{\ov}[1]{{\overline{#1}}}

\newcommand{\inn}[2]{\ensuremath{\langle #1,#2 \rangle}}
\newcommand{\tu}[1]{\textup{#1}}

\newcommand{\Ran}{\operatorname{Ran}}
\newcommand{\Ker}{\operatorname{Ker}}

\newcommand{\kr}{\operatorname{Ker}}

\newcommand{\wtil}{\widetilde}
\newcommand{\half}{\frac{1}{2}}

\newcommand{\wtilF}{\widetilde{F}}

\newcommand{\cC}{{\mathcal C}}
\newcommand{\cE}{{\mathcal E}}\newcommand{\cF}{{\mathcal F}}
\newcommand{\cH}{{\mathcal H}}

\newcommand{\cK}{{\mathcal K}}\newcommand{\cL}{{\mathcal L}}
\newcommand{\cM}{{\mathcal M}}\newcommand{\cN}{{\mathcal N}}
\newcommand{\cO}{{\mathcal O}}

\newcommand{\cS}{{\mathcal S}}
\newcommand{\cU}{{\mathcal U}}
\newcommand{\cX}{{\mathcal X}}
\newcommand{\cY}{{\mathcal Y}}
\newcommand{\ba}{{\mathbf a}}
\newcommand{\bn}{{\bf n}}

\newcommand{\B}{{\mathbb B}}

\newcommand{\bx}{{\mathbf x}}
\newcommand {\bT}{\mathbf T}

\newcommand{\free}{{\mathbb F}^{+}_{d}}

\newcommand{\BC}{{\mathbb C}}
\newcommand{\bJ}{{\mathbf J}}
\newcommand{\bU}{{\mathbf U}}

\newcommand{\fA}{{\mathfrak A}}
\newcommand{\fT}{{\mathfrak T}}

\numberwithin{equation}{section}

\begin{document}
\title[Interpolation in multivariable de Branges-Rovnyak spaces]{Interpolation in multivariable de Branges-Rovnyak spaces}

\author{Joseph A. Ball}
\address{J.A. Ball, Department of Mathematics,
Virginia Tech, Blacksburg, VA 24061-0123, USA}
\email{ball@math.vt.edu}

\author{Vladimir Bolotnikov}
\address{V. Bolotnikov, Department of Mathematics,
William \&  Mary,
Williamsburg VA 23187-8795, USA}
\email{vladi@math.wm.edu}

\author{Sanne ter Horst}
\address{S. ter Horst, Department of Mathematics, Research Focus Area:\
Pure and Applied Analytics, North-West University, Potchefstroom, 2531
South Africa and DSI-NRF Centre of Excellence in Mathematical and Statistical Sciences (CoE-MaSS)}
\email{Sanne.TerHorst@nwu.ac.za}

\thanks{This work is based on the research supported in part by the National Research Foundation of South Africa (Grant Numbers 90670 and 118513) and Simons Foundation (Grant number 524539).}

\begin{abstract}
We study a general metric constrained interpolation problem in a de Branges-Rovnyak space $\cH(K_S)$ associated with a contractive multiplier $S$ between two Fock spaces along with its commutative counterpart, a de Branges-Rovnyak space associated with a Schur multiplier on the Drury-Arveson space of the unit ball of $\mathbb C^n$.
\end{abstract}


\maketitle
\section{Introduction}\label{S:Intro}

The functions analytic on the open unit disk $\mathbb D$ and bounded by one in modulus
({\em Schur-class functions}) can be alternatively characterized as contractive multipliers
of the Hardy space $H^2$:\ a function $S$ belongs to the Schur class $\mathcal S$ if and only if the
multiplication operator $M_S: \; f\mapsto Sf$ is a contraction on $H^2$. The latter means that
any $S\in\mathcal S$ gives rise to a positive kernel
$$
K_S(\lambda,\omega)=\frac{1-S(\lambda)\overline{S(\eta)}}{1-\lambda\overline{\eta}}
$$
and subsequently, to the reproducing kernel Hilbert space $\cH(S):=\mathcal H(K_S)$, the {\em de Branges-Rovnyak space}
associated with $S$. De Branges-Rovnyak spaces play a prominent role in operator model theory and
Hilbert space approaches to $H^\infty$-interpolation. Interpolation theory
in de Branges-Rovnyak spaces themselves has been initiated in \cite{bbt2}.

To start let us recall the
Nevanlinna-Pick problem in this setting:\ {\em given $S_0\in\mathcal S$, $n$ distinct points
$\lambda_1,\ldots,\lambda_k\in\mathbb D$ and target values $x_1,\ldots,x_k\in\C$ find all
\begin{equation}
f\in\mathcal H(S_0) \; \; \mbox{such that}\; \; f(\lambda_k)=x_k\; \; \mbox{for} \; \; k=1,\ldots,n.
\label{np}
\end{equation}}
The problem can be solved as follows. Let
$$
P=\left[\frac{1-S_0(\lambda_i)\overline{S_0(\lambda_j)}}{1-\lambda_i\overline{\lambda}_j}\right]_{i,j=1}^n,\quad
T=\begin{bmatrix}\overline{\lambda}_1 & & 0 \\ &\ddots & \\ 0 & & \overline{\lambda}_n\end{bmatrix},\quad
{\bf x}=\begin{bmatrix}x_1\\ \vdots  \\ x_n\end{bmatrix},
$$
\begin{equation}
E=\begin{bmatrix}1& \ldots & 1\end{bmatrix},
\quad N=\begin{bmatrix}\overline{S_0(\lambda_1)} & \ldots & \overline{S_0(\lambda_n)}\end{bmatrix}.
\label{data}
\end{equation}
Since the kernel $K_{S_0}$ is positive on $\mathbb D\times\mathbb D$, the matrix $P$ is positive semidefinite.
In fact, it is invertible unless $S_0$ is a Blaschke product of degree less than $n$. In the latter case,
the problem \eqref{np} has a solution (which necessarily is unique) if and only if ${\bf x}$ is in the range of $P$.

\smallskip

If $P$ is invertible, we can define the $2\times 2$ matrix-function
$\Theta=\sbm{\theta_{11} & \theta_{12} \\ \theta_{21} & \theta_{22}}$ by the formula
\begin{equation}
\Theta(\lambda)=I_2+(\lambda-1)\begin{bmatrix} E \\ N\end{bmatrix}(I_n-\lambda T)^{-1}P^{-1}(I_n-T^*)^{-1}
\begin{bmatrix} E^* & -N^*\end{bmatrix},
\label{th}
\end{equation}
which is $\sbm{1 & 0 \\ 0 & -1}$-inner in $\mathbb D$. Furthermore, since $S_0\in\cS$, the function
$$
\mathcal E_0=\frac{\theta_{22}S_0-\theta_{12}}{\theta_{11}-\theta_{21}S_0}
$$
belongs to the Schur class. Then all solutions $f$ to the problem
\eqref{np} are parametrized by the formula
\begin{equation}
f=f_0+(\theta_{11}-\theta_{21}S_0)h,\qquad h\in\mathcal H(\mathcal E_0)
\label{f}
\end{equation}
where
\begin{equation}
f_0(\lambda)=(E-S_0(\lambda)N)(I_n-\lambda T)^{-1}P^{-1}{\bf x}
\label{f0}
\end{equation}
and $h$ is a free parameter from the de Branges-Rovnyak space associated with $\mathcal E_0\in\cS$.
Moreover, the representation \eqref{f} turns out to be orthogonal in the metric of $\mathcal H(S_0)$
and in addition,
$$
\|f_0\|_{\mathcal H(S_0)}^2={\bf x}^*P^{-1}{\bf x}\quad\mbox{and}\quad
\|(\theta_{11}-\theta_{21}S_0)h\|_{\mathcal H(\mathcal S_0)}=\|h\|_{\mathcal H(\mathcal E_0)}
$$
for all $h\in\mathcal H(\mathcal E_0)$. Therefore, for $f$ of the form \eqref{f}, we have
$$
\|f\|^2_{\mathcal H(S_0)}={\bf x}^*P^{-1}{\bf x}+\|h\|_{\mathcal H(\mathcal E_0)}^2
$$
which allows to solve the norm-constrained version of the problem \eqref{np} with no extra effort.
We refer to \cite{bbt2} for the proof given there in the context of a more general
{\em Operator Argument interpolation Problem} ${\bf OAP}$ with the interpolation condition
given in terms of left-tangential evaluation calculus or equivalently, in terms of
an observability operator associated with an output stable pair. In \cite{bbt3}, similar results were shown to be true
in the context of a vector-valued de Branges-Rovnyak space associated with an operator-valued Schur-class function;
this case is briefly recalled in Section \ref{S:SingleVar} below.

The main goal of the present paper is to extend some results from \cite{bbt3} to two multivariable settings:\ a free
non-commutative setting, where we consider the de Branges-Rovnyak space associated with a contractive multiplier between
two Fock spaces, discussed in Section \ref{S:FockSpace}, and the standard commutative multivariable setting dealing with the
de Branges-Rovnyak space associated with a contractive multiplier between two Drury-Arveson spaces, discussed in
Section \ref{S:DruryArveson}. In each setting we consider the ${\bf OAP}$ with the interpolation condition given in terms of
an appropriately defined left-tangential evaluation calculus, and get a parametrization of all solutions by formulas similar to \eqref{f}.
Along the way we make use of and clarify some results from \cite{bbkats} concerning the bi-contractivity of certain 
indefinite-metric Schur-class multipliers used to generate the linear-fractional maps used for the afore-mentioned parametrizations.
This latter result appears to be new even for the definite case (thereby correcting a misconception that has appeared
in the literature), as explained in Appendix B: {\em the Schur-multiplier class is
invariant under the conjugation map} $S(z) \mapsto S^\sharp(z):= S(\overline{z})^*$ (Fock-space noncommutative setting) and
$S(\lam) \mapsto S^\sharp(\lam):= S(\overline{\lam})^*$ (Drury-Arveson-space commutative setting).

In the final section, Section \ref{S:OTprelim}, we consider the generalized de Branges-Rovnyak space $\mathcal H(T)$ associated
with an arbitrary contraction operator
$T$ (rather than a contractive multiplication operator) between two Hilbert spaces along with an interpolation problem
\eqref{dBR-Douglas3} analogous to the ${\bf OAP}$. At this level of generality, the general solution is still represented
as the orthogonal sum
of a particular minimal-norm solution and a general solution of the homogeneous problem. All solutions of the homogeneous problem,
in turn, form a subspace of $\mathcal H(T)$, which in the context of the problem \eqref{np} and in the multivariable problems
considered in Sections \ref{S:FockSpace} and \ref{S:DruryArveson}, admit more detailed Beurling-type representations.

We conclude this section with some words on notations and terminology that are used at various places in the paper. All Hilbert spaces appearing in this paper are assumed to be separable. By $\cL(\cU,\cY)$ we denote the space of bounded linear operators between Hilbert spaces $\cU$ and $\cY$, abbreviated to $\cL(\cY)$ in case $\cU=\cY$. If $\cX$ is a Hilbert  space and $G$ is a selfadjoint operator on $\cX$, we use the notation $(\cX, G)$ to denote the indefinite inner product space, or {\em Kre\u{\i}n space}, $\cX_G$ with the indefinite inner product induced by $G$:
$$
       [x, y ]_{G}: = \langle Gx, y \rangle_{\cX}.
$$
In this paper we will be primarily interested in the case that the indefinite inner product is induced by a {\em signature operator}, which
is an invertible operator $J\in \cL(\cX)$ with the property that $J=J^{-1}=J^*$.

Given two signature operators $J_1\in \cL(\cX_1)$ and $J_2\in \cL(\cX_2)$, an operator $W\in\cL(\cX_1,\cX_2)$ is called a
{\em $(J_1,J_2)$-bi-contraction} in case
\begin{equation}\label{bicon}
W^* J_2 W \preceq J_1 \quad\text{and}\quad W J_1 W^* \preceq J_2.
\end{equation}
In case only the first (second) inequality holds we say that $W$ is a {\em $(J_1,J_2)$-contraction} ({\em $(J_1,J_2)$-$*$-contraction}). Moreover, if the first (second) relation in \eqref{bicon} holds with equality then we say that $W$ is a $(J_1,J_2)$-isometry ($(J_1,J_2)$-coisometry) and if both relations hold with equality then $W$ is called $(J_1,J_2)$-unitary. Whenever $J_1=J_2$ we will simply write $J_1$-bi-contraction, $J_1$-isometry, etc. More details on Kre\u{\i}n spaces and a useful Kre\u{\i}n space lemma will be given in Appendix \ref{S:Krein}.

\section{The classical vector-valued case}\label{S:SingleVar}

In this section we recall basic results concerning the ${\bf OAP}$ in vector-valued de Branges-Rovnyak spaces that we wish to extend to the multivariable setting. To fix notation, we denote by $H^2_{\cY}$ the Hardy space of analytic $\cY$-valued functions on $\D$ with square-summable sequences of Taylor coefficients
$$
H^2_{\cY}=\bigg\{f(\lambda)=\sum_{k=0}^\infty f_k\lambda^k: \, \|f\|^2_{H^2_{\cY}}:=\sum_{k=0}^\infty \|f_k\|^2_\cY<\infty\bigg\}
$$
which  turns out to be the reproducing kernel Hilbert space $\cH_\cY(k)$ with reproducing Szeg\H{o} kernel
%\begin{equation}
$$
k_\tu{Sz}(\lambda,\eta)=(1-\lambda\bar{\eta})^{-1},
%\label{sz1}\end{equation}
$$
where we follow the convention that for any scalar positive kernel $k$ and Hilbert space $\cY$ we set $\cH_\cY(k):=\cH(k I_\cY)$.

We next denote by $\cS(\cU,\cY)$ the {\em Schur class} of analytic functions on $\D$ whose values are contractive operators in
$\cL(\cU, \cY)$ and which are characterized, in particular, as
contractive multipliers from $H^2_{\cU}$ to $H^2_{\cY}$: a function $S: \, \D\to \cL(\cU,\cY)$ belongs to $\cS(\cU,\cY)$
if and only if the associated multiplication operator
$$
M_S: \, f\to Sf
$$
is a contraction from $H^2_{\cU}$ to $H^2_{\cY}$. The latter property translates to the {\em de Branges-Rovnyak kernel}
%\begin{equation}
$$
K_S(\lam, \eta)=\frac{I_\cY-S(\lam)S(\eta)^*}{1-\lambda\bar{\eta}}
%\label{dbker}\end{equation}
$$
being a {\em positive kernel} on $\D$, i.e., $K_S$ has the property that
$$
 \sum_{i,j = 1}^N \langle K_S(\lam_i, \lam_j)y_j, y_i \rangle_\cY \ge 0
 $$
for all $\lam_1, \dots, \lam_N \in {\mathbb D}$, $y_1, \dots, y_N \in \cY$ and $N\ge 1$.
This positive kernel in turn gives rise to a reproducing kernel Hilbert space $\cH(S):=\cH(K_S)$,
the de Branges-Rovnyak space associated with $S$. Alternatively, $\cH(S)$ can be defined as
the range space $\Ran (I-M_SM_S^*)^\half$ with the lifted norm
%\begin{equation}
$$
\|(I-M_SM_S^*)^\half f\|_{\cH(S)}=\|(I-\pi)f\|_{H^2_\cY},
%\label{lift}\end{equation}
$$
where $\pi$ is the orthogonal projection onto $\kr (I-M_SM_S^*)^\half$.

\smallskip

Let us say that a pair $(E,T)$ with $E\in\cL(\cY,\cX)$ and $T\in\cL(\cX)$ is {\em output stable} if the observability operator
%\begin{equation}
$$
\cO_{E,T}: \; x\mapsto E(I-\lambda T)^{-1}x=\sum_{k=0}^\infty \lambda^k ET^kx
$$
%\label{obs}\end{equation}
maps $\cX$ into $H^2_\cY$ and is bounded. Then a standard inner-product computation gives the formula for the adjoint operator
$$
\cO_{E,T}^*f=\sum_{k=0}^\infty T^{* k} E^{*} f_{k}\quad \text{if} \quad f(\lambda) =
    \sum_{k=0}^\infty f_{k} \lambda^{k}\in H^2_{\cY}
$$
where the weak convergence of the operator series is guaranteed by the output-stability of the pair $(E,T)$.
Thereby, any output stable pair $(E,T)$ gives rise to the left-tangential evaluation
\begin{equation}\label{eval}
(E^{*} f)^{\wedge L}(T^{*}) := \cO_{E,T}^*f=\sum_{k=0}^\infty T^{* k} E^{*} f_{k}\quad \text{if} \quad f(\lambda) =
\sum_{k=0}^\infty f_{k} \lambda^{k}\in H^2_{\cY},
\end{equation}
which clearly makes sense for functions from the de Branges-Rovnyak space $\cH(S)\subset H^2_{\cY}$ and extends to operator valued Schur-class functions $S\in\cS(\cU,\cY)$ via
%\begin{equation}\label{evalS}
$$
(E^{*} S)^{\wedge L}(T^{*}) := \cO_{E,T}^*M_S |_{\cU}=\sum_{k=0}^\infty T^{* k} E^{*} S_{k}\quad \text{if} \quad S(\lambda) =
\sum_{k=0}^\infty S_{k} \lambda^{k}.
$$
%\end{equation}
The restriction to $\cU$ in the above formula is to be understood as the restriction to the constant $\cU$-valued functions in $H^2_\cU$.

Given a Schur-class function $S_0$, the  {\em Operator Argument interpolation Problem in $\cH(S_0)$}
considered in \cite{bbt2,bbt3} is formulated as follows.

\medskip
\noindent
${\bf OAP}_{\cH(S_0)}$:
{\em  Given $S_0 \in \cS(\cU, \cY)$,  an output stable pair $(E,T)\in\cL(\cX,\cY)\times \cL(\cX)$, and
a vector ${\bf x}\in\cX$, the problem is:}
    \begin{equation}\label{evaldbra}
\text{find all } f\in\cH(S_0)\quad\mbox{such that}\quad (E^{*} f)^{\wedge L}(T^{*})={\bf x}.
    \end{equation}
If $\dim \cX=n$, $\dim\cU=\dim\cY=1$ and $T$, $E$ and ${\bf x}$ are chosen as in \eqref{data}, then we have from \eqref{eval}
$$
(E^{*} f)^{\wedge L}(T^{*}) = \sum_{k=0}^\infty \sbm{\lambda_1^k \\ \vdots \\ \lambda_n^k}f_k=\sbm{f(\lambda_1)\\ \vdots \\ f(\lambda_n)}
$$
from which we see that condition \eqref{evaldbra} collapses to the $n$ conditions in \eqref{np}.

The general case is handled as follows. We introduce $N\in \cL(\cX,\cU)$ and
$P\in\cL(\cX)$ by
\begin{equation}
\label{NdefIntro}
N^*:= (E^{*} S_0)^{\wedge L}(T^{*}),\quad P:=\cO_{E,T}^*\cO_{E,T}-\cO_{N,T}^*\cO_{N,T}.
\end{equation}
It is not hard to show that since $S_0$ is a Schur-class function, the pair $(N,T)$ is output stable and the associated observability
operator $\cO_{N,T}: \; \cX\to H^2_{\cU}$ satisfies
%\begin{equation}\label{SC-OpForm}
$$
\cO_{E,T}^*M_{S_0}=\cO_{N,T}^*.
$$
%\end{equation}
Therefore, the operator $P$ in \eqref{NdefIntro} is bounded and positive semidefinite:
$$
P=\cO_{E,T}^*\cO_{E,T}-\cO_{N,T}^*\cO_{N,T}=\cO_{E,T}^*(I-M_{S_0}M_{S_0}^*)\cO_{E,T}\succeq 0.
$$
It was shown in \cite{bbt3} that the problem \eqref{evaldbra} has a solution if and only if ${\bf x}\in\Ran P^{\half}$.
If $P\succ 0$ is strictly positive definite (i.e., positive semidefinite and boundedly invertible), one can construct
an operator-valued function
$$
\Theta(\lambda)=\begin{bmatrix}\Theta_{11}(\lambda) & \Theta_{12}(\lambda)\\ \Theta_{21}(\lambda) & \Theta_{22}(\lambda)\end{bmatrix}: \;
\begin{bmatrix}\cY \\ \cU\end{bmatrix}\to \begin{bmatrix}\cY \\ \cU\end{bmatrix}\quad (\lambda\in\mathbb D)
$$
satisfying the identity
\begin{equation}
\frac{J_{_{\cY,\cU}} - \Theta(\lambda)J_{_{\cY,\cU}}\Theta(\eta)^{*}}{1 -\lambda\overline{\eta}}
=\begin{bmatrix}E \\ N\end{bmatrix}(I - \lambda T)^{-1} P^{-1}(I -
\overline{\eta}T^{*})^{-1}\begin{bmatrix}E^* & N^*\end{bmatrix}
\label{iden}
\end{equation}
for all $\lambda,\eta\in\D$, where $J_{_{\cY,\cU}}$ denotes the signature operator
\begin{equation}\label{J-inter}
J_{_{\cY,\cU}} := \mat{ I_\cY & 0 \\ 0 & -I_\cU}.
\end{equation}
The function $\Theta$ is uniquely determined by \eqref{iden} up to multiplication by a constant $J_{\cY,\cU}$-unitary factor on the right. If $1\not\in\sigma(T^*)$, we can use
formula \eqref{th} (with $I_\cX$ instead of $I_n$). Otherwise, $\Theta$ can be constructed via solving a certain Cholesky
factorization problem as will be explained in the multivariable settings to come. It follows from \eqref{iden} that
$$
J_{_{\cY,\cU}} - \Theta(\lambda)J_{_{\cY,\cU}}\Theta(\lambda)^{*}\succeq 0, \quad \lambda\in\mathbb D
$$
while a similar factorization as in \eqref{iden} for $\frac{J_{_{\cY,\cU}} - \Theta(\lambda)^*J_{_{\cY,\cU}}\Theta(\eta)}{1 -\lambda\overline{\eta}}$ implies that
$$
J_{_{\cY,\cU}} - \Theta(\lambda)^*J_{_{\cY,\cU}}\Theta(\lambda)\succeq 0, \quad \lambda\in\mathbb D.
$$
Hence, we can conclude that $\Theta$ is $(J_{_{\cY,\cU}},J_{_{\cY,\cU}})$-bi-contractive in $\mathbb D$. Therefore, the formula
$$
\cE\mapsto {\fT}_\Theta[\cE]:= (\Theta_{11}\cE+\Theta_{12})(\Theta_{21}\cE+\Theta_{22})^{-1}
$$
makes sense for all $\cE\in\mathcal S(\cU,\cY)$ and defines a linear fractional map of $\mathcal S(\cU,\cY)$ into
itself. Its range coincides with the solution set of a certain Schur class ${\bf OAP}$.

\smallskip

The following result has been established via various methods, cf., \cite{FFGK98,BB18,FtHK18}, including several where
the  de Branges-Rovnyak spaces played a prominent role.  Let us mention that \cite{BB18} as well as \cite{BR07}
consider the more general Bitangential Operator-Argument interpolation Problem (with norm constraint),
where a tangential operator-argument interpolation condition is considered both from the left and the right simultaneously.

\begin{thm}  \label{T:1.0}
Given output stable pairs $(E,T)$ and $(N,T)$ with $E\in\cL(\cX,\cY)$ and $N\in\cL(\cX,\cU)$, there exists a function
\begin{equation}\label{oapschur}
S\in\cS(\cU,\cY)\quad\mbox{such that}\quad
(E^{*} S)^{\wedge L}(T^{*})=N^*
    \end{equation}
if and only if $P:=\cO_{E,T}^*\cO_{E,T}-\cO_{N,T}^*\cO_{N,T}\succeq 0$. Furthermore, if $P\succ 0$, then the solution set
of the problem \eqref{oapschur} is parametrized by the formula
\begin{equation}\label{edescr}
S=  \fT_\Theta[\cE]:= (\Theta_{11}\cE+\Theta_{12})(\Theta_{21}\cE+\Theta_{22})^{-1}
\end{equation}
with free parameter $\cE\in \cS(\cU,\cY)$, where $\Theta$
is a $(J_{_{\cY,\cU}},J_{_{\cY,\cU}})$-contractive function subject to the identity \eqref{iden}.
\end{thm}

Returning to ${\bf OAP}_{\cH(S_0)}$, in which $S_0$ is now a given element of the Schur class $\cS(\cU, \cY)$  and $N$ is defined by \eqref{NdefIntro}, it follows from the first formula in \eqref{NdefIntro} that  \eqref{oapschur} holds with $S = S_0$, so that $S_0=\fT_\Theta[\cE_0]$ for some $\cE_0$ in $\cS(\cU,\cY)$.

\smallskip

We now present the solution to ${\bf OAP}_{\cH(S_0)}$ for the case where $P\succ 0$.
\begin{thm}
\label{T:1.2}
The ${\bf OAP}_{\cH(S_0)}$ has a solution if and only if ${\bf x}\in\Ran P^{\frac{1}{2}}$.
Assume that $P\succ 0$. Let $\Theta$ is a function satisfying \eqref{iden} and let $\cE_0$ be a Schur-class
function such that $S_0=\fT_\Theta[\cE_0]$. Then
\begin{enumerate}
\item All solutions $f$ of the ${\bf OAP}_{\cH(S_0)}$ are parametrized by the formula
\begin{equation}
f(\lambda)=f_0(\lambda)+(\Theta_{11}(\lambda)-S_0(\lambda)\Theta_{21}(\lambda))h(\lambda)
\label{1.20}
\end{equation}
where
$$
%\begin{equation}
f_0(\lambda)=(E-S_0(\lam)N)(I-\lam T)^{-1}P^{-1}{\bf x}
%\label{1.20u}\end{equation}
$$
and where $h$ is a free parameter from the space $\cH(\cE_0)$.
\item The representation \eqref{1.20} is orthogonal in the metric of $\cH(S_0)$; hence $f_0$
is the minimal-norm solution to the ${\bf OAP}_{\cH(S_0)}$.
\item The multiplication operator
\begin{equation}
M_{\Theta_{11}-S_0\Theta_{21}}: \, \cH(\cE_0)\to \cH(S_0)
\label{paris}
\end{equation}
is a partial isometry.
\end{enumerate}
If in addition the operator $T$ satisfies the condition
\begin{equation}
\bigg(\bigcap_{k\ge 1}
\Ran (T^*)^k\bigg)\bigcap\Ker(T^*)=\{0\},
\label{8.30}
\end{equation}
then
\begin{itemize}
\item[(a)] The transformation $\fT_\Theta: \, \cS(\cU,\cY)\to \cS(\cU,\cY)$ is injective and hence,
$S_0=\fT_\Theta[\cE_0]$ for a unique function $\cE_0\in\cS(\cU,\cY)$.
\item[(b)] The multiplication operator \eqref{paris} is an isometry and hence, for every $f$ of the form \eqref{1.20},
$$
\|f\|^2_{\cH(S_0)}=\|f_0\|_{\cH(S_0)}^2+\|(\Theta_{11}-S_0\Theta_{21})h\|_{\cH(S_0)}^2=
\|P^{-\frac{1}{2}}{\bf x}\|^2_{\cX}+\|h\|^2_{\cH(\cE_0)}.
$$
\end{itemize}
\end{thm}
The first part in Theorem \ref{T:1.2} can be considered as a special case of Theorem 4.1 and Proposition 4.2 in \cite{bbt3}, where
the more general {\em Abstract Interpolation Problem} {\bf AIP}$_{\cH(S_0)}$ was considered.
The results in \cite{bbt3} cover the general case when $P\succeq 0$ is not necessarily strictly positive definite
and rely on a Redheffer linear-fractional representation
$$
S_0=\Sigma_{11}+\Sigma_{12}\cE_0(I-\Sigma_{22}\cE_0)^{-1}\Sigma_{21}
$$
of $S_0$ in terms of Schur-class functions $\cE$ and $\Sigma=\sbm{\Sigma_{11}& \Sigma_{12}\\ \Sigma_{21} & \Sigma_{22}}$,
which, however, can be  written in terms of the chain-matrix linear-fractional formula \eqref{edescr} if $P\succ 0$.
The ``additional" part of the theorem follows from \cite[Theorem 4.8]{bbt3}. Note that condition \eqref{8.30} holds
automatically if $\dim \cX < \infty$ (in particular, for the Nevanlinna-Pick problem considered in Section \ref{S:Intro}). Also it
holds if  $T^*$ is nilpotent or if $\Ker(T^{*}) = \{0\}$.
Let us also mention that there is now a generalization of Theorem \ref{T:1.2} (specialized
to the scalar-valued setting) to the case where the de Branges-Rovnyak-type spaces are specified as
sub-Bergman rather than sub-Hardy spaces (see \cite{BB13}).

\subsection{{\bf OAP}$_{\cH(S_0)}$ as a generalization of the Beur\-ling-Lax theorem}\label{S:BLthm}
As explained in Theorem 6.2 of \cite{bbt3} but in the context of a more general Abstract Interpolation Problem,
one can view Theorem \ref{T:1.2} as a generalization of the Beurling-Lax theorem (see e.g.~\cite{NF}).  Recall that any
vector-valued Hardy space $H^2_\cY$ is equipped with the shift operator $S_\cY \colon f(\lam) \mapsto \lam f(\lam)$.
The Beurling-Lax theorem characterizes subspaces $\cM \subset H^2_\cY$ which are invariant under $S_\cY$
as those which have a representation of the form $B \cdot H^2_{\cY_0}$ for some coefficient Hilbert space
$\cY_0$ with $\dim \cY_0 \le \dim \cY$ and $B$ an inner function on the disk ${\mathbb D}$
with values in $\cL(\cY_0, \cY)$ (i.e., such that the multiplication operator $M_B \colon h(\lam) \mapsto B(\lam) h(\lam)$
is an isometry from $H^2_{\cY_0}$ into $H^2_\cY$).

\smallskip

To make the connection between Theorem \ref{T:1.2} and the Beurling-Lax theorem,
let us consider the special case of Theorem \ref{T:1.2} where the coefficient space
$\cU$ is taken to be the zero space $\{0\}$ and hence the Schur class  $\cS(\{0\}, \cY)$  consists only
of the zero function $S_0(z) =0 \colon \{0\} \to \cY$, $N$ defined by \eqref{NdefIntro} is also $0$,
the signature matrix \eqref{J-inter} collapses to $J = I_\cY$, the $J$-bi-inner function
$\Theta$ collapses to the bi-inner function $\Theta_{11}$ with values in $\cL(\cU, \cY)$,
the parameter function $\cE_0$ such that $S_0=\fT_{\Theta}[\cE_0]=\Theta_{11}\cE_0$ is the zero function
$\cE_0(\lam) \colon \{0\} \to \cY$, the de Branges-Rovnyak space $\cH(S_0)$ collapses to the Hardy space $H^2_\cY$.
Let us take $\bx = 0$ so as to focus on the homogeneous problem. We are left with

\medskip
\noindent
{\bf OAP}$_{H^2_\cY}$:  {\em Given an output stable pair $(E, T)$ with $E \in \cL(\cX, \cY)$ and $T \in \cL(\cX)$,
find all functions $f \in H^2_\cY$ so that}
$$
%\begin{equation}  \label{hom-int}
(E^*f)^{\wedge L}(T^*):=\cO_{E,T}^* f  = 0.
%\end{equation}
$$
\noindent
Let $\cM_{E,T} \subset H^2_\cY$ denote the set of all solutions.  From the intertwining property of
$\cO_{E,T}$
$$
S_\cY^* \cO_{E,T} = \cO_{E,T} T,
$$
we see that
$$
\cO_{E,T}^* S_\cY = T^* \cO_{E,T}^*
$$
from which it follows immediately that the subspace $\cM_{E,T}$ is $S_\cY$-invariant.
As a consequence of items (1) and (3)  in Theorem \ref{T:1.2}, we see that
$$
\cM_{E,T} =  \{ \Theta_{11} h \colon h \in H^2_\cU\}
$$
where $\Theta_{11}$ is inner, i.e., $M_{\Theta_{11}} \colon H^2_\cU \to H^2_\cY$ is an isometry.  Thus
$\Theta_{11}$ serves as a Beurling-Lax representer for the shift-invariant subspace $\cM_{E,T}$.

\smallskip

To get a proof of the Beurling-Lax theorem in general from this analysis, it remains only to argue that every
shift-invariant subspace $\cM$ of $H^2_\cY$ has the form $\cM = \cM_{E,T}$ for some output stable
pair $(E,T)$.  But this is well known from operator-model theory ideas as follows.
Let us introduce the notation $\bev_0$ for the operator from $H^2_\cY$ to $\cY$ given by
evaluation at $0$:  $\bev_0 \colon f \mapsto f(0)$.  View the pair of operators
$$
(\bev_0, S_\cY^*) \in \cL(H^2_\cY, \cY) \times \cL(H^2_\cY)
$$
as an output pair and form the observability operator $\cO_{\bev_0, S_\cY^*} \colon H^2_\cY \to H^2_\cY$.
The computation
$$
  \cO_{\bev_0, S_\cY^*} \colon f \mapsto \sum_{n=0}^\infty (\bev_0 S_\cY^{*n} f) \lam^n
   = \sum_{n=0}^\infty f_n \lam^n = f(\lam)
 $$
shows that $\cO_{\bev_0, S_\cY^*}$ is the identity operator on $H^2_\cY$.
Now consider the restricted output pair $(E, T) \in \cL(\cX, \cY) \times \cL(\cX)$ given by
$$
 \cX = \cM^\perp, \quad E = \bev_0|_\cX, \quad T = S_\cY^*|_\cX.
$$
A consequence of the preceding discussion is that $\cM^\perp  = \Ran \cO_{E,T}$.
Hence $\cM = (\cM^\perp)^\perp$ is given by $\cM = \Ker \cO_{E,T}^*$ as wanted.

\smallskip

 In conclusion, we can view Theorem \ref{T:1.2} as a generalization of the Beurling-Lax theorem to a
 more general non-homogeneous context.

\section{The Fock space setting}\label{S:FockSpace}

In this section we present the analogue of the single-variable results of Section~\ref{S:SingleVar} for multipliers between two Fock spaces. To define the Fock space, we let $\free$ denote the unital free
semigroup (i.e., monoid)  generated by the set of $d$ letters $\{1, \dots,
d\}$. Elements of $\free$ are words of the form $i_{N}
\cdots i_{1}$ where $i_{\ell} \in \{1, \dots, d\}$ for each $\ell\in\{1,
\dots, N\}$ with multiplication given by concatenation. The unit element
of $\free$ is the empty word denoted by $\emptyset$.
For $\alpha = i_{N} i_{N-1} \cdots i_{1} \in \free$, we let $|\alpha|$ denote the number
$N$ of letters in $\alpha$. Furthermore, we define the {\em reversal} of $\alpha$ to be the element $\alpha^\top\in\free$ given by
\begin{equation}\label{reversal}
\alpha^\top = i_1 \cdots i_N \quad\text{if}\quad \alpha = i_N \cdots i_1.
\end{equation}

We let $z = (z_{1}, \dots,z_{d})$ to be a
collection of $d$ formal noncommuting variables and given a  Hilbert space $\cY$,
let $\cY\langle\langle z\rangle\rangle$ denote the set of noncommutative formal power series
$\sum_{\alpha \in \free} f_\alpha z^\alpha$ where $f_\alpha \in \cY$ and where
\begin{equation}
\label{1.1}
z^{\alpha} =z_{i_{N}}z_{i_{N-1}} \cdots
z_{i_{1}}\quad\mbox{if}\quad \alpha= i_{N}i_{N-1} \cdots i_{1}.
\end{equation}
The Fock space $H^2_{\cY}(\free)$ is the Hilbert space given by
\begin{equation}   \label{Fock}
H_{\cY}^{2}(\free) = \bigg\{f(z)=\sum_{\alpha \in \free}
f_{\alpha}z^{\alpha}
\in \cY\langle\langle z\rangle\rangle\colon \;
\|f\|_{H_{\cY}^{2}(\free)}^2:=\sum_{\alpha \in \free} \|f_{\alpha}\|_{\cY}^{2} <\infty \bigg\}
\end{equation}
with inner product
\[
\inn{f}{g}_{H_{\cY}^{2}(\free)}:=\sum_{\alpha \in \free} \ov{g}_\alpha f_{\alpha} \quad \mbox{for }f(z)=\sum_{\alpha \in \free}
f_{\alpha}z^{\alpha},\, g(z)=\sum_{\alpha \in \free}
g_{\alpha}z^{\alpha}\in H_{\cY}^{2}(\free),
\]
and is equipped with the tuple ${\bf R}_z$ of right coordinate-variable multipliers
\begin{equation}
{\bf R}_z=(R_{z_1},\ldots,R_{z_d}), \quad R_{z_j}: \; f(z)\mapsto f(z)z_j.
\label{intertw}
\end{equation}

The noncommutative Schur class $\mathcal S_{{\rm nc},d}(\cU,\cY)$ is then defined as the set of
{\em contractive multipliers} from $H_{\cU}^{2}(\free)$ to $H_{\cY}^{2}(\free)$, i.e.,
formal power series $S(z)=\sum_{\alpha\in\free}S_\alpha z^\alpha$
with coefficients $S_\alpha\in\cL(\cU,\cY)$ such that the associated multiplication operator
\begin{equation}
M_S: \, u(z)=\sum_{\alpha\in\free}u_\alpha z^\alpha\mapsto S(z)u(z):=\sum_{\alpha\in\free}
\bigg(\sum_{\beta\gamma=\alpha}S_\beta u_\gamma\bigg)z^\alpha
\label{mo}
\end{equation}
is a contraction from $H_{\cU}^{2}(\free)$ to $H_{\cY}^{2}(\free)$. Similar to the single-variable
case, a contraction operator $\Psi: \, H_{\cU}^{2}(\free)$ to $H_{\cY}^{2}(\free)$ equals $M_S$ for
some $S\in\mathcal S_{{\rm nc},d}(\cU,\cY)$ if and only if
\begin{equation}   \label{CharMult}
(R_{z_j}\otimes I_{\cY})\Psi=\Psi (R_{z_j}\otimes I_{\cU}) \quad\mbox{for}\quad j=1,\ldots,d.
\end{equation}
Given $S\in\mathcal S_{{\rm nc},d}(\cU,\cY)$, the associated de Branges-Rovnyak space $\cH(S)$ is defined
as the range space
\begin{equation}
{\cH}(S)=\Ran (I-M_SM^*_S)^{\frac{1}{2}}
\label{1.13}
\end{equation}
with the lifted norm defined by
\begin{equation}
\| (I - M_{S}M_{S}^{*})h\|^2_{\cH(K_{S})}=\langle (I - M_{S} M_{S}^{*})h, \,
h\rangle_{H^2_{\cY}(\free)}.
\label{1.15}
\end{equation}
This space also can be viewed as a reproducing kernel Hilbert space, but in the formal noncommutative sense of \cite{NFRKHS}, which we recall in Subsection \ref{S:NFRKHS}.

\subsection{Interpolation problem}\label{S:FockInter}
Extending the noncommutative functional calculus \eqref{1.1}
to a $d$-tuple of operators ${\mathbf T} = (T_{1}, \dots, T_{d})\in\cL(\cX)^d$ by
%\begin{equation}
$$
{\mathbf T}^{\alpha} := T_{i_{N}}T_{i_{N-1}} \cdots T_{i_{1}}\quad
\text{if}\quad   \alpha = i_{N} i_{N-1} \cdots i_{1} \in \free,
$$
%\label{1.1op}\end{equation}
and letting
\begin{equation}
Z(z) = \begin{bmatrix} z_{1} & \cdots & z_{d}
\end{bmatrix}\otimes I_{\cX} \quad\mbox{and}\quad
T=\begin{bmatrix} T_{1} \\ \vdots \\ T_{d} \end{bmatrix},
\label{1.6a}
\end{equation}
we define the noncommutative observability operator $\cO_{E,\bT}$ of the output pair $(E,\bT)$
with $E\in\cL(\cX,\cY)$ by the formula
\begin{equation}
\cO_{E, \bT} \colon \, x \mapsto \sum_{\alpha\in\free}(E\bT^\alpha x)z^\alpha=
E (I_{\cX} - Z(z)T)^{-1}x
\label{obsnc}
\end{equation}
and say that the pair $(E,\bT)$ is {\em output stable} if the operator
$\cO_{E, \bT}: \, \cX\to H^2_{\cY}(\free)$ is bounded. In this case, we
introduce the {\em observability gramian}
\begin{equation}
{\mathcal G}_{E, \bT}:={\mathcal O}_{E,\bT}^*{\mathcal O}_{E, \bT}
=\sum_{\alpha\in\free} \bT^{*\alpha^\top}E^{*} E\bT^\alpha
\label{pr}
\end{equation}
and note that the strong convergence of the series in \eqref{pr}
follows from the power-series expansion \eqref{obsnc} for the
observability operator together with the characterization
\eqref{Fock} of the $H^2_\cY(\free)$-norm.

\smallskip

For an output-stable pair $(E, {\bf T})$ as above,
we define a {\em left-tangential functional calculus}
$f \to (E^{*}f)^{\wedge  L}(\bT^{*})$ on $H^2_{\cY}(\free)$ by
          \begin{equation}
\label{2.1u}
(E^{*} f)^{\wedge L}(\bT^{*}) = \sum_{\alpha\in\free}
             \bT^{*\alpha^\top} E^{*} f_{\alpha}\quad \text{if} \quad f =
\sum_{\alpha\in\free}f_{\alpha} z^{\alpha} \in H^2_{\cY}(\free).
        \end{equation}
The computation
\begin{equation}
\Big{\langle} \sum_{\alpha\in\free} \bT^{*\alpha^\top}E^{*}f_{\alpha}, \; x {\Big \rangle}_{\cX} =
        \sum_{\alpha\in\free} \left\langle f_{\alpha}, \; E \bT^{\alpha}
        x \right \rangle_{\cY}
         = \langle f, \; \mathcal {O}_{E, \bT} x\rangle_{H^2_{\cY}(\free)}
\label{compute}
\end{equation}
        shows that the output-stability of the pair $(E, \bT)$ is exactly
        what is needed to verify that the infinite series in the definition
        \eqref{2.1u} of $(E^{*}f)^{\wedge L}(\bT^{*})$ converges
        in the weak topology on $\cX$. In fact the left-tangential
        evaluation with operator argument $f \to (E^{*}f)^{\wedge L}
(\bT^{*})$
        amounts to the adjoint of the observability operator:
%        \begin{equation}
$$
            (E^{*} f)^{\wedge L}(\bT^{*}) = \cO_{E, \bT}^{*}
            f \quad\text{for}\quad f \in H^2_{\cY}(\free).
%\label{2.2}\end{equation}
$$
The latter evaluation makes sense for all $f\in\cH(S)\subset H^2_{\cY}(\free)$ and on the other hand,
extends to contractive multipliers $S\in\cS_{{\rm nc},d}(\cU,\cY)$ by
%\begin{equation}
$$
(E^{*} S)^{\wedge L}(\bT^{*}) = \cO_{E, \bT}^{*}M_S|_{\cU}: \, \cU\to \cX
$$
%\label{2.3r}\end{equation}
with $\cU$ in the restriction to be identified with the constant $\cU$-valued functions in $H^2_\cU(\free)$. The Operator Argument interpolation Problem in the de Branges-Rovnyak space $\cH(S)$ is formulated similarly to its single-variable prototype.

\medskip
\noindent
${\bf OAP}_{\cH(S_0)}$: {\em Given an output stable pair $(E,\bT)\in\cL(\cU,\cY)\times\cL(\cX)^d$, a vector
${\bf x}\in\cX$
and a Schur-class multiplier $S_0 \in\cS_{{\rm nc},d}(\cU,\cY)$, find all $f\in\cH(S_0)$ such that
    \begin{equation}\label{evaldbr}
    (E^{*} f)^{\wedge L}(\bT^{*}):= \cO_{E, \bT}^{*}f={\bf x}.
    \end{equation}}
With a given output stable pair $(E,{\bf T})$
and a Schur-class multiplier
$$
S_0(z)=\sum_{\alpha\in\free}^\infty S_{0,\alpha} z^\alpha,
$$
we associate the operator $N\in\cL(\cX,\cU)$ defined by the formula
\begin{equation}\label{Ndef}
N:=(E^{*} S_0)^{\wedge L}(\bT^{*})^*=
\sum_{\alpha\in\free} S_{0,\alpha}^*E{\bf T}^\alpha
\end{equation}
or, equivalently, via its adjoint
\begin{equation}
N^*=\cO_{E, {\bf T}}^{*}M_{S_0}\vert_\cU: \, \cU\to\cX.
\label{1.8}
\end{equation}
By \cite[Proposition 4.1]{bbkats}, the pair $(N,{\bf T})$ is output stable and the
equality
\begin{equation}
\cO_{E,{\bf T}}^*M_{S_0}=\cO_{N,{\bf T}}^*
\label{1.8*}
\end{equation}
holds. Therefore,
\begin{equation}
P:=\cO_{E,{\bf T}}^*\cO_{E,{\bf T}}-\cO_{N,{\bf T}}^*\cO_{N,{\bf T}}
=\cO_{E,{\bf T}}^*(I-M_{S_0}M_{S_0}^*)\cO_{E,{\bf T}}\succeq 0.
\label{1.8**}
\end{equation}
Using the formula \eqref{pr} for $\cO_{E, \bT}^{*}\cO_{E, \bT}$ and a similar formula for the
gramian $\cO_{N, \bT}^{*}\cO_{N, \bT}$ we get the infinite series representation
$$
P=\sum_{\alpha\in\free} \bT^{*\alpha^\top}(E^{*} E-N^*N)\bT^\alpha,
$$
from which it is easy to verify that $P$ satisfies the Stein identity
\begin{equation}
\label{2.8}
         P - \sum_{j=1}^{d} T_{j}^{*} P T_{j} = E^*E-N^*N.
\end{equation}
We next introduce the formal power series
\begin{equation}
F^{S_0}(z)=(E-S_0(z)N)(I-Z(z)T)^{-1}.
\label{1.17}
\end{equation}

\begin{lem}\label{L:3.1}
Let $P$ be defined as in \eqref{1.8**} and $F^{S_0}$ as in \eqref{1.17}. Then for any $x\in\cX$, we have that $F^{S_0}x:=F^{S_0}(\cdot)x$ is in $\cH(S_0)$ and
\begin{equation}
\|F^{S_0}x\|_{\cH(S_0)}=\|P^\half x\|_{\cX}.
\label{1.17*}
\end{equation}
\end{lem}

\begin{proof}
Indeed, by \eqref{1.17} and \eqref{1.8*}, the multiplication operator
$M_{F^{S_0}}$ can be written in terms of observability operators as
\begin{equation}
M_{F^{S_0}}=\cO_{E,{\bf T}}-M_{S_0}\cO_{N,{\bf T}}=(I_{\cX}-M_{S_0}M_{S_0}^*)\cO_{E,{\bf T}}
\label{1.17**}
\end{equation}
which together with the range characterization  \eqref{1.13} of $\cH(S_0)$ implies
that $M_{F^{S_0}}$ maps $\cX$ into $\cH(S_0)$. We now apply equality \eqref{1.15}
to $h=\cO_{E,{\bf T}}x$ and make use of \eqref{1.8**} to conclude
\begin{align*}
\|F^{S_0}x\|^2_{\cH(S_0)}=&
\langle (I - M_{S_0} M_{S_0}^{*}) \cO_{E,{\bf T}}x, \cO_{E,{\bf T}}x\rangle_{H_\cY^2(\free)}=
\langle Px, x \rangle_{\cX}
\end{align*}
which is equivalent to \eqref{1.17*}.
\end{proof}

\begin{rem}\label{R:3.0}
{\rm For the operator
$M_{F^{S_0}} \colon \cX\to \cH(S_0)\subset H^2_\cY(\free)$ we will write $M_{F^{S_0}}^*$ for its adjoint
in the metric of $H^2_\cY(\free)$ and we will denote by $M_{F^{S_0}}^{[*]}$ the adjoint of $M_{F^{S_0}}$ in the metric
of $\cH(S_0)$. Note that these two adjoints are not the same  unless $S_0$ is inner (i.e., the multiplication
operator $M_{S_0}: \; H^2_\cU \to H^2_\cY$ is an isometry)}.
\end{rem}

\begin{lem}\label{L:3.2}
Let $M_{F^{S_0}}: \; \cX\to \cH(S_0)$ be defined as in \eqref{1.17**}. Then
\begin{equation}
M_{F^{S_0}}^{[*]}=\cO_{E,{\bf T}}^*|_{\cH(S_0)}\quad\mbox{and}\quad M_{F^{S_0}}^{[*]}M_{F^{S_0}}=P.
\label{1.10}
\end{equation}
\end{lem}
\begin{proof}
The first formula is justified by the equalities
 \begin{align*}
    \langle M_{F^{S_0}}^{[*]} f, \, x\rangle_{\cX}
    =\langle f, \,M_{F^{S_0}} x\rangle_{\cH(S_0)}
    =&\langle f, \, (I-M_{S_0}M_{S_0}^*)\cO_{E,{\bf T}}x\rangle_{\cH(S_0)}
    \notag \\
    =&\langle f, \, \cO_{E,{\bf T}}x\rangle_{H^2_\cY(\free)}=\langle
    \cO_{E,{\bf T}}^*f, \, x\rangle_{\cX}\notag
    \end{align*}
holding for all $f\in\cH(S_0)$ and $x\in\cX$. The second equality follows from \eqref{1.17*}.
\end{proof}
\begin{cor}
The ${\bf OAP}_{\cH(S_0)}$ has a solution if and only if ${\bf x}\in\Ran P^\half$.
\label{R:3.3}
\end{cor}
Indeed, by the first formula in \eqref{1.10}, the interpolation condition \eqref{evaldbr}
can be written as
\begin{equation}
M^{[*]}_{F^{S_0}}f={\bf x}.
\label{newcond}
\end{equation}
Combining the latter condition with the second equality in \eqref{1.10} we see that
the problem has a solution $f$ if and only if ${\bf x}\in\Ran M^{[*]}_{F^{S_0}}=\Ran P^\half$.

\subsection{Noncommutative formal reproducing kernel Hilbert spaces} \label{S:NFRKHS}

To proceed we need formal positive kernels and associated
reproducing kernel Hilbert spaces. Here we briefly recall the basic concepts and refer for more complete details to \cite{NFRKHS} and  as well as \cite[Section 2.1]{bbbook}.

Given a coefficient Hilbert space $\cY$, we define the space $\cL(\cY)\langle \langle z, \bzeta \rangle \rangle$
to consist of all formal power series (formal kernels)
\begin{equation}   \label{fk}
K(z,\zeta) = \sum_{\alpha, \beta \in \free} K_{\alpha, \beta}z^{\alpha}
\overline{\zeta}^{\beta^{\top}}, \quad K_{\alpha,\beta}\in\cL(\cY),
\end{equation}
in two collections of formal noncommuting indeterminates $z = (z_{1}, \dots,z_{d})$ and $\bzeta = (\bzeta_{1}, \dots,$  $\bzeta_{d})$ such that each $z_{k}$ commutes with each $\bzeta_{j}$. Here $\beta^{\top}$is defined as in \eqref{reversal}. In the algebraic manipulations to follow
we use the formal notations
 $$
   \overline{\zeta}_j^* = \zeta_j, \quad (\overline{\zeta}^\alpha)^* = \zeta^{\alpha^\top}.
 $$
The kernel $K$ in \eqref{fk} is called {\em positive} if for any finitely supported $\cY$-valued function
$\alpha \mapsto y_{\alpha} $ on $\free$,
$$
   \sum_{\alpha, \beta \in \free} \langle K_{\alpha, \beta} \,
   y_{\alpha}, y_{\beta} \rangle_{_{\cY}} \ge 0.
$$
Such kernels are characterized by the existence of a  {\em Kolmogorov decomposition}
%\begin{equation}   \label{Kol}
$$
          K(z,\zeta) = H(z) H(\zeta)^{*}
$$
% \end{equation}
with $H$ some operator-valued formal power series $H\in\cL(\cX,\cY)\langle \langle z\rangle \rangle$, for some Hilbert space $\cX$, that is, $ H(z)$ is of the form
$$
H(z)=\sum_{\alpha\in\free} H_{\alpha} z^{\alpha}\quad\text{with}\ H_{\alpha}\in \cL(\cX,\cY).
$$

We say that a Hilbert space $\mathcal H$ consisting of power series $f$ with coefficients in $\cY$
$$
f(z)= \sum_{\alpha \in \free} f_\alpha z^\alpha  \in \cY\langle \langle z \rangle \rangle,
$$
is a {\em noncommutative formal reproducing kernel Hilbert space}
(NFRKHS) if, for each $\beta \in \free$, the linear operator $\Phi_\beta \colon \cH \to \cY$ defined by
\begin{equation}  \label{Phi-beta}
  \Phi_\beta \colon f(z) = \sum_{\alpha \in \free} f_\alpha z^\alpha \mapsto f_\beta
\end{equation}
is continuous.

To explain the reproducing-kernel property for this setting, we need to introduce more general
formal inner-product pairings as follows.
Let $\overline{\zeta} = (\overline{\zeta}_1, \dots, \overline{\zeta}_d)$ be a second $d$-tuple of formal noncommuting indeterminates with properties as discussed in the previous paragraph. Let us introduce, for a general Hilbert space $\cC$, formal pairings
 \begin{align*}
& \langle \cdot, \cdot \rangle_{\cC \times \cC\langle \langle \overline{\zeta} \rangle \rangle} \to
 {\mathbb C} \langle \langle \zeta \rangle \rangle, \quad
  \langle  \cdot, \cdot \rangle_{\cC \langle \langle \zeta \rangle \rangle \times \cC}
 \to {\mathbb C}\langle \langle \zeta \rangle \rangle
 \end{align*}
by
\begin{align*}
& \Big{\langle} c, \sum_{\beta\in\free} c'_\beta \overline{\zeta}^{\beta^\top} \Big{\rangle}_{\cC \times \cC\langle \langle \overline{\zeta}
\rangle \rangle}  \!\!\! = \!\! \sum_{\beta\in\free} \langle c, c'_\beta \rangle_{_\cC} \zeta^\beta,  \quad
 \Big{\langle} \sum_{\alpha\in\free} c_\alpha \zeta^\alpha, c' \Big{\rangle}_{\cC(\zeta) \times \cC} \!\!\! = \!\!
\sum_{\alpha\in\free} \langle c_\alpha, c' \rangle_{_\cC} \zeta^\alpha.
\end{align*}
In particular, we have
$$
  \langle c, c' \overline{\zeta}^{\beta^\top} \rangle_{\cC \times \cC\langle \langle \overline{\zeta}  \rangle \rangle}
    = \langle c, c' \rangle_{_\cC} \cdot \zeta^\beta.
 $$
Now we assume that $\cH$ is a NFRKHS. With $\Phi_\beta$ as defined in \eqref{Phi-beta} we set
$$
H(z)=\sum_{\alpha\in\free} \Phi_\beta z^{\alpha}\quad\text{so that}\quad
H(\zeta)^* = \sum_{\beta\in\free} \Phi_\beta^* \overline{\zeta}^{\beta^\top}.
$$
Then for each $y \in \cY$ by definition $\Phi_\beta^* y \in \cH$ and we can write
$$
H(\zeta)^* y = \sum_{\beta\in\free} (\Phi_\beta^* y) \, \overline{\zeta}^{\beta^\top} \in \cH\langle \langle \overline{\zeta}
\rangle \rangle.
$$
For $f \in \cH$ and $y \in \cY$ we can therefore compute the pairing
\begin{align*}
\langle f, H(\zeta)^* y \rangle_{\cH \times \cH\langle \langle \overline{\zeta} \rangle \rangle}
= \sum_{\beta\in\free} \langle f, \Phi_\beta^* y \rangle_\cH \zeta^\beta
& = \sum_{\beta\in\free} \big\langle \Phi_\beta f, y \big\rangle_\cY  \zeta^\beta \\
&=
\Big{\langle}  \sum_{\beta\in\free} f_\beta  \zeta^\beta, y \Big{\rangle}_{\cY \langle \langle \zeta \rangle \rangle \times \cY}
\end{align*}
where $\Phi_\beta^* y \in \cH$ can be written more explicitly as the power series in $z$:
$$
 (\Phi_\beta^* y)(z)  = \sum_{\alpha\in\free} (\Phi_\alpha \Phi_\beta^* y) z^\alpha.
$$
If we now define $K(z, \zeta) \in \cL(\cY \langle \langle z, \overline{\zeta} \rangle \rangle$ by
$$
K(z, \zeta)  = \sum_{\alpha, \beta\in\free} \Phi_\alpha \Phi_\beta^* z^\alpha \overline{\zeta}^{\beta^\top}
= H(z) H(\zeta)^*,
$$
then we observe that $K$ has the formal reproducing property for the NFRKHS $\cH$:
\begin{itemize}

\item[(i)]  for each $y \in \cY$ we have $K(\cdot, \zeta) y \in \cH\langle \langle \overline{\zeta} \rangle \rangle$, and

\item[(ii)] for each $y \in \cY$ and $f \in \cH$ we have
$
  \langle f, K(\cdot, \zeta) y  \rangle_{\cH \times \cH\langle \langle \overline{\zeta} \rangle \rangle}
  = \langle f(\zeta), y \rangle_{\cY\langle\langle\zeta\rangle\rangle \times \cY}.$
\end{itemize}
In general, when $K$ and $\cH$ are related in this way,
we say that $K$ is the reproducing kernel for the FNRKHS $\cH$ and we write $\cH = \cH(K)$.

\begin{example}\label{E:szego}
{\rm The Fock space $H^{2}_{\cY}(\free)$ introduced in \eqref{Fock} is the NFRKHS with reproducing kernel
$k_{\rm Sz} I_{\cY}$ where $k_{\rm Sz}$ is the {\em noncommutative Szeg\H{o} kernel}
\begin{equation}
    k_{\rm Sz}(z,\zeta) =\sum_{\alpha \in \free} z^{\alpha}
    \bzeta^{\alpha^{\top}}.
\label{sz2}
\end{equation}
Indeed, for $f(z) = \sum_{\alpha \in \free}
f_{\alpha} z^{\alpha} \in H^2_{\cY}(\free)$ and $y \in \cY$, we have
\begin{align*}
    \langle f, k_{{\rm Sz}}(\cdot, \zeta) y \rangle_{H^2_{\cY}(\free) \times
    H^2_{\cY}(\free)\langle \langle \bzeta \rangle \rangle} &
= \sum_{\alpha \in \free} \langle f(z), \,  y
\, z^{\alpha} \rangle_{H^2_{\cY}(\free)} \zeta^{\alpha} \\
    & = \sum_{\alpha \in \free} \langle f_{\alpha}, y \rangle_{\cY} \zeta^{\alpha}
 = \langle f(\zeta), y \rangle_{\cY\langle \langle \zeta \rangle \rangle
    \times \cY}.
\end{align*}}
\end{example}

Given two formal positive kernels $K$ and $K'$  with coefficients in
$\cL(\cY)$ and $\cL(\cU)$ respectively, together with the associated NFRKHSs
$\cH(K)$ and $\cH(K')$, a formal power series $F\in\cL(\cU,\cY)\langle \langle z\rangle\rangle$
is called a {\em contractive multiplier from  $\cH(K')$ to $\cH(K)$}
if the left multiplication operator $M_F$ defined as in \eqref{mo} is a contraction from $\cH(K')$ to $\cH(K)$.
For details of the proof of the next result, we refer to Proposition 3.1.1 in \cite{bbbook}.

\begin{prop}
Let $K \in \cL(\cY)\langle \langle z, \bzeta \rangle \rangle$ and $K' \in \cL(\cU)\langle \langle z,
\bzeta \rangle \rangle$ be two positive formal kernels.
\begin{enumerate}
\item A formal power series
$F\in\cL(\cU,\cY)\langle \langle z\rangle\rangle$ is a contractive multiplier from $\cH(K')$ to $\cH(K)$ if and only if
$$
K_F(z,\zeta)=K(z,\zeta)-F(z)K'(z,\zeta)F(\zeta)^*\in\cL(\cY)\langle \langle z, \bzeta\rangle \rangle
$$
is a positive formal kernel.
\item $F$ is a coisometric multiplier from $\cH(K')$ to $\cH(K)$ if and only if $K_F(z,\zeta)=0$.
\end{enumerate}
\label{P:cmult}
\end{prop}
Specializing part (1) to the case where $K=k_{\rm Sz} I_{\cY}$ and $K'=k_{\rm Sz} I_{\cU}$ gives the characterization
of contractive multipliers from $\cH_\cU(k_{\rm Sz})=\cH(k_{\rm Sz} I_{\cU})=H^2_{\cU}(\free)$ to $\cH_\cY(k_{\rm Sz})=\cH(k_{\rm Sz} I_{\cY})=H^2_{\cY}(\free)$
(i.e., the elements in $\mathcal S_{{\rm nc},d}(\cU,\cY)$) in terms of the positive formal kernels:
{\em a  formal power series $S$ belongs to the
Schur class $\mathcal S_{{\rm nc},d}(\cU,\cY)$ if and only if
 \begin{equation}  \label{KS}
K_{S}(z, \zeta) : = k_{\rm Sz}(z,\zeta)I_{\cY} - S(z)( k_{\rm Sz}(z,\zeta)I_{\cU})S(\zeta)^{*}
       \end{equation}
is a positive formal kernel}.

\smallskip

One can show in much the same way as for the classical case that this FNRKHS $\cH(K_S)$ with formal
reproducing kernel \eqref{KS} can also be viewed as the de Branges-Rovnyak
space $\cH(S) =  \Ran (I_{H^2_\cY(\free)}-M_S M_S^*)^{\frac{1}{2}}$ with lifted norm
given by \eqref{1.15}.

\subsection{Indefinite noncommutative Schur class}

In order to describe the solutions to our interpolation problem ${\bf OAP}_{\cH(S_0)}$ we require Schur-class multipliers with respect to indefinite inner product spaces associated with two signature operators, which we define next. This subsection is partially a review of material from \cite{bbkats} and we refer there for further details and proofs. See the end of the introduction and Appendix \ref{S:Krein} for more on indefinite inner product spaces, signature operators and the terminology used here. We are specifically interested in the case where the signature operators have the form
\begin{equation}\label{JYU}
J_{_{\cY,\cU}}=\mat{I_\cY&0\\0& - I_\cU}\quad\mbox{or}\quad
{\bf J}_{_{\cY,\cU}}:= I_{H^2(\free)}\otimes J_{_{\cY,\cU}},
\end{equation}
for Hilbert spaces $\cY$ and $\cU$.

\begin{defn}  \label{D:1}
{\rm Given coefficient Hilbert spaces $\cU$, $\cY$, $\cF$, the {\em noncommutative indefinite Schur
class} $\cS_{{\rm nc},d}(J_{_{\cY,\cU}},J_{_{\cF,\cU}})$ consists of formal power series
%\begin{equation}
$$
{\mathfrak A}(z) = \begin{bmatrix} {\mathfrak A}_{11}(z) &{\mathfrak A}_{12}(z) \\ {\mathfrak A}_{21}(z) & {\mathfrak A}_{22}(z)
\end{bmatrix}\in\cL\left(\begin{bmatrix}\cF \\ \cU\end{bmatrix}, \, \begin{bmatrix}\cY \\ \cU\end{bmatrix}\right)\langle\langle z\rangle\rangle
$$
%\label{block}\end{equation}
such that the multiplication operator
$$
M_{\mathfrak A}: \,  (H^2_{\cF\oplus\cU}(\free),J_{_{\cF,\cU}})\to (H^2_{\cY\oplus\cU}(\free),J_{_{\cY,\cU}})
$$
is a $({\bf J}_{_{\cF,\cU}},{\bf J}_{_{\cY,\cU}})$-bi-contraction:
\begin{equation}\label{biconSchur}
M_{\mathfrak A}^*{\bf J}_{_{\cY,\cU}}M_{\mathfrak A}\preceq {\bf J}_{_{\cF,\cU}}\quad\mbox{and}\quad
M_{\mathfrak A}{\bf J}_{_{\cF,\cU}}M_{\mathfrak A}^*\preceq {\bf J}_{_{\cY,\cU}}.
\end{equation}}
\end{defn}

As the action of the operator
$M^*_{\mathfrak A}: \, H^2_{\cY\oplus\cU}(\free)\to H^2_{\cF\oplus\cU}(\free)$
on the kernel elements $k_{\rm Sz}(\cdot,\zeta)\sbm {y \\ u}$, for $y\in\cY$, $u\in\cU$, is given by
the formula
$$
M^*_{{\mathfrak A}}k_{\rm Sz}(\cdot,\zeta)\sbm {y \\ u}=k_{\rm Sz}(\cdot,\zeta)\mathfrak A(\zeta)^*\sbm {y \\ u},
$$
(see the computation leading up to formula (3.3) in \cite{bbbook}),
it follows that
\begin{align*}
{\bf J}_{_{\cY,\cU}}-M_{\mathfrak A}{\bf J}_{_{\cF,\cU}}M_{\mathfrak A}^*: \, k_{\rm Sz}(\cdot,\zeta)\sbm {y \\ u}
\mapsto & k_{\rm Sz}(\cdot,\zeta) J_{_{\cY,\cU}}\sbm {y \\ u}\notag\\
&\quad -{\mathfrak A}(\cdot)k_{\rm Sz}(\cdot,\zeta) J_{_{\cF,\cU}}
{\mathfrak A}(\zeta)^{*} \sbm {y \\ u}.\notag
%\label{may17}
\end{align*}
By linearity, the latter formula extends to linear combinations of kernel elements. Since the span of the
kernel elements is dense in $H^2_{\cY \oplus \cU}(\free)$,  we then see that
the second condition in \eqref{biconSchur} (i.e.,
${\bf J}_{_{\cY,\cU}}-M_{\mathfrak A}{\bf J}_{_{\cF,\cU}}M_{\mathfrak A}^*\succeq 0$)
is equivalent to the formal kernel
        \begin{equation}  \label{jKS}
        K^{J_{_{\cF,\cU}},J_{_{\cY,\cU}}}_{{\mathfrak A}}(z, \zeta): =
 k_{\rm Sz}(z,\zeta)J_{_{\cY,\cU}} -{\mathfrak A}(z)( k_{\rm Sz}(z,\zeta)J_{_{\cF,\cU}})
{\mathfrak A}(\zeta)^{*}
        \end{equation}
being positive. Therefore,  the associated de Branges-Rovnyak space
$$
\cH(K^{J_{_{\cF,\cU}},J_{_{\cY,\cU}}}_{{\mathfrak A}})
$$
is a Hilbert space for any formal power series $\mathfrak A$ satisfying just the second condition in \eqref{biconSchur}.

To handle the first condition in \eqref{biconSchur}, we use a duality trick as follows.
To test for positivity of the operator $\bJ_{_{\cF, \cU}} -  M_\fA^* \bJ_{_{\cY, \cU}} M_\fA$, we let it act from the right on the space
$H^{2}_{\cF' \oplus \cU'}(\free)$  consisting of formal power series with coefficients being vectors in the space $\cF' \oplus \cU'$
of linear functionals acting on $\cF \oplus \cU$ and resulting values in $H^2_{\cY' \oplus \cU'}$.
\footnote{This construct of letting an operator-valued function have both a right action on dual (row) vectors as
well as a left action on (column) vectors comes up in the duality pairing between dual vector bundles used for the construction
of an analogue of the Cauchy kernel for a vector bundle over an algebraic curve
with a given determinantal representation  (see \cite[Proposition 2.2]{BV96} and \cite[formula (5.1)]{BV99}). }
Again the operator $M_\fA^*$
can be computed when acting on the right
on a kernel element $k_{\rm Sz}( \cdot, \zeta) \sbm{ f' & u'}$, namely
$$
k_{\rm Sz} (\cdot, \zeta) \begin{bmatrix} f' & u' \end{bmatrix} M_{\fA}^* = k_{\rm Sz} (\cdot, \zeta)
\begin{bmatrix}  f' & u' \end{bmatrix} \fA(\zeta)^*.
$$
Hence we  can  compute the operator explicitly on these dual kernel functions:
\begin{align*}
& k_{\rm Sz} (\cdot, \zeta) \begin{bmatrix}  f' & u' \end{bmatrix} (\bJ_{_{\cF, \cU}}-  M_\fA^* \bJ_{_{\cY, \cU}} M_\fA)=\\
&\qquad\qquad= k_{\rm Sz} (\cdot, \zeta)  \begin{bmatrix}  f' & u' \end{bmatrix} J_{_{\cF, \cU}} -
 k_{\rm Sz}(\cdot, \zeta) \begin{bmatrix}  f' & u' \end{bmatrix} \fA(\zeta)^* J_{_{\cY, \cU}} \fA( \cdot)\\
&\qquad\qquad= \begin{bmatrix}  f' & u' \end{bmatrix} \left( k_{\rm Sz} (\cdot, \zeta)   J_{_{\cF, \cU}} -
   \fA(\zeta)^* k_{\rm Sz}(\cdot, \zeta) J_{_{\cY, \cU}} \fA( \cdot)\right).
\end{align*}
By similar arguments as for the second condition in \eqref{biconSchur} it now follows that the operator-positivity condition $\bJ_{_{\cF, \cU}} -  M_\fA^* \bJ_{_{\cY, \cU}} M_\fA \succeq 0$ is equivalent to the kernel
\begin{equation}   \label{jKS-dual}
\widetilde K_\fA^{J_{_{\cY, \cU}}, J_{_{\cF, \cU}}}(\zeta, z) =
k_{\rm Sz}(z, \zeta) J_{_{\cF, \cU}} - \fA(\zeta)^* (k_{\rm Sz}(z, \zeta) J_{_{\cY, \cU}}) \fA(z)
\end{equation}
being positive. In conclusion, we have derived the following result.

\begin{prop}\label{P:IdefSchurKernelChar}
A given power series $\fA(z) \in \cL(\cF \oplus \cU, \cY \oplus \cU)\langle \langle z, \overline{\zeta} \rangle\rangle$ is in the indefinite Schur class $\cS_{J_{_{\cY, \cU}}, J_{_{\cF, \cU}}}$ if and only if both formal kernels \eqref{jKS} and \eqref{jKS-dual} are positive.
\end{prop}

\subsection{Linear-fractional maps associated with indefinite noncommutative Schur-class multipliers}
\label{S:LFTnc}

We now describe how the indefinite noncommutative Schur-class multipliers described in the previous section can be used to
define linear-fractional maps acting on a standard Schur class of multipliers. In the finite-dimensional, one-variable case, many of
the arguments given here can also be found in \cite{AD08}, while for the infinite-dimensional case some pieces of these results
can also be derived by applying the Potapov-Ginsburg transform to derive the corresponding results in the Redheffer
(scattering) formalism (see \cite[Theorem 1.3.4]{drrov} and \cite[Theorem 3.4]{bbkats}).

We start in a slightly more general setting. Given a block $2 \times 2$ matrix formal power series
\begin{equation}\label{fA-LFT}
\fA(z)  = \begin{bmatrix} \fA_{11}(z) & \fA_{12}(z) \\
\fA_{21}(z) & \fA_{22}(z) \end{bmatrix}\in\cL\left(\mat{\cF\\ \cU},\mat{\cY\\ \cU}\right)\langle\langle z\rangle\rangle
\end{equation}
and a formal power series $\cE(z)\in \cL(\cU,\cF)\langle\langle z\rangle\rangle$,  let us say that $\cE$ is in the domain of the linear fractional map $T_\fA$ if it is the case that $\fA_{21}(z) \cE(z) + \fA_{22}(z) $ is invertible as a formal power series, i.e., its coefficient with index $\emptyset$ is invertible in $\cL(\cU)$, in which case we define $\fT_\fA[\cE]$ to be the formal power series given by
$$
\fT_\fA[\cE] = (\fA_{11} \cE + \fA_{12}) (\fA_{21} \cE + \fA_{22})^{-1}\in \cL(\cU,\cY)\langle\langle z\rangle\rangle.
$$
For this discussion the following definition will be useful.

\begin{defn}  \label{D:domLFT}
Let us say that the formal power series $\cE$ is in the {\em formal domain of $\fT_\fA$} if at least $\fA_{21} \cE + \fA_{22}$ is invertible as a formal power series, thereby making  $\fT_\fA[\cE]$ well defined as a formal power series. In case $\fA$ is a multiplier,  we say that $\cE$ is in the {\em multiplier domain of $\fT_\fA$} if $\cE$ is a multiplier and $\fA_{21} \cE + \fA_{22}$ is invertible as a multiplier on $H^2_\cU(\free)$.
 \end{defn}

Note that  for the case where $\fA$ is a multiplier, the multiplier domain of $\fT_\fA$ is contained in its formal domain,
since any invertible multiplier is  the multiplier arising from an invertible formal series.

In case $\fA$ is an indefinite noncommutative Schur-class multiplier we get the following result.

\begin{thm} \label{T:lfts}
If $\mathfrak A\in \cS_{{\rm nc},d}(J_{_{\cY,\cU}},J_{_{\cF,\cU}})$ and
${\mathcal E}\in \cS_{{\rm nc},d}(\cU,\cF)$, then the linear fractional transform of $\cE$
\begin{equation}
         \fT_{{\mathfrak A}}[{\mathcal E}]= ({\mathfrak A}_{11}{\mathcal E} +
         {\mathfrak A}_{12}) ( {\mathfrak A}_{21} {\mathcal E} +{\mathfrak A}_{22})^{-1}
\label{lfm}
\end{equation}
is a well-defined element of $\cS_{{\rm nc},d}(\cU,\cY)$.  In particular, if $\fA \in \cS_{{\rm nc},d}(J_{_{\cY,\cU}},J_{_{\cF,\cU}})$
and $\cE \in \cS_{{\rm nc},d}(\cU,\cF)$, then $\cE$ is in the multiplier domain of $\fT_\fA$.
\end{thm}

\begin{proof}
A consequence of the second relation in \eqref{biconSchur} is that
\begin{equation}   \label{ineq1}
   M_{\fA_{21}} M_{\fA_{21}}^* - M_{\fA_{22}} M_{\fA_{22}}^* \preceq -I_{H^2_\cU(\free)}
 \end{equation}
which we can rearrange to
$$
 I_{H^2_\cU(\free)} \preceq I_{H^2_\cU(\free)} + M_{\fA_{21}} M_{\fA_{21}}^* \preceq M_{\fA_{22}} M_{\fA_{22}}^*
 $$
from which we conclude that $M_{\fA_{22}}$ is surjective.  From the first relation in \eqref{biconSchur} we get that
$$
M_{\fA_{12}}^* M_{\fA_{12}} - M_{\fA_{22}}^* M_{\fA_{22}} \preceq - I_{H^2_\cU(\free)}
$$
which by parallel rearrangements as in the preceding argument leads to
$$
I_{H^2_\cU(\free)} \preceq I_{H^2_\cU(\free)} + M_{\fA_{12}}^* M_{\fA_{12}}  \preceq M_{\fA_{22}}^* M_{\fA_{22}}
$$
leading to the conclusion that $M_{\fA_{22}}$ is also injective.  We conclude that when both inequalities in \eqref{biconSchur} hold,
it follows that $M_{\fA_{22}}$ is bijective with bounded inverse $M_{\fA_{22}}^{-1}$ given by
$M_{\fA_{22}}^{-1} = M_{\fA_{22}^{-1}}$ (due to the characterization \eqref{CharMult} of multiplier operators).
Conjugating \eqref{ineq1} by $M_{\fA_{22}}^{-1}$ then gives us
 $$
 M_{\fA_{22}^{-1} \fA_{21}} (M_{\fA_{22}^{-1} \fA_{21}} )^* + M_{\fA_{22}^{-1}} (M_{\fA_{22}^{-1}})^*
 \preceq I_{H^2_\cU(\free)}
 $$
 from which we  conclude that
$\|M_{{\mathfrak A}_{22}^{-1}{\mathfrak A}_{21}}\|<1$ and hence also
$$
\| M_{{\mathfrak A}_{22}^{-1}{\mathfrak A}_{21}\cE} \|<1\quad\mbox{for any}\quad \cE\in\cS_{{\rm nc},d}(\cU,\cF).
$$
Then
\[
I_{H^2_\cU(\free)}+M_{{\mathfrak A}_{22}^{-1}{\mathfrak A}_{21}\cE}=M_{I_\cU+{\mathfrak A}_{22}^{-1}{\mathfrak A}_{21}\cE}
=M_{{\mathfrak A}_{22}^{-1}} M_{{\mathfrak A}_{22}+{\mathfrak A}_{21}\cE}
\]
is invertible on $H^2_\cU(\free)$, so that $M_{{\mathfrak A}_{22}+{\mathfrak A}_{21}\cE}$ is also
invertible on $H^2_\cU(\free)$. Hence ${\mathfrak A}_{22}+{\mathfrak A}_{21}\cE$ invertible as a multiplier on $H^2_\cU(\free)$,
making $\cE$ in the multiplier domain of $\fT_\fA$. Since a formal power series is invertible as a formal power series whenever it is
invertible as a multiplier, it follows that $\cE$ is also in the formal domain.

This proves that any $\cE$ in the Schur class $\cS_{{\rm nc}, d}(\cU, \cF)$ is in the multiplier domain for the linear-fractional map $\fT_\fA$
and hence $\fT_\fA[\cE]$ is a bounded multiplier from $H^2_\cU(\free)$ to $H^2_\cY(\free)$.
To see that
$\fT_\fA[\cE]$ is in $\cS_{{\rm nc},d}(\cU,\cY)$, we show that $\fT_\fA[\cE]$ so defined has multiplier norm at most $1$.
To this end, multiply the first inequality of \eqref{biconSchur} with $\mat{M_\cE^* & I}$ from the left and with $\mat{M_\cE^* & I}^*$
from the right to see that
\begin{align*}
\mat{M_\cE^* & I} M_\fA^* {\bf J}_{_{\cY,\cU}} M_\fA \mat{ M_\cE\\ I}  & \preceq \mat{M_\cE^* & I} {\bf J}_{_{\cF,\cU}} \mat{ M_\cE\\ I}\\
&= M_\cE^* M_\cE -I\preceq 0.
\end{align*}
By writing out $M_\fA$ as a $2 \times 2$ block-multiplier, we see that
\begin{align*}
0 &\succeq \mat{M_{\fA_{11}\cE+\fA_{12}}^* & M_{\fA_{21}\cE+\fA_{22}}^*}  {\bf J}_{_{\cY,\cU}} \mat{M_{\fA_{11}\cE+\fA_{12}}\\ M_{\fA_{21}\cE+\fA_{22}}}\\
&=M_{\fA_{11}\cE+\fA_{12}}^* M_{\fA_{11}\cE+\fA_{12}} -  M_{\fA_{21}\cE+\fA_{22}}^*  M_{\fA_{21}\cE+\fA_{22}}\\
&=M_{\fA_{21}\cE+\fA_{22}}(M_{\fT_\fA[\cE]}^*M_{\fT_\fA[\cE]}-I)M_{\fA_{21}\cE+\fA_{22}}.
\end{align*}
Since $M_{\fA_{21}\cE+\fA_{22}}$ is invertible, it follows that $M_{\fT_\fA[\cE]}^*M_{\fT_\fA[\cE]}-I\preceq 0$, so that
$\|M_{\fT_\fA[\cE]}\|\leq 1$ and thus $\fT_\fA[\cE]$ is in  $\cS_{{\rm nc},d}(\cU,\cY)$, as claimed.
\end{proof}

\subsubsection{Injectivity (or lack thereof) of linear-fractional maps}
An important feature used in the solution to the one-variable interpolation problem of Section \ref{S:SingleVar} is that, under certain conditions, in the relation $S=\fT_\fA[\cE]$ the Schur-class multiplier $\cE$ is uniquely determined by $S$, that is, the linear fractional map $\fT_\fA$ is injective (see e.g.~\cite[Section 4.1]{bbt3}).
 We now investigate at least one scenario when this injectivity property holds for $\fT_\fA$ in the multivariable setting of the present section, starting with the general setting where $\fA$ is as in \eqref{fA-LFT}.

\begin{thm}\label{T:LFTinjective}
Let $\fA$ be a formal power series as in \eqref{fA-LFT} and assume  that $\fA_{22}$ is invertible as a formal power series.
For $\cE$  in the formal domain of $\fT_\fA$,  let $S$ be the formal power series given by $S = \fT_\fA[\cE]$.
Assume also that $\fA$ is invertible as a formal power series, or equivalently (via Schur complement theory), that $\fA_{11} -
\fA_{12} \fA_{22}^{-1} \fA_{21}$ is invertible as a formal power series.  Then
$$
 \fA_{11} - S \fA_{21} = ( \fA_{11} - \fA_{12} \fA_{22}^{-1} \fA_{21}) (I + \cE \fA_{22}^{-1} \fA_{21})^{-1}
$$
is invertible as a formal power series and we recover the formal power series $\cE$ from the formal power series $S$
according to the formula
\begin{equation}  \label{recoverE}
\cE = (\fA_{11} - S \fA_{21})^{-1} ( S \fA_{22} - \fA_{12}).
\end{equation}
\end{thm}

 \begin{proof}
Let us suppose that $\cE$ is in the formal domain of $\fT_\fA$, set $S:= \fT_\fA[\cE]$ and try to solve for $\cE$ in terms of $S$.
From the definition of $\fT_\fA[\cE]$ we find that
$$
 S(\fA_{21} \cE + \fA_{22}) = \fA_{11} \cE + \fA_{12}
$$
which we rearrange as
$$
(\fA_{11} - S \fA_{21}) \cE = S \fA_{22} - \fA_{12}.
$$
If $\fA_{11} - S \fA_{21}$ is invertible, then we can solve for $\cE$:
$$
 \cE = ( \fA_{11} - S \fA_{21})^{-1} ( S \fA_{22} - \fA_{12})
$$
and we have a formula for the inverse linear-fractional map.  Let us investigate when is $\fA_{11} - S \fA_{21}$ invertible under the assumption that $S = \fT_\fA[\cE]$ and that $\fA_{22}$ is invertible, as we have in the indefinite Schur-class multiplier case. Then we have
\begin{align*}
 \fA_{11} - S \fA_{21} &  = \fA_{11} - ( \fA_{11} \cE + \fA_{12}) (\fA_{21} \cE + \fA_{22})^{-1} \fA_{21} \\
 & = \fA_{11} - (\fA_{11} \cE + \fA_{12}) (I + \fA_{22}^{-1} \fA_{21}\cE)^{-1} \fA_{22}^{-1} \fA_{21}  \\
 & = \fA_{11} - (\fA_{11} \cE + \fA_{12}) \fA_{22}^{-1} \fA_{21} (I + \cE \fA_{22}^{-1} \fA_{21})^{-1} \\
 & = [ \fA_{11} + \fA_{11} \cE \fA_{22}^{-1} \fA_{21} - \fA_{11} \cE \fA_{22}^{-1} \fA_{21}- \fA_{12} \fA_{22}^{-1} \fA_{21} ]
 (I + \cE \fA_{22} \fA_{21})^{-1}   \\
 & = (\fA_{11} - \fA_{12} \fA_{22}^{-1} \fA_{21}) (I + \cE \fA_{22}^{-1} \fA_{21})^{-1}.
\end{align*}
We have that  $I + \cE \fA_{22}^{-1} \fA_{21}$  is indeed formal-power-series invertible, since $\cE$ is in the formal
domain of $\fT_\fA$. It thus follows that $\fA_{11} - S \fA_{21}$ is formal-power-series invertible if and only if the
Schur complement of $\fA$ with respect to
$\fA_{22}$, namely, $\fA_{11} - \fA_{12} \fA_{22}^{-1} \fA_{21}$, is invertible. By standard Schur complement theory, the latter is equivalent to $\fA$ being formal-power-series invertible.
\end{proof}

\begin{rem}  \label{R:LFTinjective}
{\em Let us now suppose that $\fA$ as in \eqref{fA-LFT} is a multiplier and $\cE$ is in the multiplier domain of $\fT_\fA$
(in the sense of Definition \ref{D:domLFT}).  Assume that the multipliers $\fA$ and $\fA_{22}$ are at least formal-power-series
invertible (and thus also the Schur complement $\fA_{11} - \fA_{12} \fA_{22}^{-1} \fA_{21}$ is formal-power-series invertible).
Let us assume that $\cE$ is in the multiplier domain of $\fT_\fA$, so $S = \fT_\fA[\cE]$ is well-defined as a multiplier.
As any multiplier can be viewed as arising from a formal power series, Theorem \ref{T:LFTinjective} applies to this situation
and we conclude that we can recover $\cE$ from $S$ according to the formula \eqref{recoverE}.  From inspection of the formula we only
see that $\cE$ is a formal power series despite the fact that $\cE$ was originally chosen to be a multiplier.
}\end{rem}

Simple examples show that the invertibility assumption in  Theorem \ref{T:LFTinjective} is too strong in general  for getting
injectivity results for the associated  linear-fractional map, as illustrated in the next example

\begin{ex}  \label{E:nc-inj}
The degenerate case where one takes the input space $\cU$ appearing in the definition \eqref{fA-LFT} for the symbol
$\fA$ for the linear-fractional map $\fT_\fA$ to be $\cU = \{0\}$ is of interest as a special case.
In this case $\fA$ collapses to $\fA = \fA_{11}$, $\fA_{22}$ is invertible vacuously, the Schur complement
$\fA_{11} - \fA_{12} \fA_{22}^{-1} \fA_{21}$ collapses to $\fA$ itself, and the linear-fractional map $\fT_\fA$
collapses to the multiplication operator $M_{\fA_{11}} = M_{\fA} \colon f(z) \mapsto \fA(z) f(z)$.  As an example let us consider
the case where $\fA(z) = \fA_{11}(z) = \begin{bmatrix} z_1 & \cdots & z_d \end{bmatrix}$ and
$\fT_\fA = M_\fA$ is the multiplication operator
\begin{equation}   \label{nc-row-ex}
M_\fA \colon \begin{bmatrix} g_1(z) \\ \vdots \\ g_d(z) \end{bmatrix} \mapsto z_1 g_1(z) + \cdots + z_d g_d(z).
\end{equation}
A distinctive feature of the noncommutative Fock space setting is that this $M_\fA = M_{Z(\cdot)}$ with
$Z(z) = \begin{bmatrix} z_1 & \cdots & z_d \end{bmatrix}$ is injective, i.e.,  given a formal power series of the form
$f(z) = z_1 g_1(z) +  \cdots + z_d g_d(z)$, one can uniquely backsolve for the power series $g_1(z), \dots, g_d(z)$
due to the fact that we have the orthogonal direct-sum decomposition
$$
  H^2(\free) = {\mathbb C} \oplus \bigoplus_{j=1}^d z_j H^2(\free).
$$
Note that the range of $M_{Z(\cdot)}$ is the proper subspace $\cM$ consisting of elements of $H^2(\free)$
with $\emptyset$-coefficient $f_\emptyset$ equal to zero, and hence  $M_{Z(\cdot)}$ is not surjective.
In fact $Z(\cdot)$ amounts to the Beurling-Lax representer in the sense of Popescu \cite{popescu89} for the (proper)
right-shift-invariant subspace $\cM$.  We shall discuss this example more generally in Theorem \ref{BLPop} below.
\end{ex}

\subsubsection{Construction of indefinite Schur-class multipliers}

We next present the analogue of \eqref{iden} in the setting of the present section.  Closely related results
appear in \cite{bbkats}, specifically Theorems 3.4 and 3.7 there.  Here we clarify the role of the second kernel
\eqref{jKS-dual} and use positivity of this kernel to avoid imposing as an assumption that the multiplier  $M_{\fA_{22}}$ is invertible.

\begin{thm}  \label{T:4.1}
Given output stable pairs $(E,{\bf T})$ and $(N,{\bf T})$ with $E\in\cL(\cX,\cY)$, $N\in\cL(\cX,\cU)$ and $\bT=(T_1,\ldots,T_d)\in\cL(\cX)^d$ and given an operator $P\succ 0$ that satisfies the Stein identity \eqref{2.8}, set
$C  := \sbm{ E \\ N } \colon \cX \to \sbm{ \cY \\ \cU }$ and define $T$ as in \eqref{1.6a}. Then we have:
\begin{enumerate}
\item
There exists a Hilbert space $\cF$ so that the $J$-Cholesky factorization problem
\begin{equation}\label{10a}
\begin{bmatrix} B \\ D \end{bmatrix}J_{_{\cF,\cU}}\begin{bmatrix}
B^{*} & D^{*}\end{bmatrix}=\begin{bmatrix} P^{-1}\otimes I_d & 0 \\0 &
J_{_{\cY,\cU}} \end{bmatrix}-\begin{bmatrix} T \\ C\end{bmatrix}P^{-1}\begin{bmatrix} T^{*} & C^*\end{bmatrix}
\end{equation}
has an injective solution $\sbm{ B \\ D } \colon \cF \oplus \cU \to \sbm{ \cX^d \\ \cY \oplus \cU}$.

\item
The operator $\sbm{ T \\ C }$ is a $\big(P, \sbm{ P \otimes I_d & 0 \\ 0 & J_{\cY, \cU}} \big)$-isometry, and  for any injective solution $\sbm{ B \\ D }$ to the $J$-Cholesky factorization problem \eqref{10a}, the operator
\begin{equation} \label{bU}
\bU = \begin{bmatrix} T & B \\ C & D \end{bmatrix} \colon \begin{bmatrix} \cX \\ \cF \oplus \cU \end{bmatrix} \to \begin{bmatrix} \cX^d \\ \cY \oplus \cU \end{bmatrix}
\end{equation}
is a $\big( \sbm{ P & 0 \\ 0 & J_{\cF, \cU} }, \sbm{ P \otimes I_d & 0 \\ 0 & J_{\cY, \cU}} \big)$-unitary completion $\sbm{ T \\ C }$,
and conversely, any such $\bU$ arises from an injective solution $\sbm{ B \\ D}$ of  \eqref{10a}.
Explicitly $\bU$ satisfies  the metric relations
\begin{align}
 & \bU^* \begin{bmatrix} P \otimes I_d & 0 \\ 0 & J_{\cY, \cU} \end{bmatrix} \bU  = \begin{bmatrix} P & 0 \\ 0 & J_{_{\cF, \cU}} \end{bmatrix},
 \label{bU=Jiso} \\
 & \bU \begin{bmatrix} P^{-1} & 0 \\ 0 & J_{\cF, \cU} \end{bmatrix}  \bU^* = \begin{bmatrix} P^{-1} \otimes I_d & 0 \\ 0 & J_{\cY, \cU}
\end{bmatrix}.
\label{bU=Jcoiso}
\end{align}

\item
For any injective solution $\sbm{ B \\ D }$ to the $J$-Cholesky factorization problem \eqref{10a}, if we set $\fA$ equal to the transfer function of the colligation matrix $\bU$ in \eqref{bU} having the form
 \begin{equation}
{\mathfrak A}(z)=D+C(I-Z(z)T)^{-1}Z(z)B,
        \label{2.29a}
\end{equation}
with $Z(z)$ as in \eqref{1.6a}, then the formal kernels $K^{J_{_{\cF,\cU}},J_{_{\cY,\cU}}}_{{\mathfrak A}}(z, \zeta)$ in \eqref{jKS} and
$\widetilde K^{J_{_{\cF,\cU}},J_{_{\cY,\cU}}}_{{\mathfrak A}}(z, \zeta)$ in \eqref{jKS-dual} associated
with $\fA$ have the explicit form
\begin{align}
    &  K^{J_{_{\cF,\cU}},J_{_{\cY,\cU}}}_{{\mathfrak A}}(z, \zeta)
 =\begin{bmatrix}E \\ N\end{bmatrix}(I - Z(z) T)^{-1} P^{-1}
          (I - T^{*} Z(\zeta)^{*})^{-1}\begin{bmatrix}E^* & N^*\end{bmatrix}, \label{2.30}  \\
&  \widetilde K^{J_{_{\cY,\cU}},J_{_{\cF,\cU}}}_{{\mathfrak A}}(z, \zeta)\label{2.30'}=\\
& \  = B^* (I - Z(\zeta)^* T^*)^{-1} \big( k_{\rm Sz}(z, \zeta) ( P \otimes I_d - Z(\zeta)^* P Z(z) ) \big)  (I - T Z(z))^{-1} B.\notag
  \end{align}
Both these formal kernels turn out to be positive kernels and hence $\fA$ is in the indefinite Schur class
$\cS_{{\rm nc},d}(J_{_{\cY,\cU}},J_{_{\cF,\cU}})$.
  \end{enumerate}
\end{thm}

\begin{proof}
Our proof relies on a general Kre\u{\i}n-space lemma which we prove in the appendix, namely Lemma \ref{L:findBD}.

The proof of points (1) and (2) is just an application of Lemma  \ref{L:findBD} with
\begin{align*}
& (\cX'_1, J'_1) := (\cX, P), \quad (\cX, J)  := \bigg( \begin{bmatrix} \cX^d \\ \cY \oplus \cU \end{bmatrix},
\begin{bmatrix} P \otimes I_d & 0 \\ 0 & J_{\cY, \cU} \end{bmatrix} \bigg),  \\
& V  := \mat{ T \\ C} \text{ with } C := \begin{bmatrix} E \\  N \end{bmatrix} \mbox{ and $T$ as in \eqref{1.6a}.}
\end{align*}
The requirement that $V$ be a $(J'_1,J)$-isometry then takes the form
$$
\sum_{j=1}^d T_j^* P T_j + E^* E - N^* N = P
$$
which is just a rewriting of the Stein equation \eqref{2.8}, which holds by assumption. Hence Lemma  \ref{L:findBD} indeed applies. The fact that the Kre\u{\i}n space $(\cX_2', J'_2)$ claimed to exist
in Lemma \ref{L:findBD} can be taken to be
$$
 (\cX'_2,  J'_2) = (\cF \oplus \cU, J_{\cF, \cU})
$$
follows by item (2) of Corollary \ref{C:J-UniCompl}, where we take $(\cX_1,J_1)=(\cX,P)$ and $(\cX_2,J_2)=\big(\sbm{\cX^{d-1}\\\cY\oplus\cU},\sbm{P\otimes I_{d-1}&0\\0&J_{\cY,\cU}}\big)$, noting that the $(J,J')$-unitary operator serves as a Kre\u{\i}n-space isomorphism.

Hence by Lemma \ref{L:findBD} the $( J_1', J )$-isometric operator $\sbm{ T \\ C }$ can be completed to a  $( J', J )$-unitary operator for $J'=\sbm{J_1'&0\\0&J_2'}$, which yields precisely an operator $\bU$ as in \eqref{bU} satisfying relations \eqref{bU=Jiso} and \eqref{bU=Jcoiso}. That this provides a solution to the $J$-Cholesky factorization problem \eqref{10a} as claimed in item (1) and the relation between solutions to the $J$-Cholesky factorization problem \eqref{10a} and $( J', J )$-unitary  completions of $\sbm{ T \\ C }$ now follows directly by the second claim of Lemma \ref{L:findBD}. So, we proved items (1) and (2) and it remains to prove item (3).

A straightforward computation based solely on the equality \eqref{bU=Jcoiso} shows
that the kernel $K^{J_{_{\cF,\cU}},J_{_{\cY,\cU}}}_{{\mathfrak A}}$ in \eqref{jKS} associated with the characteristic formal power
series of the colligation \eqref{bU} given by \eqref{2.29a}  can be factored as in \eqref{2.30}.

A similar computation based solely on the equality \eqref{bU=Jiso} shows that the kernel $\widetilde K^{J_{_{\cY,\cU}},J_{_{\cF,\cU}}}_{{\mathfrak A}}$ in \eqref{jKS'} associated with the same power series can be factored as in \eqref{2.30'} or equivalently, as
\begin{align*}
 \widetilde{K}^{J_{_{\cY,\cU}},J_{_{\cF,\cU}}}_{{\mathfrak A}}(z, \zeta)=&B^* (I - Z(\zeta)^* T^*)^{-1}(P^{\half}\otimes I_d)\\
&\times\big( k_{\rm Sz}(z, \zeta)I_{\cX^d}- Z(\zeta)^*(k_{\rm Sz}(z, \zeta)I_{\cX})Z(z) ) \big) \\
&\times(P^{\half}\otimes I_d)(I - T Z(z))^{-1} B.
\end{align*}
These computations are worked out in detail for the definite case (see Theorems 2.2 and 2.3 in \cite{BTV}).

From the factorization \eqref{2.30} for the first kernel, we see immediately that the first kernel
$K_\fA{^{J_{\cF, \cU}, J_{\cY, \cU}}}$ is a positive kernel.

To show that the second kernel \eqref{2.30'} is a positive kernel, it suffices to verify that the kernel
\begin{equation}
\mathfrak K(z,\zeta)=\left[\mathfrak K_{ij}(z,\zeta)\right]_{i,j=1}^d:= k_{\rm Sz}(z, \zeta)I_d-Z(\zeta)^*k_{\rm Sz}(z, \zeta)Z(z)
\label{easy2}
\end{equation}
is positive. Since
$$
k_{\rm Sz}(z, \zeta)-\sum_{j=1}^d \overline\zeta_j k_{\rm Sz}(z, \zeta)z_j=1,
$$
we have
\begin{equation}
\mathfrak K_{ij}(z,\zeta)=\begin{cases} 1+{\displaystyle\sum_{m\neq j} \overline\zeta_mk_{\rm Sz}(z, \zeta)z_m}, \; \; \mbox{for} \; \; i=j,\\
- \overline\zeta_i k_{\rm Sz}(z, \zeta)z_j, \; \; \mbox{for} \; \; i\neq j.\end{cases}
\label{easy}
\end{equation}
Let
$$
k_{\rm Sz}(z, \zeta)=H(z)H(\zeta)^*,\quad H(z)={\rm Row}_{\alpha\in\free}z^\alpha
$$
be the Kolmogorov decomposition of $k_{\rm Sz}$ and let $R_j:=R_{z_j}$ be the coordinate-variable multipliers from \eqref{intertw}.
Since
$$
(R_jH)(z)={\rm Row}_{\alpha\in\free}z^\alpha z_j\quad\mbox{for}\quad j=1,\ldots,d,
$$
we have
\begin{equation}
((R_jH)(z))(R_iH(\zeta))^*=\overline{\zeta_i}k_{\rm Sz}(z, \zeta)z_j \quad\mbox{for}\quad i,j=1,\ldots,d.
\label{easy1}
\end{equation}
Let ${\bf e}_j$ denote the $j$-th column of the identity matrix $I_d$, let
\begin{equation}
F_{ij}(z)={\bf e}_i(R_jH)(z)-{\bf e}_j(R_iH)(z)\quad\mbox{for}\quad 1\le i<j\le d.
\label{easy3}
\end{equation}
We claim that
\begin{equation}
\mathfrak K(z,\zeta)=I_d+\sum_{1\le i<j\le d}F_{ij}(z)F_{ij}(\zeta)^*, \label{easy4}
\end{equation}
which provides a Kolmogorov decomposition of $\mathfrak K(z,\zeta)$ and proves that $\mathfrak K(z,\zeta)$ is a positive kernel. Note first that, by \eqref{easy3} and \eqref{easy1}, we have
\begin{align}
F_{ij}(z)F_{ij}(\zeta)^*&=\big({\bf e}_i(R_jH)(z)-{\bf e}_j(R_iH)(z)\big((R_jH(\zeta))^*{\bf e}_i^*-(R_iH(\zeta))^*{\bf e}_j^*\big)\notag \\
&={\bf e}_i\overline{\zeta}_jk_{\rm Sz}(z, \zeta)z_j {\bf e}_i^*+{\bf e}_j\overline{\zeta}_ik_{\rm Sz}(z, \zeta)z_i{\bf e}_j^*\notag \\
&\quad -{\bf e}_i\overline{\zeta}_ik_{\rm Sz}(z, \zeta)z_j {\bf e}_j^*-{\bf e}_j\overline{\zeta}_jk_{\rm Sz}(z, \zeta)z_i {\bf e}_i^*.\label{easy5}
\end{align}

First we consider the diagonal entries of $\mathfrak K(z,\zeta)$. From the identity \eqref{easy5}, we  see that the $(\ell,\ell)$-entry in the matrix $F_{ij}(z)F_{ij}(\zeta)^*$ is non-zero only if either $i=\ell$ or $j=\ell$. Comparing the $\ell\ell$-entries in \eqref{easy4} gives
\begin{align}
1+\sum_{1\le i<j\le d}\left[F_{ij}(z)F_{ij}(\zeta)^*\right]_{\ell\ell}&
=1+\sum_{i<\ell}\left[F_{i\ell}(z)F_{i\ell}(\zeta)^*\right]_{\ell\ell}+\sum_{j>\ell}\left[F_{\ell j}(z)F_{\ell j}(\zeta)^*\right]_{\ell\ell}\notag \\
&=1+\sum_{j\neq\ell}\big[{\bf e}_j\overline{\zeta}_\ell k_{\rm Sz}(z, \zeta)z_\ell {\bf e}_j^*+
{\bf e}_\ell\overline{\zeta}_j k_{\rm Sz}(z, \zeta)z_j {\bf e}_\ell^*\notag \\
&\qquad\qquad-{\bf e}_j\overline{\zeta}_j k_{\rm Sz}(z, \zeta)z_\ell {\bf e}_\ell^*
-{\bf e}_\ell\overline{\zeta}_\ell k_{\rm Sz}(z, \zeta)z_j {\bf e}_j^*\big]_{\ell\ell}\notag \\
&=1+\sum_{j\neq \ell}\overline\zeta_jk_{\rm Sz}(z, \zeta)z_j.\notag
%\label{easy6}
\end{align}
Hence, we see that \eqref{easy4} holds on the diagonal.

For $\ell\neq r$ the $(\ell,r)$-entry in the right hand side of \eqref{easy4} works out as
$$
\sum_{1\le i<j\le d}\left[F_{ij}(z)F_{ij}(\zeta)^*\right]_{\ell r}=
\begin{cases} \left[F_{\ell r}(z)F_{\ell r}(\zeta)^*\right]_{\ell r} \; \;  \mbox{if} \; \; \ell<r, \\
\left[F_{r\ell }(z)F_{r\ell}(\zeta)^*\right]_{\ell r} \; \;  \mbox{if} \; \; r<\ell.\end{cases}
$$
In either case, according to \eqref{easy5}, we end up with $-\overline\zeta_\ell k_{\rm Sz}(z, \zeta)z_r$, which corresponds to the $(\ell,r)$-entry of $\mathfrak K(z,\zeta)$ according to \eqref{easy}. Hence we have established \eqref{easy4} and obtained that $\mathfrak K(z,\zeta)$ is a positive kernel, and thus that $ \widetilde{K}^{J_{_{\cY,\cU}},J_{_{\cF,\cU}}}_{{\mathfrak A}}(z, \zeta)$ is a positive kernel.

Since we now see that both kernels \eqref{2.30} and \eqref{2.30'} are positive, we conclude by Proposition
\ref{P:IdefSchurKernelChar} that ${\mathfrak A}(z)$ belongs to the  indefinite Schur class
$\cS_{{\rm nc},d}(J_{_{\cY,\cU}},J_{_{\cF,\cU}})$, and the  proof of item (3) is complete.
\end{proof}

\begin{rem}  \label{R:ind-proof}
{\rm If we let $\mathfrak R_m(z,\zeta)$ denote the $m\times m$-matrix valued kernel of the form \eqref{easy2} but in variables $z_1,\ldots,z_m$, i.e.,
$$
\mathfrak R_m(z,\zeta)=k_{{\rm Sz},m}(z, \zeta)I_m-Z_m(\zeta)^*k_{{\rm Sz},m}(z, \zeta)Z_m(z),
$$
then it turns out that
\begin{equation}
\mathfrak R_m(z,\zeta)-\begin{bmatrix}\mathfrak R_{m-1}(z,\zeta) & 0 \\ 0 & 1\end{bmatrix}
=G(z)G(\zeta)^*
\label{easy9}
\end{equation}
where
$$
G(z)=\begin{bmatrix}(R_m H)(z) & 0 &\ldots&0 \\ 0 & (R_mH)(z) &\ddots & \vdots\\ \vdots & \ddots & \ddots &0\\ 0 &\ldots & 0 &(R_mH)(z)\\
-(R_1 H)(z) & -(R_2H)(z)&\ldots &-(R_{m-1}H)(z)\end{bmatrix}.
$$
Here, $H$ denotes the power series from the Kolmogorov decomposition
$$
k_{{\rm Sz},m}(z, \zeta)=H(z)H(\zeta)^*,\quad H(z)={\rm Row}_{\alpha\in\mathbb F_m^+}z^\alpha
$$
and $R_j:=R_{z_j}$ are the coordinate-variable multipliers from \eqref{intertw}. Upon making use of \eqref{easy1}, we can
write the kernel $G(z)G(\zeta)^*$ in more detail as
$$
G(z)G(\zeta)^*=\begin{bmatrix}\overline{\zeta}_m k_{{\rm Sz},m}(z, \zeta)z_m I_{m-1}&
\begin{bmatrix}\overline{\zeta}_1 \\ \overline{\zeta}_2 \\ \vdots \\ \overline{\zeta}_{m-1}\end{bmatrix}k_{{\rm Sz},m}(z, \zeta)z_m\\
\overline{\zeta}_m k_{{\rm Sz},m}(z, \zeta)\begin{bmatrix}z_1 & z_2 &\ldots&z_{m-1}\end{bmatrix}&
{\displaystyle\sum_{j=1}^{m-1} \overline\zeta_jk_{{\rm Sz},m}(z, \zeta)z_j}\end{bmatrix}
$$
which along with formulas \eqref{easy} implies \eqref{easy9}. Note that by \eqref{easy9},
$$
\mathfrak R_m(z,\zeta)\succeq \begin{bmatrix}\mathfrak R_{m-1}(z,\zeta) & 0 \\ 0 & 1\end{bmatrix} \; \; \mbox{for all} \; \; m\ge 2,
$$
and since $\mathfrak R_1(z,\zeta)\equiv 1\succeq 0$, the positivity of the kernel $\mathfrak R_m$ for all $m\ge 1$ alternatively
follows by an induction argument.
}\end{rem}

In \cite{bbkats}, the series \eqref{2.29a} and the associated linear fractional transformation \eqref{lfm}
with free parameter ${\mathcal E}\in\cS_{{\rm nc},d}(\cU, \, \cF)$ were constructed to describe all
$S\in\cL(\cU,\cY)\langle\langle z\rangle\rangle$ solving the ${\bf OAP}_{\cS_{{\rm nc},d}(\cU, \, \cY)}$ with
interpolation condition
%\begin{equation}\label{oapschurnc}
$$
(E^{*} S)^{\wedge L}({\bf T}^{*})=\cO_{E,{\bf T}}^*M_S|_{\cU}=N^*
$$
%    \end{equation}
in case $P=\cO_{E,{\bf T}}^*\cO_{E,{\bf T}}-\cO_{N,T}^*\cO_{N,{\bf T}}\succ 0$.
Although the formal power series ${\mathfrak A}$ is still crucial in context of the ${\bf OAP}_{\cH(S_0)}$, the latter
Schur-class interpolation theorem can be by-passed as we will see in the next subsection.

\subsection{${\bf OAP}_{\cH(S_0)}$: parametrization of the solution set}
\label{S:param-ncSchur}

Throughout this subsection we assume that the operator $P$ given in \eqref{1.8**} is strictly positive definite. With all the above preliminaries out of the way, we can now describe the solution set for the interpolation problem ${\bf OAP}_{\cH(S_0)}$. We start with a specific solution.

\begin{prop}\label{P:4.2}
The power series
\begin{equation}
f_0(z)=(E-S_0(z)N)(I_\cX-Z(z)T)^{-1}P^{-1}{\bf x}
\label{fonc}
\end{equation}
is a solution to the ${\bf OAP}_{\cH(S_0)}$.
Furthermore, we have $\|f_0\|_{\cH(S_0)}=\|P^{-\half} {\bf x}\|_{\cX}$.
\end{prop}
\begin{proof} Let $F^{S_0}$ be given by \eqref{1.17}. Since $f_0=M_{F^{S_0}}P^{-1}{\bf x}$,
the formula for $\|f_0\|_{\cH(S_0)}$ follows from \eqref{1.17*}. Furthermore, by \eqref{1.10}, we have
\begin{align*}
(E^{*} f_0)^{\wedge L}(T^{*})&=\cO_{E,T}^*M_{F^{S_0}}P^{-1}{\bf x}\\
&=M_{F^{S_0}}^{[*]}M_{F^{S_0}}P^{-1}{\bf x}
=PP^{-1}{\bf x}={\bf x},
\end{align*}
which completes the proof.
\end{proof}
We next observe from \eqref{1.17*} that the operator $M_{F^{S_0}}P^{-\frac{1}{2}}: \; \cX\to \cH(S_0)$ is an isometry and that the space
\begin{equation}
{\mathcal N}=\{F^{S_0}(z)x: \, x\in\cX\}
\quad\mbox{with norm}\quad \|F^{S_0}x\|_{\cN}=\|P^{\frac{1}{2}}x\|_{_\cX}
\label{2.11}
\end{equation}
is isometrically included in  $\cH(S_0)$. Furthermore, the orthogonal complement of ${\mathcal N}$
in $\cH(S_0)$ is the range of the operator $I-M_{F^{S_0}}P^{-1}M_{F^{S_0}}^*$ with lifted norm of $\cH(S_0)$ and hence
it is a NFRKHS with reproducing kernel
\begin{equation}
K_{{\mathcal N}^\perp}(z,\zeta)=
K_{S_0}(z,\zeta)-F^{S_0}(z)P^{-1}F^{S_0}(\zeta)^*.
\label{2.12}
\end{equation}
On the other hand, $\cN^\perp$ can be characterized as the solution set of the homogeneous ${\bf OAP}_{\cH(S_0)}$:
$$
\cN^\perp=\left\{h\in \cH(S_0): \; (E^{*} h)^{\wedge L}(T^{*})=0\right\}.
$$
Indeed, for each $h\in\cH(S_0)$ we have $(E^{*} h)^{\wedge L}(T^{*})=M_{F^{S_0}}^{[*]}h=0$ (see \eqref{newcond}) if and only if
$$
\langle h, \, M_{F^{S_0}}x\rangle_{\cH(S_0)}=0\quad\mbox{for all}\quad x\in\cN.
$$
By the latter characterization and by Proposition \ref{P:4.2},
we conclude that $f\in\cY\langle\langle z\rangle\rangle$ solves the problem \eqref{evaldbr} if and only
if it is of the form $f=f_0+h$ where $h$ is an element in $\cN^\perp=\cH(K_{\cN^\perp})$.

To get a more specific description for $\cN^\perp$ we take a closer look at the kernel \eqref{2.12}.
We first use the definitions \eqref{KS} and \eqref{1.17} to write
\begin{align*}
& K_{\cN^\perp}(z,\zeta)
=\begin{bmatrix}I & -S_0(z)\end{bmatrix} (k_{\rm Sz}(z,\zeta) J_{\cY,  \cU})\begin{bmatrix}I \\ -S_0(\zeta)^*\end{bmatrix}\\
&\quad-\begin{bmatrix}I & -S_0(z)\end{bmatrix}\begin{bmatrix}E \\ N\end{bmatrix}(I-Z(z)T)^{-1}P^{-1}
(I-T^*Z(\zeta)^*)^{-1}\begin{bmatrix}E^* & N^*\end{bmatrix}\begin{bmatrix}I \\ -S_0(\zeta)^*\end{bmatrix}.
\end{align*}
Let $\mathfrak{A}=\sbm{\mathfrak{A}_{11}& \mathfrak{A}_{12}\\ \mathfrak{A}_{21} & \mathfrak{A}_{22}}$ be the power series constructed in Theorem \ref{T:4.1}. Then by identity \eqref{2.30}
we can represent the kernel $K_{\cN^\perp}$ as
\begin{equation}
K_{\cN^\perp}(z,\zeta)=\begin{bmatrix}I & -S_0(z)\end{bmatrix}
{\mathfrak A}(z)( k_{\rm Sz}(z,\zeta) J_{_{\cF,\cU}}) {\mathfrak A}(\zeta)^{*}
\begin{bmatrix}I \\ -S_0(\zeta)^*\end{bmatrix}.
\label{jan3}
\end{equation}
Now set
\begin{equation}
u(z) = {\mathfrak A}_{11}(z) - S_0(z) {\mathfrak A}_{21}(z), \qquad
        v(z) = S_0(z) {\mathfrak A}_{22}(z)-{\mathfrak A}_{12}(z),
\label{jan4}
\end{equation}
i.e.,
 % \begin{equation}   \label{def-uv}
$$
      \begin{bmatrix} u(z) & - v(z) \end{bmatrix} := \begin{bmatrix} I & -
           S_0(z) \end{bmatrix} {\mathfrak A}(z)\in\cL(\cF\oplus\cU,\cY)\langle\langle z\rangle\rangle.
$$
% \end{equation}
Then, on account of \eqref{JYU}, we can rewrite \eqref{jan3} as
        \begin{equation} \label{uvineq}
K_{\cN^\perp}(z,\zeta)= u(z) (k_{\rm Sz}(z,\zeta)I_{\cF}) u(\zeta)^{*} - v(z)(k_{\rm Sz}(z,\zeta)I_{\cU}) v(\zeta)^{*}.
        \end{equation}
Since the kernel $K_{\cN^\perp}$ is positive, it follows by the non-commutative Leech theorem (see \cite[Section 3.3]{bbbook}) that
there is a ${\mathcal E}_0 \in
        {\mathcal S}_{{\rm nc},d}(\cU, \cF)$ so that $v= u{\mathcal E}_0$, which being combined with \eqref{jan4} gives
$$
S_0{\mathfrak A}_{22}-{\mathfrak A}_{12}=\left({\mathfrak A}_{11}-S_0{\mathfrak A}_{21}\right){\mathcal E}_0.
$$
We thus recover $S_0$ as $S_0 = \fT_{{\mathfrak A}}[{\mathcal E}_0]$. On the other hand, plugging in $v= u{\mathcal E}_0$
into \eqref{uvineq} gives
\begin{align*}
K_{\cN^\perp}(z,\zeta)&= u(z)\left(k_{\rm Sz}(z,\zeta)I_{\cF}-{\mathcal E}_0(z)k_{\rm Sz}(z,\zeta)
{\mathcal E}_0(\zeta)^*\right)u(\zeta)^{*}
=u(z)K_{\mathcal E_0}(z,\zeta)u(\zeta)^{*}.
\end{align*}
Combining the latter equality with \eqref{2.12} gives
%\begin{equation}
$$
K_{S_0}(z,\zeta)=F^{S_0}(z)P^{-1}F^{S_0}(\zeta)^{*}+u(z)K_{\mathcal E_0}(z,\zeta)u(\zeta)^{*}.
%\label{2.12dec}\end{equation}
$$
The first and the second terms on the right are reproducing formal kernels of the subspace $\mathcal N$ in
\eqref{2.11} and its orthogonal complement ${\mathcal N}^\perp$ in $\mathcal H(S_0)$, respectively. By part (2) in Proposition
\ref{P:cmult}, $M_u: \; \cH({\mathcal E_0})\to {\mathcal N}^\perp$ is a coisometry and
hence $M_u: \; \cH(\cE_0)\to \mathcal H(S_0)$ is a partial isometry.

\smallskip

The space $\mathcal M:=\cH({\mathcal E_0})\ominus \Ker M_u$
is isometrically included in $\cH({\mathcal E_0})$
and the restriction of $M_u$ to this space is an isometry from $\mathcal M$ into $\mathcal H(K_{S_0})$.

\smallskip

Putting all these pieces together, we arrive at the following result.

\begin{thm}\label{T:AIPsol}
Let $(E,\bT)\in\cL(\cU,\cY)\times\cL(\cX)^d$ be an output stable pair, ${\bf x}$ a vector in $\cX$ and $S_0$ a Schur-class multiplier in $\in\cS_{{\rm nc},d}(\cU,\cY)$. Assume that $P$ defined as in \eqref{1.8**} is strictly positive definite. Let $\mathfrak A(z)$ be the formal power series constructed in Theorem \ref{T:4.1}. Then there exists a ${\mathcal E}_0 \in
        {\mathcal S}_{{\rm nc},d}(\cU, \cF)$ such that $S_0 = \fT_{{\mathfrak A}}[{\mathcal E_0}]$.
Furthermore, let
$$
u(z) = {\mathfrak A}_{11}(z) - S_0(z) {\mathfrak A}_{21}(z)\quad\mbox{and}\quad \mathcal M=\cH({\mathcal E_0})
\ominus \Ker M_u.
$$
Then $f\in\cY\langle\langle z\rangle\rangle$ is a solution to the ${\bf OAP}_{\cH(S_0)}$
if and only if $f$ is of the form
\begin{equation}
f(z)=f_0(z)+u(z)\sigma(z),\qquad \sigma\in\cM,
\label{2.14}
\end{equation}
where $f_0$ is defined in \eqref{fonc} and $\sigma$ is a free parameter from $\cM$.
Furthermore, for $f$ as in \eqref{2.14} we have
\begin{align*}
\|f\|^2_{\cH(S_0)}&=\|f_0\|_{ \mathcal H(S_0)}^2+\|u\sigma\|^2_{ \mathcal H(S_0)}
=\|P^{-\half}{\bf x}\|^2_{\cX}+\|\sigma\|^2_{\cH(\cE_0)}.
\end{align*}
\end{thm}

\subsection{Interpolation in $H^2_\cY(\free)$}
Following the single-variable prototype in Subsection \ref{S:BLthm}, we specify Theorem \ref{T:AIPsol} to the case where the coefficient
space $\cU$ is taken to be the zero space $\{0\}$ and hence the only member of the Schur class
${\mathcal S}_{{\rm nc},d}(\{0\}, \cF)$ is the zero power series $S_0=0: \; \{0\}\to\cY$. Then
$\cH(S_0)=H^2_{\cY}(\free)$, $N: \, \cX\to\{0\}$ (by \eqref{Ndef}), and
$F^{S_0}(z)=E(I-Z(z)T)^{-1}$ (see formula \eqref{1.17}), so that
$M_{F^{S_0}}=\cO_{E,{\bf T}}$. If we assume that the operator $P:=\cO_{E,{\bf T}}^*\cO_{E,{\bf T}}$
is strictly positive definite, i.e., the pair $(E,{\bf T})$ is {\em exactly observable},
we get the following result on $H^2_\cY(\free)$-interpolation as a
special case of Theorem \ref{T:AIPsol}. We recall that a power series $\Phi(z)\in {\mathcal S}_{{\rm nc},d}(\cF, \cY)$
is called {\em (strictly) inner} if the multiplication operator $M_\Phi$ is an isometry from $H^2_\cF(\free)$ onto
$H^2_\cY(\free)$.
\begin{thm}\label{T:H2AIPsol}
Given an output stable exactly observable pair $(E,{\bf T})\in\cL(\cX,\cY)\times \cL(\cX)^d$
(i.e., $P:=\cO_{E,{\bf T}}^*\cO_{E,{\bf T}}\succ 0$), there exists an auxiliary Hilbert space $\cF$
and a strictly inner $\Phi(z)\in {\mathcal S}_{{\rm nc},d}(\cF, \cY)$ such that
\begin{align}
K_{\Phi}(z, \zeta) &: = k_{\rm Sz}(z,\zeta)I_{\cY} - \Phi(z)( k_{\rm Sz}(z,\zeta)I_{\cF})\Phi(\zeta)^{*}\label{2.30nc}\\
&=E(I - Z(z) T)^{-1} P^{-1}(I - T^{*} Z(\zeta)^{*})^{-1}E^*.\notag
\end{align}
Furthermore, for a given vector ${\bf x}\in\cX$, a power series $f\in H^2_{\cY}(\free)$ satisfies the interpolation condition
$$
(E^*f)^{\wedge L}(\bT^{*}):= \cO_{E, \bT}^{*}f={\bf x}
$$
if and only if it is of the form
\begin{equation}
f(z)=E(I-Z(z)T)^{-1}P^{-1}{\bf x}+\Phi(z)\sigma(z),
\label{2.14nc}
\end{equation}
where $\sigma$ is a free parameter from $H^2_{\cF}(\free)$. Finally, for $f$ as in \eqref{2.14nc} we have
\begin{equation}\label{normid}
\|f\|^2_{H^2_{\cY}(\free)}=\|P^{-\half}{\bf x}\|_{\cX}^2+\|h\|^2_{H^2_{\cF}(\free)}.
\end{equation}
\end{thm}

\begin{proof}
To specify Theorem \ref{T:AIPsol} to the current particular case, we first note
that for $\cU=\{0\}$, the power series $\mathfrak{A}$ in  \eqref{2.29a} collapses to
\begin{equation}
\Phi(z):=\mathfrak A_{11}(z)=D_1+E(I-Z(z)T)^{-1}Z(z)B_1,
\label{2.15nc}
\end{equation}
where $\sbm{B_1\\ D_1}:   \cF\to\sbm{\cX^d\\ \cY}$ is an injective solution to the Cholesky factorization problem \eqref{10a}, which in its turn, now amounts to
\begin{equation}\label{cholh2}
\begin{bmatrix} B_1 \\ D_1 \end{bmatrix}
\begin{bmatrix} B_1^{*} & D_1^{*}\end{bmatrix}=\begin{bmatrix} P^{-1}\otimes I_d & 0 \\0 &
I_{\cY} \end{bmatrix}-\begin{bmatrix} T \\ E\end{bmatrix}P^{-1}\begin{bmatrix} T^{*} & E^{*}\end{bmatrix}.
\end{equation}
Note that this can be rewritten as
\[
\begin{bmatrix} T& B_1 \\ E& D_1 \end{bmatrix}
\begin{bmatrix} P^{-1} & 0 \\0 &
I_{\cF} \end{bmatrix}
\begin{bmatrix} T^*&E^*\\ B_1^{*} & D_1^{*}\end{bmatrix}
=\begin{bmatrix} P^{-1}\otimes I_d & 0 \\0 &
I_{\cY} \end{bmatrix}.
\]

The ``whole" identity \eqref{2.30} now amounts to the equality of the $11$-blocks, that is to the identity \eqref{2.30nc}. We next observe from  \eqref{pr} and
the Stein identity \eqref{2.8}, which in the present setting reduces to
\begin{equation}
P - \sum_{j=1}^{d} T_{j}^{*} P T_{j} = E^*E,
\label{sth2}
\end{equation}
that in the weak operator topology we have
\begin{align*}
P={\mathcal O}_{E,\bT}^*{\mathcal O}_{E, \bT}&=\lim_{n\to\infty}
\sum_{\alpha\in\free: |\alpha|< n}\bT^{*\alpha^\top}E^{*} E\bT^\alpha \\
&=\lim_{n\to\infty}\sum_{\alpha\in\free: |\alpha|< n}\bT^{*\alpha^\top}
\big(P - \sum_{j=1}^{d} T_{j}^{*} P T_{j}\big)\bT^\alpha \\
&=\lim_{n\to\infty}\bigg(\sum_{\alpha\in\free: |\alpha|< n}\bT^{*\alpha^\top}
P\bT^\alpha -\sum_{\alpha\in\free: 0<|\alpha|\le n}\bT^{*\alpha^\top}
P\bT^\alpha\bigg)\\
&= P-\lim_{n\to\infty}\sum_{\alpha\in\free: |\alpha|= n}\bT^{*\alpha^\top}P\bT^\alpha.
\end{align*}
Hence, for every ${\bf x}\in\cX$,
$$
\lim_{n\to\infty}\sum_{\alpha\in\free: |\alpha|= n}\|P^{\frac{1}{2}}\bT^\alpha {\bf x}\|=0.
$$
Since $P\succ 0$, we also have
$$
\lim_{n\to\infty}\sum_{\alpha\in\free: |\alpha|= n}\|\bT^\alpha {\bf x}\|=0,
$$
meaning that the tuple ${\bf T}$ is strongly stable. Due to equalities \eqref{cholh2} and \eqref{sth2}
and since the tuple ${\bf T}$ is strongly stable, the power series $\Phi$ in \eqref{2.15nc} is strictly inner by \cite[Theorem 3.2.11]{bbbook}. Moreover, since $S_0=0$, we obtain that $u$ in Theorem \ref{T:AIPsol} is given by $u=\mathfrak A_{11}=\Phi$.

Another character in Theorem \ref{T:AIPsol} is the Schur-class multiplier ${\mathcal E}_0 \in
        {\mathcal S}_{{\rm nc},d}(\cU, \cF)$ such that $S_0 = \fT_{{\mathfrak A}}[{\mathcal E}_0]$.
In the present setting, we have $0=S_0=\Phi{\mathcal E}_0$ with $\Phi$ strictly inner, so that $\cE_0=0: \; \{0\}\to\cF$, and hence $\cH({\mathcal E}_0)=
H^2_\cF(\free)$. Since $\Phi$ is strictly inner, $\Ker M_\Phi=\{0\}$, and the parametrization formula  \eqref{2.14} takes the form \eqref{2.14nc}. Finally, the norm identity \eqref{normid} now follows immediately.
\end{proof}
As an example of an exactly observable pair, let us consider the pair $({\bf ev}_\emptyset, {\bf R}_z^*)$, where
${\bf R}_z^*=(R_{z_1}^*,\ldots,R_{z_d}^*)$ is the backward-shift tuple on $H^2_{\cY}(\free)$ consisting of the
adjoint operators of right coordinate multipliers \eqref{intertw}:
$$
R^*_{z_j}: \; \sum_{\alpha\in\free} f_\alpha z^\alpha\mapsto \sum_{\alpha\in\free} f_{\alpha j} z^\alpha
\quad\mbox{for}\quad j=1,\ldots,d
$$
and ${\bf ev}_\emptyset$ is the free-coefficient evaluation operator
$$
{\bf ev}_\emptyset: \; \sum_{\alpha\in\free} f_\alpha z^\alpha\mapsto f_\emptyset.
$$
The computation
$$
\cO_{{\bf ev}_\emptyset,{\bf R}_z^*}f=\sum_{\alpha\in\free} ({\bf ev}_\emptyset {\bf R}_z^{*\alpha}f)z^\alpha=
\sum_{\alpha\in\free} f_\alpha z^\alpha=f
$$
shows that $\cO_{{\bf ev}_\emptyset,{\bf R}_z^*}$ is the identity operator on $H^2_{\cY}(\free)$.
If $\cM$ is a closed subspace of $H^2_{\cY}(\free)$ which is ${\bf R}_z$-invariant, i.e.,
$$
R_{z_j}\cM\subset \cM\quad\mbox{for}\quad j=1,\ldots,d,
$$
then the restricted output pair $(E, {\bf T}) \in \cL(\cX, \cY) \times \cL(\cX)^d$ given by
$$
 \cX = \cM^\perp, \quad E = {\bf ev}_\emptyset|_\cX, \quad {\bf T} = {\bf R}_z^*|_\cX
$$
is exactly observable and $\cM^\perp  = \Ran \cO_{E,{\bf T}}$. Therefore,
$\cM = \Ker \cO_{E,{\bf T}}^*$ coincides with the solution set of the
homogeneous Operator Argument interpolation Problem ${\bf OAP}_{H^2_{\cY}(\free)}$ with the interpolation condition
$$
(E^*f)^{\wedge L}(\bT^{*})= \cO_{E, \bT}^{*}f=0.
$$
Then part (2) in Theorem \ref{T:H2AIPsol} implies the non-trivial ``only if" part in
the following Beurling-Lax type theorem; see \cite{popescu89}.
\begin{thm}
\label{BLPop}
A closed subspace $\cM$ of $H^2_{\cY}(\free)$ is ${\bf R}_z$-invariant
if and only if there exist a Hilbert space $\cF$ and a strictly inner multiplier $\Phi\in {\mathcal S}_{{\rm nc},d}(\cF, \cY)$
such that  $\cM=\Phi\cdot H^2_{\cF}(\free)$.
\end{thm}

\section{The Drury-Arveson space setting}\label{S:DruryArveson}

If we replace the tuple of freely noncommutative indeterminates $z=(z_1,\ldots,z_d)$ by a  commutative $d$-tuple of complex numbers $\lam= (\lam_1, \dots, \lam_d)$,  i.e., so that $\lam_i \lam_j = \lam_j \lam_i$ for all $i,j=1,\ldots,d$, then the Fock space \eqref{Fock} becomes the {\em Drury-Arveson space}. To be more precise, let us recall the following (commutative) multivariable notation: for multi-integers $\bn =(n_{1},\ldots,n_{d})\in{\mathbb Z}_+^d$ and points $z=(z_1,\ldots,z_d)\in\C^d$ we set
$$
|\bn| = n_{1}+n_{2}+\cdots +n_{d},\quad \bn!  = n_{1}!n_{2}!\ldots n_{d}!, \quad
z^\bn = z_{1}^{n_{1}}z_{2}^{n_{2}}\ldots z_{d}^{n_{d}}.
$$
Given a word $\alpha= i_{N}i_{N-1} \cdots i_{1}\in\free$  we let ${\mathbf a}(\alpha) \in {\mathbb Z}^{d}_{+}$ be the {\em abelianization} of $\alpha$, given by
$$
     {\mathbf a}(\alpha) = (n_{1}, \dots, n_{d})\in\Z_+^d\quad \text{where}\quad
     n_{j} = \#\{ k \colon i_{k} = j \} \quad\text{for}\quad j = 1, \dots, d
$$
(where in general $\# \Xi$ denotes the cardinality of the set $\Xi$).
Then the noncommutative and commutative multivariable functional calculus are related by
$z^\alpha \to \lam^{{\mathbf a}(\alpha)}$.

If ${\bf T}=(T_1,\ldots,T_d)\in\cL(\cX)^d$ is a tuple of {\em commuting}
Hilbert-space operators, we have
$$
{\bf T}^{\alpha}:=T_{i_{N}} \cdots T_{i_{1}}={\bf T}^{\ba(\alpha)}.
$$
Observe that for a given ${\mathbf n} \in {\mathbb Z}^{d}_{+}$,
%\begin{equation}
$$
     \# \{ \alpha \in \free \colon {\mathbf a}(\alpha) = {\mathbf n}\} =
     \frac{|{\mathbf n}|!}{{\mathbf n}!}.
$$
%\label{comb}\end{equation}
Using the latter combinatorial fact, we may compute the commutative version $k_d$ of the
noncommutative Szeg\H{o} kernel \eqref{sz2}:
\begin{equation}
k_d(z,\zeta) \mapsto k_d(\lam, \eta) =  \sum_{\alpha \in \free} \lam^{\alpha}\overline{\eta}^{\alpha^{\top}}=
\sum_{{\bf n}\in{\mathbb Z}^d_+} \frac{|{\mathbf n}|!}{{\mathbf n}!}\lam^{\bf n}  \overline{\eta}^{\bf n}.
\label{Sz3}
\end{equation}
We will write $\langle \lam,  \eta \rangle=\sum_{j=1}^d \lam_j \overline{\eta}_j$
for the standard inner product in $\C^d$ and we will denote by
$\B^d=\left\{\lam=(\lam_1,\dots,\lam_d)\in\C^d \colon  \langle \lam, \lam\rangle<1\right\}$ the unit ball of the Euclidean
space $\C^d$. The power series in \eqref{Sz3} converges on ${\mathbb B}^{d} \times {\mathbb B}^{d}$ and
can be written in the closed form as
$$
k_d(\lam,\eta)=\frac{1}{1-\langle \lam, \, \eta\rangle},\quad \lam,\eta\in\mathbb B^d.
$$
The reproducing kernel Hilbert space (RKHS) $\cH_\cY(k_d)$ associated with the positive kernel $k_d(z,\zeta)I_{\cY}$, called the {\em Drury-Arveson space}, can be explicitly characterized as
%\begin{equation}
$$
\cH_{\cY}(k_d)=\bigg\{f(\lam)=\sum_{\bn \in{\mathbb Z}^d_{+}}f_{\bn}
\lam^\bn:\|f\|^{2}=\sum_{\bn \in {\mathbb Z}^d_{+}}
\frac{\bn!}{|\bn|!}\cdot \|f_{\bn}\|_{\cY}^2<\infty\bigg\}.
$$
%\label{1.1aip}\end{equation}
Note that if  we let the variables commute for an element $f$ of the Fock space $H^2_\cY(\free)$, we get a map
\begin{equation}   \label{Fock-to-DA}
f(z) = \sum_{\alpha \in \free} f_\alpha z^\alpha  \in H^2_\cY(\free) \mapsto \sum_{\bn \in {\mathbb Z}^+_d}
\bigg( \sum_{\alpha \colon \ba(\alpha) = \bn} f_\alpha \bigg) \lam^\bn.
\end{equation}
This map is not injective. However to make a canonical choice of $f\in H^2_\cY(\free)$ we can consider the following Calculus of Variations problem:\  {\em for a fixed index $\bn \in {\mathbb Z}^d_+$,  given a coefficient $f_\bn\in\cY$, among all collections of coefficient vectors $\{f_\alpha\colon \alpha \in \ba^{-1}(\bn)\}$ such that
\begin{equation}  \label{constraint}
\sum_{\alpha \in \ba^{-1}(\bn)} f_\alpha = f_\bn,
\end{equation} find the choice of the collection which minimizes $\sum_{\alpha \in \ba^{-1}(\bn)} \| f_\alpha \|^2$.} The solution is to choose $f_\alpha$ to be independent of $\alpha$ subject to the constraint \eqref{constraint}, i.e.,
$$
f_\alpha = \frac{{\mathbf n}!}{|{\mathbf n}|!}f_\bn\quad\mbox{for all}\quad \alpha\in\ba^{-1}(\bn).
$$
If $\cY$ is separable, as we assume throughout the paper, the problem reduces to the scalar-valued case, and in that case the solution follows by the classical arithmetic-quadratic mean inequality
\[
\left| \frac{a_1+\cdots+a_n}{n}\right|\leq \sqrt{\frac{|a_1|^2+\cdots+|a_n|^2}{n}}.
\]

\smallskip

If we restrict to the {\em permutation-invariant Fock space}, i.e., the subspace $H^2_\cY(\Pi \free)$
of $H^2_\cY(\free)$ consisting of power series $f = \sum_{\alpha \in \free} f_\alpha z^\alpha$ so that
$$
\ba(\alpha) = \ba(\alpha')  \Rightarrow f_\alpha = f_{\alpha'},
$$
it follows that the map \eqref{Fock-to-DA} is an isometric-isomorphism of $H^2_\cY(\Pi \free)$ onto $H^2_\cY(k_d)$
with inverse given by
%\begin{equation}   \label{DA-Fock}
$$
\sum_{\bn \in {\mathbb Z}^d_+} f_\bn \lam^\bn \mapsto \sum_{\bn \in {\mathbb Z}^d_+}
\sum_{\alpha \colon \ba(\alpha) = \bn} \frac{{\mathbf n}!}{|{\mathbf n}|!} f_\bn z^\alpha .
$$
%\end{equation}

In short we see that there are two distinct approaches to study the Drury-Arveson space:\ directly, or via making use of known structure for the Fock space to project down via the abelianization map \eqref{Fock-to-DA} and arrive at results for the Drury-Arveson space. For example the papers of Arias-Popescu \cite{arpop} and Davidson-Pitts \cite{davpitts} arrive at the solution of the Nevanlinna-Pick interpolation problem for Drury-Arveson space multipliers in papers mostly concerned with the Fock space setting, whereas the papers of Ball-Trent-Vinnikov \cite{BTV} and  Agler-McCarthy \cite{AgMcC} proceed directly.   There are more recent examples of both
approaches which are too numerous to mention here in detail.
Here we shall work directly with the Drury-Arveson space while using the Fock-space results as a guide for the proofs. Since many of the formulas and results from Section \ref{S:FockSpace} carry over directly to the commutative setting of the current section, we will often refer to these formulas and results and just give sketches of the proofs rather than writing out formulas and providing full details here again.

\begin{rem} \label{cO/cOa}  {\em Given an output pair $(E, \bT)\in\cL(\cX,\cY)\times \cL(\cX)^d$ as in Section \ref{S:FockSpace}, so with the entries $T_j$ of $\bT$ not necessarily commutating, one can associate two different observability operators, leading to two notions of output stability:\ We can define $\cO_{E,\bT}$ as in Section \ref{S:FockSpace} via \eqref{obsnc} and again call $(E,\bT)$ output stable if $\cO_{E,\bT}$ maps $\cX$ into $H^2_\cY(\free)$, or work with the commutative observability operator
$$
\cO^{\bf a}_{E, {\mathbf T}} \colon x \mapsto E (I_{\cX}- Z(\lam) T)^{-1}x = \sum_{\bn \in {\mathbb Z}^d_+}
 E  \,\big(\sum_{\alpha \colon \ba(\alpha)  = \bn} T^\alpha \big)\,  \lam^\bn,
$$
with $T$ and $Z(\lam)$ as in \eqref{1.6a}, and say that $\cO_{E,\bT}$ is ${\bf a}$-output stability when $\cO^{\bf a}_{E, {\mathbf T}}$ determines a bounded map from $\cX$ into $\cH_\cY(k_d)$. In case the tuple ${\bf T}$ is commutative (or at least when $E{\bf T}^\alpha=E{\bf T}^\beta$ for all $\alpha,\beta\in\free$ such that ${\bf a}(\alpha)={\bf a}(\beta)$) the two notions of output stability coincide, but in general, ${\bf a}$-output stability is a stronger notion than just output stability. Indeed, as was shown in \cite[Proposition 3.8]{bbf1}, if a pair $(E,{\bf T})$ is output stable, then
$$
\cO^{\bf a *}_{E, {\mathbf T}} \cO^{\bf a}_{E, {\mathbf T}}\preceq \cO_{E, {\mathbf T}}^*\cO_{E, {\mathbf T}}
$$
and hence, $(E,{\bf T})$ is also ${\bf a}$-output stable. Throughout this section, however, the tuple ${\bf T}$ is assumed to be commutative and we will drop the superscript ${\bf a}$ and simply write output stable instead of ${\bf a}$-output stable and $\cO_{E, {\mathbf T}}$ instead of $\cO^{\bf a}_{E, {\mathbf T}}$.}
\end{rem}

The $d$-variable Schur class $\mathcal S_d(\cU,\cY)$ is defined as the set of all contractive multipliers from $\cH_{\cU}(k_d)$ to $\cH_{\cY}(k_d)$, i.e., the set of all $\cL(\cU,\cY)$-valued analytic functions  $\lam \mapsto S(\lam)$ on $\B^d$ such that the multiplication operator $M_S \colon f(\lam) \mapsto F(\lam)f(\lam)$ is a contraction from $\cH_\cU(k_d)$ into $\cH_\cY(k_d)$. The latter is equivalent to the kernel
%\begin{equation}
$$
K_S(\lam,\eta)=\frac{I-S(\lam)S(\eta)^*}{1-\langle \lam, \, \eta\rangle}
$$
%\label{1.6aip}\end{equation}
being positive on ${\mathbb B}^{d} \times {\mathbb B}^{d}$. The associated reproducing kernel Hilbert space $\cH(S):=\mathcal H(K_S)$ is the $d$-variable de Branges-Rovnyak space associated with $S\in \mathcal S_d(\cU,\cY)$.

\smallskip
To formulate the Operator Argument interpolation Problem \textbf{OAP}$_{\cH(S_0)}$ in $\mathcal H(S_0)$, we need a left-tangential calculus, which can be obtained via abelianization of \eqref{2.1u}. Given an output stable pair $(E, \bT)\in\cL(\cX,\cY)\times \cL(\cX)^d$, see Remark \ref{cO/cOa}, since for $x\in \cX$, the vector
$\cO_{E, {\mathbf T}}x$ is now a vector-valued function analytic around the origin, we can
compute its Taylor expansion
\begin{align}
        (\cO_{E, {\mathbf T}} x)(\lam) & = E (I_{\cX} - Z(\lam) T)^{-1} x
         = E \sum_{\alpha \in \free} {\mathbf T}^{\alpha}
        \lam^{\alpha}x \notag\\
& = E\sum_{{\mathbf n} \in {\mathbb Z}^{d}_{+}} \bigg(
       \sum_{\alpha \in \free \colon {\mathbf a}(\alpha) = \bn} {\mathbf T}^{\alpha}
        \lam^{\alpha}   \bigg)  x
        = \sum_{{\mathbf n} \in {\mathbb Z}^{d}_{+}} \frac{ |{\mathbf
        n}|!}{{\mathbf n}!} E {\mathbf T}^{{\mathbf n}} \lam^{{\mathbf n}}x.\notag
%        \label{obsTaylor}
\end{align}
We then define a {\em left-tangential functional calculus}
$f \to (E^{*}f)^{\wedge  L}(\bT^{*})$ on $\cH_{\cY}(k_{d})$ for the output stable pair $(E, \bT)$  by
\begin{equation}\label{2.1aip}
(E^{*} f)^{\wedge L}(\bT^{*}) := \sum_{ \bn \in {\mathbb Z}^{d}_{+}}
             \bT^{* \bn} E^{*} f_{\bn}\quad \text{if} \quad f(\lam) =
\sum_{\bn \in
{\mathbb Z}^{d}_{+}} f_{\bn} \lam^{\bn} \in \cH_{\cY}(k_{d}).
        \end{equation}
The computation similar to that in \eqref{compute}:
       \begin{align*}
            \left\langle \sum_{\bn \in {\mathbb Z}^{d}_{+}} \bT^{* \bn}E^{*}
        f_{\bn}, \; x \right \rangle_{\cX} & =
        \sum_{\bn \in {\mathbb Z}^{d}_{+}} \left\langle f_{\bn}, \; E \bT^{\bn}
        x \right \rangle_{\cY} \\
        & = \sum_{\bn \in {\mathbb Z}^{d}_{+}} \frac{\bn !}{|\bn|!}
\left\langle
         f_{\bn}, \frac{|\bn|!}{\bn !} E \bT^{\bn} x \right \rangle_{\cY}
         = \langle f, \; \mathcal {O}_{E, \bT} x\rangle_{\cH_{\cY}(k_{d})},
        \end{align*}
shows that the left-tangential evaluation \eqref{2.1aip}
        amounts to the adjoint of the observability operator:
%        \begin{equation}
$$
            (E^{*} f)^{\wedge L}(\bT^{*}) = \cO_{E, \bT}^{*}
            f \quad\text{for}\quad f \in \cH_{\cY}(k_{d})
$$
%\label{2.2aip}\end{equation}
and applies to Schur-class multipliers as well as to functions from a given de Branges-Rovnyak space $\mathcal H(S_0)$. The Operator Argument interpolation Problem in the de Branges-Rovnyak space $\cH(S_0)$ (\textbf{OAP}$_{\cH(S_0)}$) is now formulated as follows.

\medskip
\noindent
${\bf OAP}_{\cH(S_0)}$: {\em Given an output stable pair $(E,\bT)$ with $E\in\cL(\cX,\cY)$ and a commutative $d$-tuple ${\bf T}=(T_1,\ldots,T_d)\in\cL(\cX)^d$,
a vector ${\bf x}\in\cX$ and a Schur-class multiplier $S_0 \in\cS_{d}(\cU,\cY)$, find all $f\in\cH(S_0)$ such that
    \begin{equation}\label{2.5aip}
    (E^{*} f)^{\wedge L}(\bT^{*})= \cO_{E, \bT}^{*}f={\bf x}.
    \end{equation}}

Following the lines in Section 3.1 we define the operator $N\in\cL(\cX,\cU)$ by the formula
\begin{equation}\label{Ndefaip}
N:=\sum_{{\bf n}\in\mathbb Z_+^d} S_{0,\bf n}^*E{\bf T}^{\bf n},\quad\mbox{where}\quad
    S_0(\lam)=\sum_{{\bf n}\in\mathbb Z_+^d} S_{0,\bf n} \lam^{\bf n},
\end{equation}
or equivalently, via its adjoint, by formula \eqref{1.8}. By \cite[Proposition 3.1]{bbieot},
the pair $(N,{\bf T})$ is output stable and equality
%\begin{equation}
$$
\cO_{E,{\bf T}}^*M_{S_0}=\cO_{N,{\bf T}}^*: \; \cH_{\cY}(k_d)\to\cX
%\label{1.8*com}\end{equation}
$$
holds. Then the operator $P$ defined as in \eqref{1.8**} is positive semidefinite. It is now defined
via abelianized infinite series
\begin{equation}
P:=\cO_{E,{\bf T}}^*\cO_{E,{\bf T}}-\cO_{N,{\bf T}}^*\cO_{N,{\bf T}}
=\sum_{\bn\in\mathbb Z_+^d}\frac{|\bn|!}{\bn !}\bT^{*{\bn}}(E^{*} E-N^*N)\bT^{\bn}
\label{1.8**com}
\end{equation}
and still satisfies the Stein identity \eqref{2.8}. We next use the formula \eqref{1.17} to define
the holomorphic $\cL(\cX,\cY)$-valued function $F^{S_0}$ such that for every $x\in\cX$, the $\cY$-valued function
$F^{S_0}x = F^{S_0}(\cdot) x $ belongs to the de commutative Branges-Rovnyak space $\mathcal H(S_0)$ and satisfies equality \eqref{1.17*}.
The abelianized version of Lemma \ref{L:3.2} establishes equalities \eqref{1.10} where now $M_{F^{S}_0}^{[*]}$ denotes the
adjoint of the operator $M_{F^S}: \, \cX\to\cH_{\cY}(k_d)$ in the metric of $\cH(S_0)$.
Hence, as in the non-commutative case, the interpolation condition \eqref{2.5aip} can be written as
%\begin{equation}
$$
\cO_{E,{\bf T}}^*f=M_{F^S}^{[*]}f={\bf x}
$$
%\label{1.8**cond}\end{equation}
implying that the problem ${\bf OAP}_{\cH(S_0)}$ has a solution if and only if
${\bf x}\in\Ran M_{F^S}^{[*]}=\Ran P^\half$.

If we assume that the operator $P$ in \eqref{1.8**com} is strictly positive definite, we can use Lemma \ref{L:findBD}
and Corollary \ref{C:J-UniCompl} to construct a $\big( \sbm{P  & 0 \\ 0 & J_{\cF, \cU}}, \sbm{ P \otimes I_d & 0 \\ 0 & J_{\cY, \cU}}
\big)$-unitary operator $\bU = \sbm{ T & B \\ C & D}$ (where we set $C = \sbm{ E \\ N}$) which we can use as the
colligation matrix to define an abelianized transfer function
%\begin{equation}   \label{2.29a'}
$$
\fA(\lambda) = D + C ( I - Z(\lam) A)^{-1} Z(\lam) B
$$
%\end{equation}
Then it is a matter of checking that an abelianized version of the algebra behind the proof of item (3) in Theorem \ref{T:4.1}
shows that both kernels
\begin{align}
K_\fA^{J_{\cF, \cU}, J_{\cY, \cU}}(\lam, \eta) = k_d(\lam, \eta)J_{\cY, \cU} - \fA(\lam) (k_d(\lam, \eta) J_{\cF, \cU}) \fA(\eta)^* \label{jKS'} \\
\wtil K_\fA^{J_{\cY, \cU}, J_{\cF, \cU}}(\lam, \eta)=  k_d(\lam, \eta)J_{\cF, \cU} - \fA(\eta)^* (k_d(\lam, \eta) J_{\cY, \cU} ) \fA(\lam)\notag
%\label{jKS-dual'}
\end{align}
are positive.

Furthermore, the first kernel \eqref{jKS'} satisfies the identity
\begin{align}
&\frac{J_{_{\cY,\cU}}\!\! - {\mathfrak A}(\lam)J_{_{\cF,\cU}}{\mathfrak A}(\eta)^{*}}{1 -
        \langle \lam, \eta \rangle}
\!=\!\begin{bmatrix}E \\ N\end{bmatrix}\! (I - Z(\lam) T)^{-1} P^{-1}
          (I - T^{*} Z(\eta)^{*})^{-1}\! \begin{bmatrix}E^*\!\!\!\!\! & N^*\end{bmatrix},\label{2.30aip}
        \end{align}
the abelianization of \eqref{2.30}. Positivity of the kernels $K_\fA^{J_{\cF, \cU}, J_{\cY, \cU}}$ and $\wtil K_\fA^{J_{\cY, \cU}, J_{\cF, \cU}}$
(more precisely, of their compressions to $\cU$) guarantees that the operator $\mathfrak A_{22}(\lam)$ is invertible at every point $\lam\in\mathbb B^d$
and moreover, that
\begin{equation}  \label{int-contr}
\|\fA_{22}(\lam)^{-1} \fA_{21}(\lam) \| < 1 \text{ for } \lam\in\mathbb B^d.
\end{equation}
However we can follow the same operator-theoretic argument as used for the noncommutative Fock-space setting (not involving
evaluations at interior points in the ball) to show that the multiplier-norm analogue of \eqref{int-contr} holds, namely,
$$
\| M_{\fA_{22}}^{-1} M_{\fA_{21}} \| < 1,
$$
and that the Drury-Arveson linear-fractional map
$$
\cE \mapsto \fT_\fA[\cE] = (\fA_{11} \cE + \fA_{12})( \fA_{21} \cE + \fA_{22})^{-1}
$$
maps the Schur class $\cS_d(\cU, \cF)$ into the Schur class $\cS_d(\cU, \cY)$ via the Drury-Arveson space analogue of
Theorem \ref{T:lfts}.

\begin{rem}  \label{R:com-LFTinj}
{\em Let us note that Theorem \ref{T:LFTinjective} and Remark \ref{R:LFTinjective} formally apply to the present commutative situation
without change.  However Example \ref{E:nc-inj} is distinctively different in the commutative situation in that the commutative
analogue of \eqref{nc-row-ex}, namely the operator
$$
  M_{Z_{\rm com}} \colon \begin{bmatrix} g_1(\lam) \\ \vdots \\ g_d(\lam) \end{bmatrix} \mapsto \lam_1 g_1(\lam)
  + \cdots + \lam_d g_d(\lam)
$$
is definitely not injective as an operator from $\cH_{{\mathbb C}^d}(k_d)$ to $\cH(k_d)$. For example, for $d=2$, the Schur multiplier
$\sbm{ \lam_2 \\ -\lam_1}$ is in the kernel of $M_{Z_{\rm com}}$.
It is the case that $\begin{bmatrix} \lam_1 & \lam_2 \end{bmatrix}$ is the Beurling-Lax representer (in the sense of
McCullough-Trent \cite{mcctr}) for the shift-invariant subspace $\{ f \in \cH(k_d) \colon f(0) = 0\}$, but for this setting
one must take the Beurling-Lax representer  $Z_{\rm com}$ to be such that
$$
M_{Z_{\rm com}} \colon \cH_{{\mathbb C}^2}(k_2) \to \cH(k_2)
$$
to be only a partial isometry rather than an isometry (i.e., to be what we call McCT-inner).  This is a particular instance of the
homogeneous version of Theorem \ref{T:H2AIPsolcom} discussed below.
}\end{rem}

The parametrization of the solution set for the ${\bf OAP}_{\cH(S_0)}$ can be obtained along the same lines as in Section 3.4.
We first verify that the function $f_0$ defined by the formula \eqref{fonc} belongs to $\cH(S_0)$ and satisfies
\eqref{2.5aip} with norm equality $\|f_0\|_{\cH(S_0)}=\|P^{-\half} {\bf x}\|_{\cX}$. We next introduce the subspace
$\cN$ of $\cH(S_0)$ defined as in \eqref{2.11} and its orthogonal complement $\cN^\perp$ which is on the one hand the solution
set for the homogeneous problem \eqref{2.5aip} with ${\bf x}=0$, and on another hand a reproducing kernel
Hilbert space with reproducing kernel the abelianization of \eqref{2.12}. The same manipulations with $K_{\cN^\perp}$ as in
the non-commutative case, but based on the identity \eqref{2.30aip} rather than \eqref{2.30}, lead us to the formula
$$
K_{\cN^\perp}(\lam,\eta)=\begin{bmatrix}I & -S_0(\lam)\end{bmatrix}\frac{{\mathfrak A}(\lam)J_{_{\cF,\cU}} {\mathfrak A}(\eta)^{*}}{1-\langle \lam,\eta\rangle}
\begin{bmatrix}I \\ -S_0(\eta)^*\end{bmatrix}.
$$
With holomorphic operator-valued functions $u$ and $v$ defined by formulas \eqref{jan4}, we write the previous equality as
$$
%        \begin{equation} \label{uvineqaip}
K_{\cN^\perp}(\lam,\eta)= \frac{u(\lam)u(\eta)^{*}- v(\lam)v(\eta)^{*}}{1-\langle \lam,\eta\rangle},\quad \lam,\eta\in\mathbb B^d.
$$
%        \end{equation}
The positivity of the latter kernel on $\mathbb B^d\times \mathbb B^d$ implies, by the commutative multivariable Leech theorem
(see \cite{ABK,bb4})  that there is a ${\mathcal E}_0 \in {\mathcal S}_{d}(\cU, \cF)$ so that $v= u{\mathcal E}_0$, from which we
recover $S_0$ as $S_0 = T_{{\mathfrak A}}[{\mathcal E}_0]$ and get the eventual representation formula
%\begin{equation}
$$
K_{S_0}(\lam,\eta)=F^{S_0}(\lam)P^{-1}F^{S_0}(\eta)^{*}+u(\lam)K_{\mathcal E_0}(\lam,\eta)u(\eta)^{*}.
$$
%\label{2.12decaip}\end{equation}
As in the non-commutative case, $M_u: \; \cH(\cE_0)\to {\mathcal M}^\perp$ is a coisometry and hence
$M_u: \; \cH(\mathcal E_0)\to \mathcal H(S_0)$ is a partial isometry, while the restriction of $M_u$ to
the space $\mathcal M=\cH(\mathcal E_0)\ominus \Ker M_u$ maps this space isometrically into $\mathcal H(S_0)$.
Next we state the analog of Theorem \ref{T:AIPsol}.

\begin{thm}\label{T:AIPsolaip}
Let $(E,\bT)\in\cL(\cU,\cY)\times\cL(\cX)^d$ be an output stable pair with $\bT$ a commutative tuple, ${\bf x}$ a vector in $\cX$ and $S_0$ a Schur-class multiplier in $\in\cS_{d}(\cU,\cY)$. Assume that $P\succ 0$. Let $\mathfrak A(\lam)$ be a function on $\B^d$ subject to the identity \eqref{2.30aip} with $N$ as in \eqref{Ndefaip}, let  ${\mathcal E}_0 \in
        {\mathcal S}_{d}(\cU, \cF)$ be such that $S_0 = T_{{\mathfrak A}}[{\mathcal E}_0]$, and let
$$
u(\lam) = {\mathfrak A}_{11}(\lam) - S_0(\lam) {\mathfrak A}_{21}(\lam)\quad\mbox{and}\quad \mathcal M=\cH(\mathcal E_0)\ominus \Ker M_u.
$$
Then $f\in\cY\langle\langle z\rangle\rangle$ is a solution to the ${\bf OAP}_{\cH(S_0)}$
if and only if $f$ is of the form
\begin{equation}
f(\lam)=f_0(\lam)+u(\lam)h(\lam),\qquad h\in\cM,
\label{2.14aip}
\end{equation}
where $f_0$ is defined in \eqref{fonc} and $h$ is a free parameter from $\cM$.
Furthermore, for $f$ defined by \eqref{2.14aip}, we have
\begin{align*}
\|f\|^2_{\cH(S_0)}=\|f_0\|_{ \mathcal H(S_0)}^2+\|uh\|^2_{ \mathcal H(S_0)}=\|P^{-\half}{\bf x}\|^2_{\cX}+\|h\|^2_{\cH(\cE_0)}.
\end{align*}
\end{thm}

Let us say that a Schur-class multiplier $S\in\mathcal S_d(\cF,\cY)$ is {\em McCT-inner} (referring to the authors of the seminal paper \cite{mcctr})
if the operator $M_S: \, \cH_\cF(k_d)\to \cH_\cY(k_d)$ is a partial isometry.

\smallskip

If we specify Theorem \ref{T:AIPsolaip} to the case where $\cU=\{0\}$ (as we did in Section 3.5 in the non-commutative setting),
then $\mathfrak A$ collapses to the analytic function $\Phi(\lam)$ given by the formula \eqref{2.15nc}. The assumption
$P:=\cO_{E,{\bf T}}^*\cO_{E,{\bf T}}\succ 0$ still implies that ${\bf T}$ is strongly stable, i.e., that
$$
\lim_{N\to\infty}\sum_{\bn\in\mathbb Z_+^d: |\bn|=N}
\frac{\bn !}{|\bn|!}\|{\bf T}^{\bn}x\|^2_{\cH_{_\cY}(k_d)}=0\quad\mbox{for all} \; \; x\in\cX,
$$
which this time guarantees only $\Phi$ be McCT-inner; see \cite[Section 5]{bbf2}. The rest is the same as in Theorem \ref{T:H2AIPsol}.

\begin{thm}\label{T:H2AIPsolcom}
Given an output stable, exactly observable pair $(E,{\bf T})\in\cL(\cX,\cY)\times \cL(\cX)^d$
(i.e., $P:=\cO_{E,{\bf T}}^*\cO_{E,{\bf T}}\succ 0$) with $\bT$ a commutative tuple,
there exists an auxiliary Hilbert space $\cF$ and a  McCT-inner function $\Phi(\lam)\in {\mathcal S}_{d}(\cF, \cY)$
such that
$$
\frac{I_{\cY}-\Phi(\lam)\Phi(\eta)^{*}}{1-\langle \lam,\eta\rangle}=E(I - Z(\lam) T)^{-1} P^{-1}(I - T^{*} Z(\eta)^{*})^{-1}E^*.
$$
Furthermore, for a given vector ${\bf x}\in\cX$, a function $f\in \cH_{\cY}(k_d)$ satisfies the interpolation condition
$$
(E^*f)^{\wedge L}(\bT^{*}):= \cO_{E, \bT}^{*}f={\bf x}
$$
if and only if it is of the form
%\begin{equation}
$$
f(\lam)=E(I-Z(\lam)T)^{-1}P^{-1}{\bf x}+\Phi(\lam)h(\lam),\quad h\in\cH_{\cF}(k_d).
$$
%\label{2.14com}\end{equation}
\end{thm}

The homogeneous version of the latter theorem contains the $\cH_\cY(k_d)$-version of the Beurling-Lax theorem \cite{mcctr}. Still writing ${\bf R}_z$ for the $d$-tuple of coordinate-variable multipliers $R_{\lam_j}$ in $\cH_\cY(k_d)$, we note that their adjoints are now given by formulas
$$
R_{\lam_j}^*: \; \sum_{\bn\in\mathbb Z_+^d}f_{\bn}\lam^{\bn}\mapsto \sum_{\bn\in\mathbb Z_+^d}f_{\bn+e_j}\lam^{\bn}\quad\mbox{for}\quad
j=1,\ldots,d
$$
where $e_j$ is the element in $\mathbb Z_+^d$ having the $j$-th partial index equal to one and all other partial
indices equal to zero. If we let ${\bf ev}_0: \, f(\lam)\to f(0)$ denote the zero-evaluation operator on $\cH_\cY(k_d)$, then
it is easily verified that ${\bf ev}_0{\bf R}_\lam^{\bn} \colon \sum_{\bn' \in {\mathbb Z}^d_+} f_{\bn'}  \lam^{\bn'} \to f_{\bn}$ for all
$\bn \in \mathbb Z_d^+$ and subsequently, that
the observability operator $\cO_{{\bf ev}_0,{\bf R}_\lam}$ equals the identity operator on $\cH_\cY(k_d)$. Indeed,
for any $f\in\cH_\cY(k_d)$,
we have
$$
\cO_{{\bf ev}_0,{\bf R}_\lam^*}f=\sum_{\bn\in\mathbb Z_+^d} ({\bf ev}_0 {\bf R}_\lam^{*{\bn}}f)\lam^{\bn}=
\sum_{\bn\in\mathbb Z_+^d} f_{\bn}\lam^{\bn}=f.
$$
If $\cM$ is a closed subspace of $\cH_\cY(k_d)$ which is ${\bf R}_\lam$-invariant,
then the restricted output pair $(E, {\bf T}) \in \cL(\cX, \cY) \times \cL(\cX)^d$ given by
$$
 \cX = \cM^\perp, \quad E = {\bf ev}_0|_\cX, \quad {\bf T} = {\bf R}_\lam^*|_\cX
$$
is exactly observable and $\cM^\perp  = \Ran \cO_{E,{\bf T}}$. Therefore,
$\cM = \Ker \cO_{E,{\bf T}}^*$ coincides with the solution set of the
homogeneous Operator Argument interpolation Problem ${\bf OAP}_{\cH_\cY(k_d)}$ with interpolation conditions
$$
(E^*f)^{\wedge L}(\bT^{*}):= \cO_{E, \bT}^{*}f=0.
$$
Then it follows by Theorem \ref{T:H2AIPsolcom} that {\em there exist a Hilbert space $\cF$ and a
McCT-inner multiplier $\Phi\in {\mathcal S}_{d}(\cF, \cY)$
such that  $\cM=\Phi\cdot \cH_\cF(k_d)$.}

\section{An operator theoretical view}
\label{S:OTprelim}

In this section we present a general, purely operator theoretical perspective on the problems considered above, relying on what we will call a generalized de Branges-Rovnyak space. Our approach in this section is to view the problem in the context of a Douglas Factorization Problem with respect to the lifted norms from the generalized de Branges-Rovnyak spaces. For that purpose we also derive some results on the classical Douglas factorization problem.

\subsection{Generalized de Branges-Rovnyak spaces}

With a given contraction operator $T\in\cL(\cU,\cY)$ we may associate the operator range of the positive semidefinite operator $I_{\cY}-TT^*\succeq 0$:
$$
\cH(T):=\Ran(I-TT^*)^{\frac{1}{2}}\subset \cY
$$
with the lifted norm
\begin{equation}
\|(I-TT^*)^{\frac{1}{2}}y\|_{\cH(T)}=\|(I-\pi)y \|_{\cY}
\label{lifted}
\end{equation}
where $\pi$ is the orthogonal projection onto $\Ker(I-TT^*)^{\frac{1}{2}}$.
It follows from \eqref{lifted} that $\|y\|_{\cH(T)}\ge
\|y\|_{\cY}$ for every $y\in\cH(T)$ and thus $\cH(T)$ is contractively included in $\cY$.
Upon letting $y= (I-TT^*)^{\frac{1}{2}}y'$ in the last formula we get
\begin{equation}
\|(I-TT^*)y'\|_{\cH(T)}=\langle (I-TT^*)y', \, y'\rangle_{\cY}.
\label{lifted'}
\end{equation}
The original characterization of $\cH(T)$ as the space of all vectors in $y\in \cY$ such that
$$
\kappa(y):=\sup_{u\in\cU}\big\{\|y+Tu\|^2_{\cY}-\|u\|^2_{\cU}\big\}
$$
is finite and the identity $\|y\|^2_{\cH(T)}=\kappa(y)$ is due to de Branges and Rovnyak \cite{dbr2}.
Let us note that de Branges and Rovnyak worked with the special case where $\cU$ and $\cY$ are replaced by
Hardy spaces $H^2_\cU$ and $H^2_\cY$ and the operator $T \colon H^2_\cU \to H^2_\cY$ is the operator
$T = M_S \colon f(\lam) \mapsto S(\lam) f(\lam)$ of multiplication by a Schur-class function $S \in \cS(\cU, \cY)$.
The spaces $\cH(T)$ (operator- rather than function-theoretic) were proposed by Sarason \cite{sarasonsubh} and will here be referred to as {\em generalized de Branges-Rovnyak spaces.}

\subsection{The solutions to Douglas Factorization Problem}
\label{SubS:DFP}

Consider $A\in\cL(\cY,\cX)$ and $B\in\cL(\cU,\cX)$. The following well-known result due to Douglas \cite{Douglas} describes when there exists an operator $Y\in\cL(\cU,\cY)$ that satisfies
\begin{equation}\label{2.1}
AY = B \quad\text{\em and}\quad \|Y\| \le 1.
\end{equation}
\begin{lem}\label{L:Douglas}
There exists a $Y\in\cL(\cU,\cY)$ satisfying \eqref{2.1} if and only
if $AA^*\succeq BB^*$. In this case, there exists a unique
$Y$ satisfying \eqref{2.1} and the additional constraints
$\Ran Y\subset \Ran A^*$
and $\Ker Y= \Ker B$.
\end{lem}
The next proposition contains several characterizations of the operators $Y\in\cL(\cU,\cY)$ that satisfy \eqref{2.1}.
\begin{prop}\label{L:2.1}
Given $A\in\cL(\cU,\cX)$ and $B\in\cL(\cY,\cX)$, let us assume that
\begin{equation}\label{defp}
Q:=A A^{*}-BB^{*}\succeq 0.
\end{equation}
Then there exist unique contractions
\[
X_1\in\cL(\cU,\overline{\Ran} A) \quad \mbox{and} \quad
X_2\in\cL(\cY,\overline{\Ran} A)
\]
so that
\begin{equation}\label{2.2u}
(AA^*)^{\frac{1}{2}}X_1=B,\ \
(AA^*)^{\frac{1}{2}}X_2=A, \ \ \Ker X_1=\Ker B,\ \
\Ker X_2=\Ker A,
\end{equation}
with $X_2$ being a coisometry. Given $Y\in\cL(\cU,\cY)$, define
%\begin{equation}\label{fy}
$$
F_{_{Y}}:=A^*-YB^*\in \cL(\cX,\cY).
$$
%\end{equation}
Then the following statements are equivalent:
\begin{enumerate}
\item $Y$ satisfies conditions \eqref{2.1}.
\item The operator
%\begin{equation}\label{2.3p}
$$
\mathbb P=\begin{bmatrix} Q & F_{_{Y}}^* \\ F_{_{Y}} & I_{\cY}-YY^* \end{bmatrix}: \;
\begin{bmatrix}\cX\\\cY\end{bmatrix}\to\begin{bmatrix}\cX\\\cY\end{bmatrix}
$$
%\end{equation}
is positive semidefinite, or, which is the same, $Y$ is a contraction and $F_{_{Y}}$ maps $\cX$ into
$\cH(Y)$ with the property
%\begin{equation}
$$
\|F_{_{Y}}x\|_{\cH(Y)}\le \|Q^{\frac{1}{2}}x\|_\cX \quad\mbox{for
every}\quad x\in\cX.
$$
%\label{2.19n}
%\end{equation}
\item $Y$ is a contraction and $F_{_{Y}}$ maps $\cX$ into $\cH(Y)$ with the property
\begin{equation}
\|F_{_{Y}}x\|_{\cH(Y)}= \|Q^{\frac{1}{2}}x\|_\cX \quad\mbox{for every}\quad x\in\cX.
\label{2.20n}
\end{equation}
\item $Y$ is of the form
\begin{equation}\label{2.4u}
Y=X_2^*X_1+(I-X_2^*X_2)^{\frac{1}{2}}K (I-X_1^*X_1)^{\frac{1}{2}}
\end{equation}
where $X_1$ and $X_2$ are defined as in \eqref{2.2u} and
where the parameter $K$ is an arbitrary contraction from
$\overline{\Ran}(I-X_1^*X_1)$ into $\overline{\Ran}(I-X_2^*X_2)$.
\end{enumerate}
Moreover, there is a unique operator $Y\in\cL(\cU,\cY)$  satisfying \eqref{2.1} if and only if
$X_1$ is isometric on $\cU$ or $X_2$ is isometric on $\cY$. Moreover, for $Y$ as in \eqref{2.4u} we have
\begin{equation}\label{min norm}
\|Yh\|^2=\|X_2^*X_1 h\|^2+\|(I-X_2^*X_2)^{\frac{1}{2}}K
(I-X_1^*X_1)^{\frac{1}{2}}h\|^2, \quad h\in\cU
\end{equation}
so that $X_2^*X_1$ is the minimal norm solution to the problem \eqref{2.1}.
\end{prop}

\begin{proof}
The existence and uniqueness of the contractions $X_1$ and $X_2$ satisfying \eqref{2.2u} is a direct consequence
of Lemma \ref{L:Douglas}. That $X_2$ is a coisometry can be seen as a consequence of the identity
$$
  (AA^*)^{\frac{1}{2}} (I_{\overline{\Ran} A}  - X_2 X_2^*) (A A^*)^{\frac{1}{2}}  = 0
$$
which in turn is a consequence of the second equation in \eqref{2.2u}.

\smallskip

The equivalence of (1), (2) and (4) was established in \cite[Lemma 2.2]{bbt3} via
Schur-complement arguments.  The central observation there was that
$Y \in \cL(\cU, \cY)$ satisfies \eqref{2.1} if and only if it satisfies the condition
\begin{equation}\label{2.3u}
\begin{bmatrix}  I_{\cU} & B^{*} & Y^{*} \\ B & A A^{*} & A \\ Y & A^{*} & I_{\cY} \end{bmatrix}\succeq 0.
\end{equation}
The implication (3) $\Rightarrow$ (2) is trivial. We verify (1)$\Rightarrow$(3) as follows.
If $Y$ satisfies conditions \eqref{2.1}, then for every $x\in\cX$ we have
$$
F_{_{Y}}x=(A^*-YB^*)x=(A^*-YY^*A^*)x=(I-YY^*)A^*x.
$$
We now apply the formula \eqref{lifted'} to $y'=A^*x$ and then make subsequent use of \eqref{2.1} and \eqref{defp} to get
\begin{align}
\|F_{_{Y}}x\|_{\cH(Y)}&=\langle (I-YY^*)A^*x, \, A^*x\rangle_{\cY}\notag\\
&=\langle A(I-YY^*)A^*x, \, x\rangle_{\cX}\notag\\
&=\langle (AA^*-BB^*)x, \, x\rangle_{\cX}=\langle Qx, \, x\rangle_{\cX}=\|Q^\frac{1}{2}x\|_{\cX},\notag
%\label{normfy}
\end{align}
which is equivalent to \eqref{2.20n}.

Finally, to see \eqref{min norm} simply note that $(I-X_2^*X_2)^{\frac{1}{2}}$ is the
orthogonal projection onto $\cY\ominus\Ker A=\cY\ominus\Ker X_1$ since $X_2$ is a coisometry.
\end{proof}
\begin{rem}  \label{R:MatrixCompPb}
{\rm The equivalence of conditions \eqref{2.1} and \eqref{2.3u} means that
the Douglas factorization problem can be reformulated as a {\em matrix completion problem}:
{\em given a partially defined matrix
\[
\mat{ I_{\cH_1} & B^* & ?^* \\ B  & AA^* & A \\ ? & A^* &  I_{\cH_2 }},
\]
find an operator $Y$ such that plugging in $Y$ for ? leads to a positive semidefinite operator matrix.}
This at various times has been an active area of research in its own right (see e.g.~\cite{Johnson90}).}
\end{rem}

\subsection{Douglas Factorization Problem in generalized de Branges-Rovn\-yak spaces}
\label{SubS:DFP-dBR}

We next consider a problem similar to \eqref{2.1} but now $Y$ is contractive with respect to the norms induced by two generalized  de Branges-Rovnyak spaces: {\em given operators $A\in\cL(\cY,\cX)$, $B\in\cL(\cU,\cX)$ and contractions
$T_1\in\cL(\wtil{\cU},\cU)$ and $T_2\in\cL(\wtil{\cY},\cY)$, find an operator $Y: \; \cH(T_1)\to \cH(T_2)$ such that
\begin{equation}\label{dBR-Douglas}
AY=B|_{\cH(T_1)}\quad\mbox{and}\quad \|Y\|\leq 1.
\end{equation}}
If $A$ and $B$ are a priori considered as operators in $\cL(\cH(T_2),\cX)$ and $\cL(\cH(T_1),\cX)$ respectively (or, equivalently, $T_1=0$ and $T_2=0$), then this is the classical Douglas Factorization Problem from the previous subsection. The following lemma shows that the case where $A\in\cL(\cY,\cX)$,
$B\in\cL(\cU,\cX)$ also can be translated to the classical case.
Given a contraction $T\in\cL(\wtil{\cH},\cH)$ and an operator $C\in\cL(\cH,\cX)$, the range of the operator
\begin{equation}\label{Fc}
 F_{{_C}} =(I-TT^*)C^*: \; \cX\to \cH
\end{equation}
is contained in $\cH(T)\subset \cH$, and therefore, $F_{{_C}}$ can also be viewed as an operator in $\cL(\cX,\cH(T))$.
Following Remark \ref{R:3.0}, we write $F_{{_C}}^*$ for the adjoint of $F_{{_C}}$ in $\cL(\cH,\cX)$ and
$F_{{_C}}^{[*]}$ for the adjoint of $F_{{_C}}$ in $\cL(\cH(T),\cX)$.
\begin{lem}\label{L:dBR-renorm}
Given $C\in\cL(\cH,\cX)$ and a contraction $T\in\cL(\wtil{\cH},\cH)$, define $F_{{_C}}$ as in \eqref{Fc}. Then
\begin{equation}
F_{{_C}}^{[*]}g=Cg \; \; \mbox{for all} \; \; g\in\cH(T)\quad\mbox{and}\quad F_{{_C}}^{[*]}F_{{_C}}=C(I-TT^*)C^*.
\label{fc1}
\end{equation}
\end{lem}

\begin{proof}
For any $g\in\cH(T)$ and $x\in\cX$, we have
\begin{align*}
\langle x, \, F_{_{C}}^{[*]}g\rangle_{\cX}&=\langle F_{_{C}}x, \, g\rangle_{\cH(T)}
=\langle (I-TT^*)C^*x, \, g\rangle_{\cH(T)}\\
&=\langle C^*x, \, g\rangle_{\cH}=\langle x, \, Cg\rangle_{\cX},
\end{align*}
and the first equality in \eqref{fc1} follows.
Letting $g:=F_{{_C}}\wtil{x}=(I-TT^*)C^* \wtil{x}$ in this equality we get
$$
F_{{_C}}^{[*]}F_{{_C}}\wtil{x}=C(I-TT^*)C^*\wtil{x}\quad\mbox{for all}\quad \wtil{x}\in\cX
$$
which justifies the second equality in \eqref{fc1}.
\end{proof}
Given operators $A\in\cL(\cY,\cX)$, $B\in\cL(\cU,\cX)$ and contractions
$T_1\in\cL(\wtil{\cU},\cU)$ and $T_2\in\cL(\wtil{\cY},\cY)$, define the operators
\begin{equation}\label{FAFB}
F_{{_A}}:=(I-T_2T_2^*)A^* \quad \mbox{and}\quad F_{{_B}}:=(I-T_2T_2^*)B^*
\end{equation}
which can also be viewed as operators in $\cL(\cX,\cH(T_2))$ and $\cL(\cX,\cH(T_1))$, respectively.
Using Lemma \ref{L:dBR-renorm} it follows that \eqref{dBR-Douglas} can be rewritten as
%\begin{equation}\label{dBR-Douglas2}
$$
F_{{_A}}^{[*]}Y=F_{{_B}}^{[*]}\quad\mbox{and}\quad \|Y\|\leq 1.
$$
%\end{equation}
In this form, Lemmas \ref{L:Douglas} and \ref{L:2.1} are applicable with $A$ and $B$ replaced by
$F_{{_A}}^{[*]}$ and $F_{{_B}}^{[*]}$, respectively, leading to the following result.

\begin{prop}\label{P:dBR-Douglas}
The problem \eqref{dBR-Douglas} has a solution if and only if
$$
Q:=A(I-T_2T_2^*)A^*-B(I-T_1T_1^*)B^*\succeq 0.
$$
If this is the case, then the following are equivalent:
\begin{enumerate}
\item $Y\in \cL(\cH(T_1),\cH(T_2))$ satisfies \eqref{dBR-Douglas}.

\item The operator
%\begin{equation}\label{2.3p-dBR}
$$
\mathbb P=\begin{bmatrix} Q & F^{[*]}_{{_A}}-F^{[*]}_{{_B}}Y^*\\ F_{{_A}}-Y F_{{_B}} & I_{\cH(T_2)}-YY^* \end{bmatrix}: \;
\begin{bmatrix}\cX\\\cH(T_2)\end{bmatrix}\to\begin{bmatrix}\cX\\\cH(T_2)\end{bmatrix}
$$
%\end{equation}
is positive semidefinite or, which is the same, $Y$ is a contraction and $F_{_{Y}}:=F_{{_A}}-Y F_{{_B}}$ maps $\cX$ into $\cH(Y)$ with the property
%\begin{equation}
$$
\|F_{_{Y}}x\|_{\cH(Y)}\le \|Q^{\frac{1}{2}}x\|_\cX \quad\mbox{for
every}\quad x\in\cX.
$$
\item $Y$ is a contraction and $F_{_{Y}}$ maps $\cX$ into $\cH(Y)$ with the property
%\begin{equation}
$$
\|F_{_{Y}}x\|_{\cH(Y)}= \|Q^{\frac{1}{2}}x\|_\cX \quad\mbox{for every}\quad x\in\cX.
$$
%\label{2.20n-dBR}\end{equation}
\item $Y$ is of the form
\begin{equation}\label{2.4u-dBR}
Y=X_2^*X_1+(I-X_2^*X_2)^{\frac{1}{2}}K (I-X_1^*X_1)^{\frac{1}{2}}
\end{equation}
where  $X_1$ the contraction in $\cL(\cH(T_1),(\kr (I-T_2T_2^*)A^*)^\perp)$ and $X_2$ the contraction in $\cL(\cH(T_2),(\kr (I-T_2T_2^*)A^*)^\perp)$ that are uniquely determined by
\begin{align*}
(A(I-T_2T_2^*)A^*)^{\frac{1}{2}}X_1 g_1= Bg_1,& \ \mbox{ for }\ g_1\in\cH(T_1),\\
(A(I-T_2T_2^*)A^*)^{\frac{1}{2}}X_1 g_2= Ag_2,& \ \mbox{ for }\ g_2\in\cH(T_2),
\end{align*}
with $X_2$ being a coisometry, and where the parameter $K$ is an arbitrary contraction from
$\overline{\Ran}(I-X_1^*X_1)$ into $\overline{\Ran}(I-X_2^*X_2)$.
\end{enumerate}
Moreover, there is a unique operator $Y\in\cL(\cH(T_1),\cH(T_2))$  satisfying \eqref{dBR-Douglas} if and only if $X_1$ is isometric on $\cH(T_1)$ or $X_2$ is isometric on $\cH(T_2)$.
Furthermore, for $Y$ as in \eqref{2.4u-dBR} we have
%\begin{equation}\label{min norm-dBR}
$$
\|Yh\|^2=\|X_2^*X_1 h\|^2+\|(I-X_2^*X_2)^{\frac{1}{2}}K
(I-X_1^*X_1)^{\frac{1}{2}}h\|^2, \quad h\in\cH(T_1)
$$
%\end{equation}
so that $X_2^*X_1$ is the minimal norm solution to the problem \eqref{dBR-Douglas} .
\end{prop}

Of particular interest for this paper is the special case of the de Branges-Rovnyak Douglas Factorization Problem \eqref{dBR-Douglas}
with $\cH_1=\wtil{\cH}_1=\BC$ and $T_1$ the zero operator.
Then $\cH(T_1)=\BC$ and we identify $\cL(\cH(T_1),\cH(T_2))$ with $\cH(T_2)$ and $\cL(\cH_1,\cX)$
with $\cX$. Then we arrive at the following problem: {\em for $A\in\cL(\cH_2,\cX)$,
${\bf x}\in \cX$ and a contraction $T\in \cL(\wtil{\cH}_2),\cH_2$, find a vector $g\in\cH_2$ so that
\begin{equation}\label{dBR-Douglas3}
g\in\cH(T_2),\quad Ag={\bf x},\quad \|g\|_{\cH(T_2)}\leq 1.
\end{equation}}
Hence we consider the problem \eqref{dBR-Douglas} with $B={\bf x}$ and $Y=g$. We split the result of Proposition \ref{P:dBR-Douglas} for the case at hand into two results.

\begin{lem}\label{L:dBR-Douglas3}
Let $A\in\cL(\cH_2,\cX)$, ${\bf x}\in \cX$ and $T_2\in\cL(\wtil{\cH}_2,\cH_2)$ with $T_2$ a contraction. Define $F_{_{A}}$ as in \eqref{FAFB}.
Then there exists a vector $g\in \cH_2$ satisfying
\eqref{dBR-Douglas3} if and only if $P:=F_{_{A}}^{[*]}F_{_{A}}=A(I-T_2T_2^*)A^* \succeq {\bf x}{\bf x}^*$.
Furthermore, a vector $g\in \cH_2$ satisfies \eqref{dBR-Douglas3} if and only if
%\begin{equation}
$$
\mathbb P=\begin{bmatrix}
P-{\bf x}{\bf x}^* & F_{_{A}}^{[*]}-{\bf x} g^{[*]}  \\ F_{_{A}}-g{\bf x}^*
& I_{\cH(T)}-gg^{[*]}\end{bmatrix}\succeq 0
$$
%\label{bp}\end{equation}
or equivalently, if and only if
%\begin{equation}
$$
\begin{bmatrix} 1 & {\bf x}^* & g^{[*]} \\ {\bf x} & P & F_{_{A}}^{[*]} \\ g & F_{_{A}} & I_{\cH(T)}\end{bmatrix}\succeq 0.
$$
%\label{bp'}\end{equation}
\end{lem}

In the second result we provide the parametrization of the solutions. For this purpose, note that in case $F_{_{A}}^{[*]}F_{_{A}} \succeq {\bf x}{\bf x}^*$,
then by Lemma \ref{L:Douglas}
there exist unique $\wtil{\bf x}\in (\kr F_{_{A}})^\perp$ and $\wtilF_{_{A}}\in \cL((\kr F_{_{A}})^\perp,\cH(T_2))$ so that
\begin{equation}\label{tilx-tilFA}
{\bf x}=(F_{_{A}}^{[*]}F_{_{A}})^{\half} \wtil{\bf x}\quad \mbox{and}\quad
F_{_{A}}=\wtil{F}_{_{A}} (F_{_{A}}^{[*]}F_{_{A}})^{\half},
\end{equation}
with $\wtil{F}_{_{A}}$ being an isometry. The space
\begin{equation}
\cN:=\overline{\Ran}\big(I_{\cH(T_2)}-\widetilde{F}_{_{A}}\widetilde{F}_{_{A}}^{[*]}\big)^{\frac{1}{2}}=
\cH(\widetilde{F}_{_{A}})
\label{dcn}
\end{equation}
is isometrically included in $\cH(T_2)$, and its orthogonal complement ${\mathcal N}^\perp$ in $\cH(T_2)$ can be characterized as
\begin{equation}
\cN^\perp=\{F_{_{A}}x: \, x\in\cX\} \quad\mbox{with norm}\quad \|F_{_{A}}x\|_{\cH(T)}=\|(F_{_{A}}^{(*)}F_{_{A}})^{\half}x\|_{_\cX}.
\label{nperp}
\end{equation}

\begin{thm}\label{T:dBR-Douglas3}
Let $A\in\cL(\cH_2,\cX)$, ${\bf x}\in \cX$ and $T_2\in\cL(\wtil{\cH}_2,\cH_2)$ with $T_2$ a contraction.
Define $F_{_{A}}$ as in \eqref{FAFB} and assume that $F_{_{A}}^{(*)}F_{_{A}} \succeq {\bf x}{\bf x}^*$.
Define $\wtil{\bf x}$ and $\wtil{F}_{_{A}}$ by \eqref{tilx-tilFA} and $\cN$ as in \eqref{dcn}.
Then all solutions $g$ to the problem \eqref{dBR-Douglas3} are given by the formula
\begin{equation}
g=F_{_A}\widetilde{\bf x}+h
\label{descg}
\end{equation}
where $h$ is a free parameter from the subspace $\cN=\cH(\widetilde{F}_{_{A}})\subset \cH(T_2)$ subject to
\begin{equation}
\|h\|_{\cH(T_2)}\le \sqrt{1-\|\widetilde{\bf x}\|^2_{\cX}}.
\label{normin}
\end{equation}
Furthermore, the problem has a unique solution if and only if $\|\widetilde {\bf x}\|=1$ or
$\widetilde{F}_{_{A}}\widetilde{F}_{_{A}}^{[*]}=I_{\cH(T)}$. Furthermore, we have
$$
\|g\|_{\cH(T_2)}^2=\|F_{_A}\widetilde{\bf x}\|_{\cH(T_2)}^2+\|h\|_{\cH(T_2)}^2=\|\widetilde{\bf x}\|_{\cX}^2+\|h\|_{\cH(T_2)}^2.
$$
\end{thm}

\begin{proof}
In the present framework, the operators $X_1$ and $X_2$ in Proposition \ref{P:dBR-Douglas} amount to $\widetilde{\bf x}$ and
$\widetilde{F}_{_{T}}^{[*]}$ respectively, and therefore, the parametrization formula \eqref{2.4u} takes
the form
\begin{equation}
g=F_{_A}\widetilde{\bf x}+\big(I_{\cH(T_2)}-\widetilde{F}_{_{A}}\widetilde{F}_{_{A}}^{[*]}\big)^{\frac{1}{2}}K
\sqrt{1-\widetilde{\bf x}^*\widetilde{\bf x}}
\label{gd}
\end{equation}
where now $K$ is an arbitrary vector in $\cN$ with $\|K\|_{\cN}=\|K\|_{\cH(T_2)}\le 1$. Therefore, the second term
on the right side of \eqref{gd} represents a generic vector $h\in\cN$ subject the inequality \eqref{normin}.
The uniqueness statement follows immediately from \eqref{gd}.
\end{proof}

\begin{rem}\label{R:2.6} {\em
We remark that the second term $h$ on the right hand side of \eqref{descg} represents in fact the general solution of the homogeneous interpolation problem
(with interpolation condition $F_{_A}^{[*]}g=0$). If $h$ runs through the whole space $\cH(T_2)$, then formula \eqref{descg}
produces all $g\in\cH(T_2)$ such that $F_{_A}^{[*]}g={\bf x}$.
This unconstrained interpolation problem has a solution if and only if ${\bf x}\in\Ran F_{_{A}}^{(*)}F_{_{A}}^{\frac{1}{2}}$ and has a unique solution
if and only if $\widetilde{F}_{_{A}}\widetilde{F}_{_{A}}^{[*]}=I_{\cH(T_2)}$.
If $\widetilde{F}_{_{T}}\widetilde{F}_{_{T}}^{[*]}\neq I_{\cH(T)}$,
then the unconstrained problem has infinitely many solutions, and
$F_{_A}\widetilde{\bf x}$ has the minimal possible $\cH(T_2)$-norm, which is $\|\widetilde{\bf x}\|$, by \eqref{nperp}. Thus, if $\|\widetilde{\bf x}\|=1$,
then uniqueness occurs since the {\em minimal norm solution} already has unit norm.
}\end{rem}

\appendix

\section{A Kre\u{\i}n space lemma}\label{S:Krein}

In this section we prove a general Kre\u{\i}n space lemma. See the end of the introduction for the basic definitions of Kre\u{\i}n spaces and some classes of Kre\u{\i}n space operators. Before we state the result, some further preliminaries are required.

Given two Kre\u{\i}n spaces $(\cX', J')$ and $(\cX, J)$ and given an operator $T \in\cL(\cX',\cX)$, we define  the $(J', J)$-adjoint of $T$ to be the operator $T^{[*]}\in\cL(\cX,\cX')$  so that
$$
  [ T x, y ]_J= [ x, T^{[*]} y]_{J'}  \text{ for all } x \in \cX', \, y \in \cX.
 $$
The elementary computation
\begin{align*}
[x, T^{[*]}y]_{J'} & = [ Tx, y]_J =  \langle J T x, y \rangle_\cX = \langle x, T^* J y \rangle_{\cX'} =
\langle J' x, J^{\prime -1} T^* J y \rangle_{\cX'}  \\
& = [ x, (J^{\prime -1} T^* J) y ]_{J'}
\end{align*}
shows that
%\begin{equation}\label{KreinAdj}
$$
T^{[*]} = J^{\prime -1} T^* J.
$$
%\end{equation}
In case $(\cX', J') = (\cX, J)$ we use the term $J$-adjoint rather than $(J, J)$-adjoint for $T^{[*]}$.
If it happens that furthermore $T = T^{[*]}$, we say that $T$ is {\em $J$-self-adjoint}.
Note that elsewhere in this paper the notation $T^{[*]}$ is also used for the adjoint with respect to a de Branges-Rovnyak space
(cf., Remark \ref{R:3.0}).  We use the same notation here for a Kre\u{\i}n space adjoint since this notation is standard and is not used
outside the current appendix.

 Using the Kre\u{\i}n space adjoint notation we note that $T$ is a $(J,J')$-isometry if $T^{[*]}T=I_{\cX}$. We then say that $T$ is a
 $(J,J')$-coisometry if $TT^{[*]}=I_{\cX'}$. The latter corresponds to the identity $T J'^{-1}T^*=J^{-1}$, which corresponds to the
 second entry in \eqref{bicon} with the inequality replaced by an equality in case $J$ and $J^{'}$ are signature operators. As before,
 $T$ is $(J,J')$-unitary if $T$ is both $(J,J')$-isometric and $(J,J')$-coisometric, i.e., when $T$ is invertible with $T^{-1}=T^{[*]}$.

A well-known property for the Hilbert-space case is that a surjective isometry is unitary.  This property extends to the Kre\u{\i}n
space setting, as we record in the following remark.

\begin{rem} \label{R:Junitary} An onto $(J', J)$-isometry is in fact $(J', J)$-unitary.
{\em Indeed, suppose that $V \colon \cX' \to \cX$ is an onto
$(J', J)$-isometry.   To show that $V$ is $(J',J)$-coisometric we must show that
$ V J^{\prime -1} V^* = J^{-1}$.  Given $x \in \cX$,  we can find $y \in \cX'$ so that $x = J V y$.
We then compute
$$
V J^{\prime -1} V^* x  = V J^{\prime -1} V^* (J V y) =  V J^{\prime -1} (V^* J V) y
= V J^{\prime -1} J' y =V y  = J^{-1} x
$$
implying that $V$ is also ($J, J')$-coisometric as claimed.}
\end{rem}

Finally let us say that two invertible self-adjoint operators $J$ on $\cX$ and $J'$ on $\cX'$ are {\em congruent} if there exists a
linear bijection $R\in\cL(\cX',\cX)$ so that $ R^* J R = J'$.  By taking inverses in this equation we see that if $R$ implements
a congruence from $J'$ to $J$, then $R^{* -1}$ implements a congruence from $J^{\prime -1}$ to $J^{-1}$.

The following lemma shows how a Kre\u{\i}n-space isometry $V$ has a row extension $\begin{bmatrix} V & W \end{bmatrix}$
to a Kre\u{\i}n space unitary operator. It is a variation
on \cite[Theorem 2.3 (3)]{bbieot} in a more general setting. Let us also mention that there is an extensive discussion in Section 2.3
of \cite{drrov} of
row extensions of a Kre\u{\i}n-space bicontraction operator $T$ to a Kre\u{\i}n space bicontraction operator  on a larger space $\begin{bmatrix} T & S \end{bmatrix}$.  For simplicity, we give a self-contained direct proof of the special case of interest here.

\begin{lem} \label{L:findBD}
Suppose that $(\cX'_1, J_1')$ and $(\cX, J)$ are Kre\u{\i}n spaces. Let $V\in\cL(\cX'_1,\cX)$ be a $(J'_1, J)$-isometry. Then there exist
a Kre\u{\i}n spaces $(\cX'_2, J'_2)$ and an operator $W\in\cL(\cX'_2,\cX)$ so that $J$ is congruent with $J':=\sbm{J'_1&0\\0& J_2'}$ and
$\begin{bmatrix}V&W \end{bmatrix}$ is $\big(J,J' \big)$-unitary. Moreover, given a Kre\u{\i}n space $(\cX'_2, J'_2)$,
an operator  $W\in\cL(\cX'_2,\cX)$ completes $V$ to a  $(J,J')$-unitary operator if and only if $W$ is an injective solution to the following
indefinite-metric Cholesky factorization problem:
\begin{equation}  \label{JCholeskyfact}
W J_2^{\prime -1} W^* = J^{-1} - V J_1^{\prime -1} V^*.
\end{equation}
\end{lem}

\begin{proof}
Since $V$ is a $(J'_1, J)$-isometry, we have $V^* J V =J_1'$. For $x,y\in\cX_1'$ we then have that
\[
[Vx,Vy]_J =\inn{JVx}{Vy}_\cX=\inn{V^*JVx}{y}_{\cX_1'}=\inn{J_1'x}{y}_{\cX_1'}=[x,y]_{J_1'}.
\]
Hence the restriction of the indefinite inner product of $(\cX,J)$ to $\cM:=\Ran V$ defines a Kre\u{\i}n space on $\cM$, that is, $\cM$ is
a regular (or Kre\u{\i}n) subspace of $\cX$.  Another formulation of regularity of the subspace $\cM$ as a subspace of $\cK$
(see e.g.~\cite[Theorem 1.1.1]{drrov}) is that $\cM$ be the range of a $J$-orthogonal projection operator, i.e., a bounded
$J$-self-adjoint idempotent operator on $\cX$.  Explicitly for
the case here with $\cM = \operatorname{Ran} V$, one can check that $\cM = \operatorname{Ran} P$ where the $J$-orthogonal projection operator $P$ is given by
$$
 P = V J_1^{\prime -1} V^* J,
$$
i.e., one can check that $P = J^{-1} P^* J = P^{[*]}$, $P^2 = P$ and
$PV = V$ (implying when combined with the definition of $P$ that $\operatorname{Ran} P = \operatorname{Ran} V$).

Another characterization of regularity of $\cM$ (see e.g.~Theorem 1.3 in \cite{DR96}) is that
$\cX$ is obtained as the direct sum of $\cM$ and the Kre\u{\i}n space orthogonal complement
$\cM^{\perp_J}:=\{x\in\cX\colon [x,y]_J=0 \text{ for all } y\in\cM\}$ of $\cM$ in $\cX$.
Explicitly the associated $J$-orthogonal projection operator $Q$ such that $\operatorname{Ran} Q = \cM^{\perp_J}$
is given  as the $J$-orthogonal projection operator $Q$ $J$-complementary to $P$, namely
\begin{equation}  \label{Q1}
 Q = I_\cX - P = I_\cX -  V J_1^{\prime -1} V^* J.
 \end{equation}
Indeed one can check directly that
$$
 Q = Q^{[*]}, \quad Q^2 = Q, \quad QV = 0, \quad P + Q = I_\cX.
 $$

We can get another formula for the $J$-orthogonal projection operator with range equal to $\cM^{\perp_J}$ as follows.
Set $\wtil \cX'_2 = \cM^{\perp_J}$ and let $\wtil W \colon \wtil \cX_2 \to \cX$ be the inclusion map of $\wtil \cX'_2$ into $\cX$.
Hence by construction $\wtil W$ is injective.
Define $\wtil J'_2$ by
$$
  \wtil J'_2 = \wtil W^* J \wtil W.
$$
Since $\operatorname{Ran} \wtil W = \cM^{\perp_J}$ is a regular subspace of $\cX$, it follows that $\wtil J'_2$ is an invertible
self-adjoint operator on $\wtil \cX'_2$.  By an argument parallel to that done above for the operator $V$ having
$\operatorname{Ran} V = \cM$, we see that $Q$ (the $J$-orthogonal projection operator with range equal to $\cM^{\perp_J}$)
is given by
\begin{equation}   \label{Q2}
Q = \wtil W \wtil J_2^{\prime -1} \wtil W^* J.
\end{equation}
Combining \eqref{Q1} with \eqref{Q2} gives us
$ \wtil W \wtil J_2^{\prime -1} \wtil W^* J = I_\cX -  V J_1^{\prime -1} V^* J$, or in more Hermitian form,
%\begin{equation}  \label{JCholesky1}
$$
   \wtil W \wtil J_2^{\prime -1} \wtil W^*  = J^{-1}-  V J_1^{\prime -1} V^*,
$$
% \end{equation}
 i.e., $\wtil W$ arises as an injective solution of an indefinite-metric Cholesky factorization problem of the type \eqref{JCholeskyfact}.

 Since $\operatorname{Ran} V = \cM$ is $J$-orthogonal to $\operatorname{Ran} \wtil W = \cM^{\perp_J}$, we certainly have
 $$
    0 = \wtil W^*JV, \quad 0 = V^*J \wtil W.
 $$
 Let us next compute
 $$
 \begin{bmatrix} V^* \\ \wtil W^* \end{bmatrix} J \begin{bmatrix} V & \wtil W \end{bmatrix}
 = \begin{bmatrix} V^* J V  & V^* J \wtil W \\ \wtil W^* J V & \wtil W^* J \wtil W \end{bmatrix} = \begin{bmatrix} J_1' & 0 \\ 0 & \wtil J_2' \end{bmatrix},
 $$
 i.e., $\begin{bmatrix} V & \wtil W \end{bmatrix}$ is a $\big( \sbm{ J_1' & 0 \\ 0 & \wtil J_2' }, J \big)$-isometry.  But we also know
 that $\operatorname{Ran} \begin{bmatrix} V & \wtil W \end{bmatrix} = \cM + \cM^{\perp_J} = \cX$, so by Remark \ref{R:Junitary}
 it follows that in fact $\begin{bmatrix} V & \wtil W \end{bmatrix}$ is a $\big( \sbm{ J_1' & 0 \\ 0 & \wtil J_2'}, J \big)$-unitary
 operator, and we conclude that $\wtil W$ implements the desired completion of $V$ to a Kre\u{\i}n space unitary operator.

 Now suppose that $W$ is any injective solution of the $J$-Cholesky factorization problem \eqref{JCholeskyfact}.
 As observed in the first part of the proof, the operator to be factored on the right-hand side of \eqref{JCholeskyfact} multiplied
 by $J$ on the right,
 namely,  $(J^{-1} - V J_1^{\prime -1} V^*) J = I_\cX - V J_1^{\prime -1} V^* J$, is equal to the $J$-orthogonal projection operator
 with range equal to $(\operatorname{Ran} V)^{\perp_J}$, and hence has a fairly substantial kernel, namely
 $$
  \operatorname{Ker} \, (I_\cX - V J_1^{\prime -1} V^* J) = \operatorname{Ran} V.
 $$
   Hence the Hermitian operator $J^{-1} - V J_1^{\prime -1} V^*$
 has closed range and its restriction to $(\operatorname{Ker} (J^{-1} - V J_1^{\prime -1} V^*))^\perp$ is bounded below.  Using spectral
 theory we can then find a factorization
\begin{equation}  \label{Cholesky2}
J^{-1} - V J_1^{\prime -1} V^* =  W J_2^{\prime -1} W^*
\end{equation}
with $J_2'$ a bounded, invertible signature operator on $\cX'_2 = (\operatorname{Ker} (J^{-1} - V J_1^{\prime -1} V^*))^\perp$,
i.e., the pair $(W, J'_2)$ is a solution of the $J$-Cholesky factorization problem \eqref{JCholeskyfact} with $W$ injective.
Note that any other choice of $J_2$ in such a factorization is then uniquely determined up to a congruence with any particular
such $J_2^\prime$, with a corresponding adjustment of $\wtil W$ to get $W$:
$$
J'_2 = X \wtil J'_2 X^*, \quad W = \wtil W X^{-1} \text{ for some invertible } X \colon \wtil \cX'_2 \to \cX'_2.
$$
 The equality \eqref{Cholesky2} then can be interpreted as saying that
$\operatorname{Ran} W = \cM^{\perp_J}$ (where $\cM = \operatorname{Ran} V \subset \cX$).  We may then proceed as in the first
part of the proof to see that $\begin{bmatrix} V & W \end{bmatrix}$ is $\big( \sbm{ J_1' & 0 \\ 0 & J_2'},  J \big)$-unitary as expected.
\end{proof}

The following corollary adds some additional information which will be needed to complete the proof of Theorem \ref{T:4.1} in Section \ref{S:LFTnc}. We shall make use of the following notation:\
for an invertible self-adjoint operator $J$ on some separable Hilbert space $\cX$,
$$
\operatorname{Inertia}(J) = (\kappa_+(J), \kappa_-(J))
$$
means that $\kappa_+(J)$ and $\kappa_-(J)$ denote the respective dimensions
(with $\infty$ allowed) of the positive (respectively negative) spectral subspaces of $J$.  Note that two Kre\u{\i}n spaces
$(\cX, J)$ and $(\cX', J')$ are Kre\u{\i}n-space isomorphic if and only if $\operatorname{Inertia}(J) =
\operatorname{Inertia}(J')$.

\begin{cor}  \label{C:J-UniCompl}  Suppose that $(\cX_1, J_1)$, $(\cX_2, J_2)$, $(\cX'_1, J'_1)$, $(\cX'_2, J'_2)$ are
Kre\u{\i}n spaces such that
\begin{itemize}
\item[(i)] $\big(\cX_1 \oplus \cX_2, \sbm{ J_1 & 0 \\ 0 & J_2}\big)$ is Kre\u{\i}n-space isomorphic to
$\big(\cX'_1 \oplus \cX'_2, \sbm{ J'_1 & 0 \\ 0 & J'_2} \big)$, and
\item[(ii)] $(\cX_1, J_1)$ is Kre\u{\i}n-space isomorphic to $(\cX'_1, J'_1)$.
\end{itemize}

Then:
\begin{enumerate}
\item
If $\dim \cX_j < \infty$  and $\dim \cX'_j < \infty$ for $j=1,2$, then $(\cX_2, J_2)$ is Kre\u{\i}n-space isomorphic to $(\cX'_2, J'_2)$, i.e.,
\[
\kappa_+(J_2) = \kappa_+(J'_2)\quad \text{and}\quad \kappa_-(J_2) = \kappa_-(J'_2).
\]

\item
If $\kappa_-(J_1) = \kappa_-(J_1) = 0$, i.e., $(X_1, J_1)$ and $(X'_1, J'_1)$ are Hilbert spaces of the same dimension,  then we can  at least still conclude that
\begin{align*}
\kappa_+(J_1) + \kappa_+(J_2) & = \kappa_+(J'_1) + \kappa_+(J'_2), \\
\kappa_-(J_2) & = \kappa_-(J'_2).
\end{align*}
\end{enumerate}
\end{cor}

\begin{proof}  In case (1), all the indices $\kappa_\pm(J_j)$and $\kappa_{\pm}(J'_j)$ ($j=1,2$) are finite.  Note that hypothesis
(i) is the statement that
\begin{align}
& \kappa_+(J_1) + \kappa_+(J_2) = \kappa_+(J'_1) + \kappa_+(J'_2), \notag \\
& \kappa_-(J_1) + \kappa_-(J_2) = \kappa_-(J'_1) + \kappa_-(J'_2),  \label{inertia=1}
\end{align}
while the hypothesis in (ii) is that
\begin{equation}  \label{inertia=2}
 \kappa_+(J_1) = \kappa_+(J'_1), \quad \kappa_-(J_1) = \kappa_-(J'_1).
\end{equation}
Subtraction of the first of \eqref{inertia=2} from the first of \eqref{inertia=1} and of the second of \eqref{inertia=2} from the second
of \eqref{inertia=1} and using that all these integers are finite leads us to the two identities
$$
 \kappa_+(J_2) = \kappa_+(J'_2), \quad \kappa_-(J_2) = \kappa_-(J'_2)
$$
which in turn is that statement that $(\cX_2, J_2)$ is Kre\u{\i}n-space isomorphic to $(\cX'_2, J'_2)$.

Take now the extra hypothesis to be as in case (2).  In this case the quantity $\kappa_+(J_1) + \kappa_+(J_2)$
can be identified as the dimension of a maximal positive subspace in $\big(\cX_1 \oplus \cX_2, \sbm{ J_1 & 0 \\ 0 & J_2} \big)$
while the quantity $\kappa_+(J'_1) + \kappa_+(J'_2)$ can be identified as the dimension of a maximal positive subspace in
$\big( \cX'_1 \oplus \cX'_2, \sbm{ J'_1 & 0 \\ 0 & J'_2} \big)$.  The assumption that $\big(\cX_1 \oplus \cX_2,
\sbm{ J_1 & 0 \\ 0 & J_2} \big)$ is Kre\u{\i}n-space isomorphic to $\big( \cX'_1 \oplus \cX'_2, \sbm{ J'_1 & 0 \\ 0 & J'_2} \big)$ then
implies that these two quantities must be the same (finite or infinite).   Similarly, one can identify $\kappa_-(J_2)$
as the dimension of a maximal negative subspace in $\big(\cX_1 \oplus \cX_2, \sbm{ J_1 & 0 \\ 0 & J_2} \big)$
(since $\cX_1$ is a Hilbert space by assumption) and similarly $\kappa_-(J'_2)$ is the dimension of a maximal
negative subspace in $\big( \cX'_1 \oplus \cX'_2, \sbm{ J'_1 & 0 \\ 0 & J'_2} \big)$.  Again, the assumed Kre\u{\i}n-space
isomorphism between $\big(\cX_1 \oplus \cX_2, \sbm{ J_1 & 0 \\ 0 & J_2} \big)$ and
$\big( \cX'_1 \oplus \cX'_2, \sbm{ J'_1 & 0 \\ 0 & J'_2} \big)$ implies that we must have the equality
$\kappa_-(J_2) = \kappa_-(J'_2)$ (finite or infinite).
\end{proof}

\section{Schur multipliers and their adjoints}\label{S:Schur}
As is well known, any Schur-class function $S\in\mathcal S(\cU,\cY)$ gives rise to a Schur-class function $S^\sharp(\lambda):=S(\overline{\lambda})^*\in
\mathcal S(\cY,\cU)$. As a consequence of \eqref{easy2} and \eqref{easy4}, it follows that a similar result holds in multivariable settings.
\begin{thm}
If $S(z)={\displaystyle \sum_{\alpha\in\free}S_\alpha z^\alpha}$ belongs to $\cS_{{\rm nc},d}(\cU,\cY)$,
then $S^\sharp(z):={\displaystyle \sum_{\alpha\in\free}S^*_\alpha z^{\alpha^\top}}$ belongs to $\cS_{{\rm nc},d}(\cY,\cU)$.
\label{T:b1}
\end{thm}
\begin{proof}
Since $S\in\mathcal S(\cU,\cY)$, i.e., the formal kernel \eqref{KS} is positive, it follows (see e.g., \cite[Theorem 3.1]{bbf3}) that
there exists a Hilbert space $\cX$ and a unitary connection
       operator $\bU$ of the form
       \begin{equation} \label{NCcolligation}
       \bU = \begin{bmatrix} A & B \\ C & D \end{bmatrix} =
       \begin{bmatrix} A_{1} & B_{1} \\ \vdots & \vdots \\ A_{d} &
       B_{d} \\ C & D \end{bmatrix} \colon \begin{bmatrix} \cX \\ \cU
       \end{bmatrix} \to \begin{bmatrix}  \cX \\ \vdots \\ \cX \\ \cY
       \end{bmatrix}
       \end{equation}
       so that $S(z)$ can be realized as a formal power series in the
       form
       \begin{align}
         S(z) &= D + \sum_{j=1}^{d} \sum_{\alpha \in \free} C {\bf A}^{\alpha}B_{j}
         z^{\alpha}\cdot z_{j}=
         D + C (I - Z(z) A)^{-1} Z(z) B,\notag
        \end{align}
Then
\begin{align}
S^\sharp(z)&=D^* + \sum_{j=1}^{d} \sum_{\alpha \in \free}B_{j}^*{\bf A}^{*\alpha^\top}C^*z_jz^{\alpha^\top}\notag\\
&=D^* + \sum_{j=1}^{d} \sum_{\alpha \in \free}B_{j}^*z_j{\bf A}^{*\alpha}C^*z^{\alpha} \notag\\
&=D^* + B^*Z(z)^\top(I - A^*Z(z)^\top)^{-1}C^*=D^* + B^*(I -Z(z)^\top A^*)^{-1}Z(z)^\top C^*,\notag
 \end{align}
and since ${\bf U}$ is unitary, we have
\begin{align}
&K_{S^\sharp}(z, \zeta): = k_{\rm Sz}(z,\zeta)I_{\cU} - S^\sharp(z)( k_{\rm Sz}(z,\zeta)I_{\cU})S^\sharp(\zeta)^{*}\\
&=B^*(I - Z(z)^\top A^*)^{-1}\big( k_{\rm Sz}(z,\zeta)I_{\cX^d} -Z(z)^\top  k_{\rm Sz}(z,\zeta) Z(\overline{\zeta})\big)(I-AZ(\overline{\zeta}))^{-1}B.
\notag
\end{align}
As we have seen in the proof of Theorem \ref{T:4.1}, the kernel
$$
 k_{\rm Sz}(z,\zeta)I_{\cX^d} -Z(z)^\top  k_{\rm Sz}(z,\zeta) Z(\overline{\zeta})
$$
is positive; therefore the kernel $K_{S^\sharp}(z, \zeta)$ is also positive and hence $S^\sharp\in\cS_{{\rm nc},d}(\cY,\cU)$.
\end{proof}
In \cite[p. 110]{bb2}, it was erroneously suggested that for an $S\in\mathcal S_d(\cU,\cY)$, the function $S^\sharp$ does not have to belong to
$\mathcal S_d(\cY,\cU)$. It was stated that for the Schur multiplier $S(\lambda)=\begin{bmatrix}\lambda_1 & \lambda_2\end{bmatrix}$, 
the associated function $S^\sharp(\lambda)=\sbm{\lambda_1 \\ \lambda_2}$ is not a Schur multiplier, since the kernel
\begin{equation}
K^{S^\sharp}(\lambda,w)=\frac{\left[\begin{array}{cc}1 & 0 \\ 0 &
1\end{array}\right]-\left[\begin{array}{c}\lambda_1\\
\lambda_2\end{array}\right]\begin{bmatrix}\overline{w}_1 &
\overline{w}_2\end{bmatrix} }{1-\lambda_1\overline{w}_1-\lambda_2\overline{w}_2} 
\label{ker2002}
\end{equation}
is not positive on $\B^d$, as the matrix
$$
\left[\begin{array}{cc}K^{S^\sharp}(w^{(1)}, w^{(1)}) &
K^{S^\sharp}(w^{(1)},w^{(2)})\\
K^{S^\sharp}(w^{(2)}, w^{(1)})&
K^{S^\sharp}(w^{(2)},w^{(2)})\end{array}\right],\quad w^{(1)}=\left(\frac{1}{2}, 0\right), \; \; w^{(1)}=\left(0, \frac{1}{2}\right)
$$
is not positive semidefinite. More careful later inspection revealed that the latter matrix equals
$$
\left[\begin{array}{rccr}1 &0&1&-\frac{1}{4} \\0& \frac{4}{3}&0&1
\\1&0&\frac{4}{3}&0\\-\frac{1}{4} &1&0&1\end{array}\right]
$$
and is positive definite. Moreover, the kernel \eqref{ker2002} is actually positive on $\mathbb B^2$ by the following 
abelianized version of Theorem \ref{T:b1}.
\begin{thm}
If $S(\lambda)$ belongs to $\cS_{d}(\cU,\cY)$,
then $S^\sharp(\lambda):=S(\overline\lambda)^*$ belongs to $\cS_{d}(\cY,\cU)$.
\label{T:b2}
\end{thm}

\paragraph{\bf Acknowledgments}
This work is based on research supported in part by the National Research Foundation of South Africa (NRF) and the DST-NRF
Centre of Excellence in Mathematical and Statistical Sciences (CoE-MaSS). Any opinion, finding and conclusion or recommendation
expressed in this material is that of the authors and the NRF and CoE-MaSS do not accept any liability in this regard.

\bibliographystyle{amsplain}
\providecommand{\bysame}{\leavevmode\hbox to3em{\hrulefill}\thinspace}

\end{document}